\documentclass[12pt]{article}
\usepackage[T1]{fontenc}
\usepackage{verbatim}
\usepackage{amsmath,amsthm}
\usepackage{xspace}
\usepackage{pifont}
\usepackage{graphicx}
\usepackage{amssymb}
\usepackage{epic, eepic}
\usepackage{dsfont}
\usepackage{amssymb}
\usepackage{makeidx}
\usepackage{mathrsfs}
\usepackage{xcolor, colortbl}
\usepackage{subcaption}
\usepackage[numbers]{natbib}
\usepackage{enumerate}
\usepackage{enumitem}
\usepackage{afterpage}
\usepackage{mathtools}
\usepackage{epsf}
\usepackage{graphics,color}
\usepackage[export]{adjustbox}
\usepackage{tcolorbox}
\usepackage{authblk}
\usepackage[normalem]{ulem}
\RequirePackage[colorlinks,citecolor=blue,urlcolor=blue]{hyperref}
\mathtoolsset{showonlyrefs}

\def\eps{\varepsilon}
\def\lam{\lambda}
\def\al{\alpha}
\def\wt{\widetilde}
\def\Q{\mathbb{Q}}
\def\P{\mathbb{P}}
\def\E{{\mathbb{E}}}
\newcommand{\R}{{\mathbb R}}
\newcommand{\F}{{\mathbb F}}
\newcommand{\N}{{\mathbb N}}
\newcommand{\bl}{\color{blue}}

\newcommand{\bd}{\begin{displaymath}}
\newcommand{\ed}{\end{displaymath}}
\newcommand{\be}{\begin{equation}}
\newcommand{\ee}{\end{equation}}
\newcommand{\bq}{\begin{eqnarray}}
\newcommand{\eq}{\end{eqnarray}}
\newcommand{\bn}{\begin{eqnarray*}}
\newcommand{\en}{\end{eqnarray*}}
\newcommand{\dl}{\delta}

\newtheorem{theorem}{Theorem}[section]
\newtheorem{lemma}[theorem]{Lemma}
\newtheorem{proposition}[theorem]{Proposition}

\newtheorem{corollary}[theorem]{Corollary}

\newtheorem{remark}[theorem]{Remark}

\newtheorem{example}[theorem]{Example}

\newtheorem{definition}[theorem]{Definition}
\newtheorem{assumption}[theorem]{Assumption}
\numberwithin{equation}{section}

\date{\today}
\title{Stochastic Graphon Games with Memory }
\author[]{Eyal Neuman}
\author[]{Sturmius Tuschmann\thanks{ST is supported by the EPSRC Centre for Doctoral Training in Mathematics of Random \mbox{Systems}: Analysis, Modelling and Simulation (EP/S023925/1).}}
\affil[]{Department of Mathematics, Imperial College London}

\begin{document}
\maketitle

\begin{abstract} 
We study finite-player dynamic stochastic games with heterogeneous interactions and non-Markovian linear-quadratic objective functionals. We derive the Nash equilibrium explicitly by converting the first-order conditions into a coupled system of stochastic Fredholm equations, which we solve in terms of operator resolvents. When the agents' interactions are modeled by a weighted graph, we formulate the corresponding non-Markovian continuum-agent game, where interactions are modeled by a graphon. We also derive the Nash equilibrium of the graphon game explicitly by first reducing the first-order conditions to an infinite-dimensional coupled system of stochastic Fredholm equations, then decoupling it using the spectral decomposition of the graphon operator, and finally solving it in terms of operator resolvents.

Moreover, we show that the Nash equilibria of finite-player games on graphs converge to those of the graphon game as the number of agents increases. This holds both when a given graph sequence converges to the graphon in the cut norm and when the graph sequence is sampled from the graphon. We also bound the convergence rate, which depends on the cut norm in the former case and on the sampling method in the latter.
Finally, we apply our results to various stochastic games with heterogeneous interactions, including systemic risk models with delays and stochastic network games.


\end{abstract} 

\begin{description}
\item[Mathematics Subject Classification (2020):] 	49N10, 91A15, 91A43, 93E20 
\item[Keywords:] graphon games, stochastic games,  network games, Nash equilibrium, non-Markovian, Volterra stochastic control
\end{description}


\newpage
\section{Introduction}\label{sec:introduction}
Network games are games in which the interactions between agents are heterogeneous, often modeled using a graph. They have numerous applications across diverse fields, including economics, social dynamics, transportation, epidemiology and cybersecurity among others (see e.g. \citet{galeotti2010network,jackson2008social}). However, analyzing network games is often challenging or even intractable due to the heterogeneity of the interactions. To address this issue, researchers frequently study their scaling limits as the number of agents tends to infinity (\citet{caines2021graphon,lacker2022case,parise2019graphon}).

The scaling limits of graph sequences are of independent interest, particularly for dense sequences, which have garnered significant attention, while results on sparse sequences are less abundant. \citet{lovasz2006limits} demonstrated that dense graph sequences with converging subgraph densities converge to a natural limit known as a graphon, which is a symmetric measurable function $W:[0,1]^2\to [0,1]$ that captures these densities. \citet{borgs2008convergent,borgs2012convergent} further substantiated the role of graphons as the appropriate limiting objects for dense graph sequences and introduced the cut distance, which characterizes the convergence of these sequences and facilitates sampling graphs from graphons. See \citet{lovasz2012large} for a comprehensive overview of these concepts, which summarizes convergence results for both dense and sparse graph sequences.

Lovász and Szegedy's seminal work \cite{lovasz2006limits} on the convergence of dense graph sequences introduces the concept of graphon games, which serve as infinite-population analogues of finite-player games, where interactions are modeled via a graph. Motivated by the application of graphons in large population analysis, \citet{parise2019graphon} introduced a new class of static infinite-player games called graphon games. In these games, a continuum of agents interacts according to a graphon $W$, where $W(u,v)$ describes the interaction of the two infinitesimal agents $u,v\in [0,1]$. They demonstrated that Nash equilibria of static graphon games can approximate Nash equilibria of network games on large graphs, which are sampled from the graphon. See also \citet{carmona2022stochastic} and \citet{parise2021analysis} for further work on static graphon games.

Dynamic stochastic graphon games were first studied by \citet{gao2020linear}, who derived an approximate Nash equilibrium for the corresponding finite-player games on large graphs in the linear-quadratic case. This case was further investigated by \citet{aurell2022stochastic}, who incorporated idiosyncratic Brownian shocks into each player's state, reconciling these shocks using Fubini extensions. \citet{bayraktar2023propagation} studied, in greater generality, the convergence of the Nash equilibrium of network games to that of the limiting graphon game, using the propagation of chaos. A similar approach was taken by \citet{tangpi2024optimal} in the context of an optimal investment problem under relative performance criteria. \citet{lacker2023label} proposed a novel formulation for stochastic games on large graphs and their graphon limits, focusing on the label-state distribution of a representative agent. This approach facilitates some existing continuum-player models and circumvents the use of the aforementioned Fubini extensions.

Another related class of games, called graphon mean-field games, was introduced by \citet{caines2018graphon,caines2019graphon,caines2021graphon}. In these games, each node in the network is associated with not just one agent, but a cluster of agents, each with homogeneous interactions among them.
We also briefly mention the recent work by \citet{lacker2022case}, who investigated linear-quadratic stochastic differential network games, where the agents' interactions are modeled via a simple graph. Assuming the underlying graph to be vertex-transitive, they derived semiexplicit expressions for the Markovian Nash equilibrium of the game, enabling large-population asymptotics for various graph sequences, including several sparse and dense examples. See also \citet{hu2024finite} for further extensions of this framework.

All the previous works described above focus on Markovian network games and their scaling limits. However, stochastic games with heterogeneous interactions, where both the states and controls are non-Markovian, arise in various applications, and their solutions are typically only available for specific examples. Important applications of network games arise in models of systemic risk (see, e.g., \citet{carmona2015mean}, \citet{carmona2018systemic}, and \citet{sun2022mean}), which often involve heterogeneous interactions between banks with different characteristics. In this class of games, the controls correspond to the rates of lending or borrowing between banks, while the logarithmic monetary reserves of each bank are described by delayed equations in relation to these controls.

The main goal of this work is to provide a general framework for explicitly deriving the Nash equilibrium of linear-quadratic non-Markovian games with heterogeneous interactions. In the finite-player setting, this is achieved by deriving the first-order conditions of the agents' performance functionals as a coupled system of stochastic Fredholm equations, which we solve in terms of operator resolvents. When the agents' interactions are modeled via a weighted graph, we formulate the scaling limits of the game as a non-Markovian continuum-player graphon game. We then introduce a novel approach to solving this game by first deriving the first-order conditions in terms of an infinite-dimensional stochastic Fredholm equation. We decouple the system using the spectral decomposition of the graphon operator and derive its solution in terms of operator resolvents. Additionally, we show that the Nash equilibria of the finite-player games converge to those of the graphon game as the number of agents approaches infinity, both when a given graph sequence converges to the graphon in the cut norm and when the graph sequence is sampled from the graphon. We also specify the convergence rate, which depends on the cut norm in the former case and on the chosen sampling method in the latter. Finally, we apply our results to various stochastic games with heterogeneous interactions.

In addition to its application to the aforementioned systemic risk models, we demonstrate that our general framework accommodates and extends results on stochastic differential network games, as studied by \citet{aurell2022stochastic}, as well as stochastic differential games on simple graphs, as investigated by \citet{lacker2022case} and \citet{hu2024finite}. This illustrates that our framework can also be applied to solve Markovian games. Furthermore, it paves the way for studying stochastic Volterra games with heterogeneous interactions.
\paragraph{Our main contributions.} 
Below, we summarize our main contributions, distinguishing between non-Markovian network games with a finite number of agents, non-Markovian graphon games as their scaling limits, and convergence results for both given and sampled graph sequences.
\begin{itemize} 
\item \textbf{Network games:}  we provide a general framework for deriving explicit finite-player Nash equilibria in terms of operator resolvents for non-Markovian linear-quadratic network games with heterogeneous interactions. Our approach incorporates both progressively measurable idiosyncratic and common noise, and requires minimal coercivity assumptions on the interaction operators between agents (see Assumption~\ref{assum:nonnegative-definite} and Theorem~\ref{thm:finite-player-equilibrium}). Our results generalize Theorem~2.7 from \citet{abijaber2023equilibrium}, who studied a fully connected network with homogeneous interactions.

We apply the results of our general framework to various classes of non-Markovian and Markovian network games, including several important cases that were previously considered intractable. These include stochastic Volterra games with heterogeneous interactions, stochastic differential games with delayed controls arising from heterogeneous inter-bank lending and borrowing models (\citet{carmona2018systemic}, \citet{carmona2015mean}, \citet{sun2022mean}), stochastic differential network games (\citet{aurell2022stochastic}), and stochastic differential games on simple graphs (\citet{lacker2022case}, \citet{hu2024finite}). We refer to Section~\ref{sec:examples} for additional details. 

\item \textbf{Graphon games:} We formulate the infinite-player version of a non-Markovian linear-quadratic network game on a weighted graph as a graphon game, and establish the first-order condition in terms of an infinite-dimensional system of stochastic Fredholm equations, with both forward and backward components (see Proposition~\ref{prop:fredholm-foc-infinite}). We develop a novel approach for explicitly solving this system using the spectral decomposition of the graphon (see Theorem~\ref{thm:infinite-player-equilibrium}). By providing a canonical solution to non-Markovian linear-quadratic graphon games, our approach extends the existing literature on dynamic graphon games (\citet{aurell2022stochastic}, \citet{gao2020linear}, \citet{lacker2023label}, \citet{tangpi2024optimal}), which addresses the Markovian case, to the non-Markovian setting. Note that even in the aforementioned studies on Markovian graphon games, the solutions are often represented in terms of FBSDE systems, which are solvable only in specific examples. Finally, our methods extend the solution techniques developed in \citet{abijaber2023equilibrium}, \citet{abi2024optimal} from finite-dimensional systems of stochastic Fredholm equations to the infinite-dimensional case.

\item \textbf{Convergence:}  We prove convergence and derive a bound on the convergence rate of the finite-player network game to the corresponding graphon game using two approaches. First, we consider a sequence of weighted graphs with adjacency matrices $(w^N)_{N\in\N}$ that converge to a graphon $W$ in the cut norm. We show that the sequence of Nash equilibria for the finite-player games converges to the Nash equilibrium of the corresponding graphon game (see Theorem~\ref{thm:convergence-given-graphs}). This result extends the convergence results of \citet{bayraktar2023propagation}, \citet{carmona2022stochastic}, \citet{cui2021learning}, \citet{gao2020linear}, \citet{lacker2023label}, and \citet{tangpi2024optimal}, among others, to the non-Markovian case in the linear-quadratic setting

Second, we prove convergence for sampled graph sequences (see Theorem~\ref{thm:convergence-sampled-graphs}). We sample sequences of weighted and simple graphs from a fixed graphon 
$W$ using four different procedures and show that the Nash equilibria of the sampled games converge to the Nash equilibrium of the graphon game. This approach, which takes the reverse perspective of the first approach, has been employed by \citet{carmona2022stochastic} and \citet{parise2019graphon} for static games, and by \citet{aurell2022stochastic} for linear-quadratic Markovian games. See Section~\ref{sec:convergence} for a detailed description of the convergence results.
 \end{itemize}

\paragraph{Structure of the paper.} In Section~\ref{sec:finite-player-game} we present our results on network games with memory. In Section~\ref{sec:infinite-player-game} we derive the main results for the corresponding infinite-player graphon game. Section~\ref{sec:convergence} is dedicated to the convergence results. In Section~\ref{sec:examples} we provide applications of our general framework. Sections \ref{sec:proofs-finite}--\ref{sec:proofs-convergence} are dedicated to the proofs of our main results.  

\section{The Finite-Player Game}\label{sec:finite-player-game}
In this section, we formulate the finite-player game  with heterogeneous interactions and derive its Nash equilibrium in explicit form. For this, we first introduce in Section~\ref{subsec:function-spaces} the integral operators which will appear in the players' objective functionals. In Section~\ref{subsec:definition-finite-player-game}, we define the finite-player game. In Section~\ref{subsec:results-finite-player-game}, we then solve it explicitly (see Theorem~\ref{thm:finite-player-equilibrium}).
\subsection{Function spaces and integral operators} \label{subsec:function-spaces}
Let $n\in\N$ be a positive integer. Throughout this paper, $n$ will mostly be equal to either 1 or the number of players $N$ in the finite-player game. We denote by $\langle \cdot, \cdot \rangle_{L^2,n}$ the inner product on $L^2([0,T], \R^n)$, that is 
\be  \label{eq:inner-product}
\langle f, g\rangle_{L^2,n} = \int_0^T f(t)^\top g(t) dt, \quad f,g\in L^2\left([0,T],\mathbb R^n\right),
\ee
and by $\| \cdot\|_{L^2,n}$ the norm induced by it. If $n=1$, we just write $\langle \cdot, \cdot \rangle_{L^2}$ and $\| \cdot\|_{L^2}$, respectively.
We define $L^2\left([0,T]^2,\mathbb R^{n\times n}\right)$ to be the space of Borel-measurable kernels $G:[0,T]^2 \to \R^{n\times n}$ such that 
\be\label{eq:square-integrable}
\int_0^T \int_0^T \|G(t,s)\|^2 dt ds < \infty,
\ee
where $\|\cdot\|$ denotes the Frobenius norm.
For any  kernel $G \in L^2\left([0,T]^2,\mathbb R^{n\times n}\right)$, we denote by {$\boldsymbol G$} the integral operator induced by the kernel $G$, given by
\be\label{eq:induced-operator}
({\boldsymbol G} f)(t)=\int_0^T G(t,s) f(s)ds,\quad f \in L^2\left([0,T],\mathbb R^n\right).
\ee
Due to condition \eqref{eq:square-integrable}, $\boldsymbol G$ is a bounded linear operator from  $L^2\left([0,T],\R^n \right)$ into itself (see Theorem~9.2.4 and Proposition~9.2.7 (iii) in \citet{gripenberg1990volterra}). 
We denote by $G^*$ the adjoint kernel of $G$ with respect to $\langle \cdot, \cdot \rangle_{L^2,n}$, that is 
\be \label{eq:adjoint-operator} 
G^*(t,s) = \; G(s,t)^\top, \quad  (t,s) \in [0,T]^2,
\ee
and by $\boldsymbol{G}^*$ the corresponding adjoint integral operator.  Moreover, we say that a kernel $G\in L^2\left([0,T]^2,\mathbb R^{n\times n}\right)$ is a Volterra kernel, if $G(t,s)=0$ for $s \geq t$. $G$ is called nonnegative definite if for every $f\in L^2\left([0,T],\R^n\right)$ we have  
\be \label{eq:nonnegative-definite}
\int_{0}^T\int_{0}^Tf(t)^\top G(t,s)f(s)dsdt=\langle f,\boldsymbol{G}f\rangle_{L^2,n} \geq 0. 
\ee
\begin{remark}\label{rem:nonnegative-operators}
Note that if $G\in L^2([0,T]^2,\R^{n\times n})$, then it holds that
\be \label{prod-id}
\langle f,\boldsymbol{G}f\rangle_{L^2,n}=\langle f,\boldsymbol{G}^*f\rangle_{L^2,n}=\frac{1}{2}\langle f,(\boldsymbol{G}+\boldsymbol{G}^*)f\rangle_{L^2,n}, \quad  \text{for all } f\in L^2([0,T],\R^n).
\ee
Moreover, recalling \eqref{eq:induced-operator}, \eqref{eq:adjoint-operator} and \eqref{eq:nonnegative-definite}, the following statements are equivalent:
\begin{enumerate}
\item [\textbf{(i)}] The integral operator $\boldsymbol{G}$ is nonnegative definite,
\item [\textbf{(ii)}] The integral operator $\boldsymbol{G}^*$ is nonnegative definite,
\item [\textbf{(iii)}] The integral operator $\boldsymbol{G}+\boldsymbol{G}^*$ is nonnegative definite,
\end{enumerate}
\end{remark}

\begin{definition}
\label{def:admissible-volterra-operator} 
Define $\mathcal G$ to be the class of Volterra kernels in  $L^2\left([0,T]^2,\R\right)$. An integral operator which is induced by a kernel in $\mathcal{G}$ as in \eqref{eq:induced-operator} is called an admissible Volterra operator.
\end{definition} 

\subsection{Definition of the finite-player game} \label{subsec:definition-finite-player-game} 
Let $T>0$ denote a finite deterministic time horizon and let $N\in \N$ be an integer. We fix a filtered probability space $(\Omega, \mathcal F, \mathbb F:=(\mathcal F_t)_{0 \leq t \leq T}, \mathbb P)$ satisfying the usual conditions of right-continuity and completeness. We use the notation $\mathbb E_t\vcentcolon = \mathbb E [\, \cdot\, | \mathcal F_t]$ to represent the conditional expectation with respect to $\mathcal{F}_t$. 
We consider $N$ players $i \in \{1, \ldots, N\}$ who select their controls $\al^{i,N}$ from the set of admissible controls  
\be  
\mathcal A :=\left\{\al:\Omega\times [0,T]\to\R \, : \, \mathbb{F}\textrm{-progressively measurable with } \int_0^T \mathbb E\big[\al_t^2 \big] dt  <\infty \right\},
\ee
and use the notation
\be \label{eq:notation-alpha}
\al^N \vcentcolon= (\al^{1,N},\ldots, \al^{N,N}),\quad \al^{-i,N} \vcentcolon= (\al^{1,N},\ldots,\al^{i-1,N},\al^{i+1,N},\ldots, \al^{N,N}).
\ee
In order to define the individual objective functionals of the players, we introduce integral operators $\boldsymbol{A}^{i}$,  $\boldsymbol{B}^{ij}$, $\boldsymbol{\bar{B}}^{ij}$ and $\boldsymbol{C}^{ijk}$, which satisfy the following assumptions. 
\begin{assumption} \label{assum:A-B-C} Consider the following bounded linear operators on $L^2([0,T],\R)$:
\begin{itemize}
    \item For $i=1,\dots ,N$, we assume that $\boldsymbol{A}^{i}$ is given by 
\be \label{eq:A^i} 
\boldsymbol{A}^{i} := \boldsymbol{B}^{ii} + \lambda^i \boldsymbol{I}, 
\ee
where $\boldsymbol{B}^{ii}$ is an admissible Volterra operator induced by a kernel $B^{ii}\in\mathcal{G}$, $\lambda^i>0$ is a real constant and $\boldsymbol{I}$ denotes the identity operator on $L^2\left([0,T], \R\right)$. Denote by $\boldsymbol{\bar{B}}^{ii}$ and $\bar{B}^{ii}$ copies of $\boldsymbol{B}^{ii}$ and $B^{ii}$, respectively.
\item For $i,j=1,\ldots,N$ with $i\neq j$, we assume that $\boldsymbol{B}^{ij}$ and $\boldsymbol{\bar{B}}^{ji}$ are admissible Volterra operators induced by kernels $B^{ij}$ and $\bar{B}^{ji}$ in $\mathcal{G}$, respectively.
\item For $i,j,k=1,\ldots,N$ with $i\neq j$, $i\neq k$, we assume that $\boldsymbol{C}^{ijk}$ is an admissible Volterra operator induced by a kernel $C^{ijk}\in \mathcal{G}$.
\end{itemize}
Let $b^{ij}= (b^{ij}_t)_{0 \leq t\leq T}$ be $\mathbb{F}$-progressively measurable processes and $c^i$ be $\mathcal{F}_T$-random variables satisfying
\be  \label{eq:assumption-b-c}
\int_0^T \mathbb E\big[(b_t^{ij})^{2} \big] dt < \infty, \quad \E[c^i] < \infty, \quad \text{for }i,j=1,\ldots,N.
\ee
\end{assumption} 
Under Assumption~\ref{assum:A-B-C}, each player $i \in \{1, \ldots, N\}$ seeks to maximize the following individual objective functional:
\be
\begin{aligned}\label{eq:J^iN}
    J^{i,N}(\al^{i,N}; \al^{-i,N }) :=\E\bigg[&-\langle \al^{i,N},\boldsymbol A^{i} \al^{i,N} \rangle_{L^2} - \sum_{\substack{j=1\\j\neq i}}^N \langle \al^{i,N},\big(\boldsymbol B^{ij}+(\boldsymbol{\bar{B}}^{ji})^*\big) \al^{j,N} \rangle_{L^2}  \\ 
    -& \sum_{\substack{j,k=1\\j,k\neq i}}^N\langle \al^{j,N},\boldsymbol{C}^{ijk}\al^{k,N} \rangle_{L^2}+ \langle b^{ii}, \al^{i,N} \rangle_{L^2} +\sum_{\substack{j=1\\j\neq i}}^N \langle b^{ij}, \al^{j,N}\rangle_{L^2} + c^i \bigg]. 
\end{aligned}
\ee
\begin{remark}
The operators $\boldsymbol{B}^{ij}$, $\boldsymbol{\bar{B}}^{ij}$ and $\boldsymbol{C}^{ijk}$ from Assumption~\ref{assum:A-B-C} model the heterogeneous interactions between the $N$ players, while the operators $\lam^i\boldsymbol{I}$ function as regularization terms that ensure the concavity of the objective functionals in \eqref{eq:J^iN} whenever all  the operators $\boldsymbol{B}^{ii}$ are nonnegative definite or Assumption~\ref{assum:nonnegative-definite} holds. If the operators and processes in Assumption~\ref{assum:A-B-C} are chosen to be identical for all $i,j,k=1,\ldots,N$ with $i\neq j$, $i\neq k$, such that $\boldsymbol{C}^{ijk}=\boldsymbol{A}_1/N^2$, 
$\boldsymbol{B}^{ij}=\boldsymbol{\bar{B}}^{ji}=\boldsymbol{A}_3/N+\boldsymbol{A}_1/N^2$,
$\boldsymbol{A}^{i}=\boldsymbol{A}_2+2\boldsymbol{A}_3/N+\boldsymbol{A}_1/N^2$, $b^{ii}=b^i+b^0/N$, and $b^{ij}=b^0/N$, the mean-field framework with homogeneous interactions from \citet{abijaber2023equilibrium} is recovered.
\end{remark}
\begin{remark}\label{rem:J^iN-concise-form}
    For every $i=1,\ldots, N$, consider the matrix-valued Volterra kernel $F^i=(F^i_{jk})_{j,k=1}^N: [0,T]^2\to\R^{N\times N}$ whose entries are given by  
\be
F^i_{jk}(t,s)\vcentcolon =\begin{cases}
    B^{ik}(t,s),\hspace{2cm} &j=i,\\
    \bar{B}^{ji}(t,s), &k=i,\\
    C^{ijk}(t,s), &j\neq i,\,k\neq i,
\end{cases}\qquad \text{for }j,k=1,\ldots, N,\  (t,s)\in [0,T]^2,
\ee
and the $\R^N$-valued $\F$-progressively measurable process given by $b^i_t\vcentcolon=(b_t^{i1},\ldots, b_t^{iN})^\top$. Then, recalling \eqref{eq:inner-product}, \eqref{eq:induced-operator}, \eqref{eq:notation-alpha} and Remark \ref{rem:nonnegative-operators}, the objective functional in \eqref{eq:J^iN} can be written in the following compact form
\be
\label{eq:J^iN-concise-form}
    J^{i,N}(\al^{i,N}; \al^{-i,N }) =  \, \E\left[ -\lam^i\langle \al^{i,N}, \al^{i,N} \rangle_{L^2}-\langle \al^N,\boldsymbol F^{i} \al^N \rangle_{L^2,N} + \langle b^{i}, \al^N \rangle_{L^2,N} + c^i \right].  
\ee
\end{remark}
The main goal of this section is to solve simultaneously for each of the $N$ players their individual stochastic optimal control problem
\begin{equation} 
\underset{\al^{i,N} \in \mathcal{A}} \max \ J^{i,N}(\al^{i,N};\al^{-i,N}), \quad i=1,...,N. 
\end{equation}
The solution will establish a Nash equilibrium in the following sense.
\begin{definition} \label{def:Nash-finite-player}
A set of strategies $\hat{\al}^N = (\hat{\al}^{1,N},\ldots,\hat{\al}^{N,N}) \in \mathcal{A}^N$ is called a  Nash equilibrium if for all $i \in \{1,\dots, N\}$ and for all admissible strategies $\beta\in \mathcal{A}$ it holds that
\be
J^{i,N}(\hat{\al}^{i,N};\hat{\al}^{-i,N}) \geq J^{i,N}(\beta;\hat{\al}^{-i,N}). 
\ee
\end{definition}

\subsection{Main results for the finite-player game} \label{subsec:results-finite-player-game}
For convenience, we introduce the Volterra kernels in $L^2([0,T]^2,\R^{N\times N})$ given by
\be\label{eq:B-Bbar}
B\vcentcolon = \begin{pmatrix} 
    B^{11} & \dots  & B^{1N}\\
    \vdots & \ddots        & \vdots\\
    B^{N1} & \dots  & B^{NN} 
    \end{pmatrix}, \quad\quad
\bar B\vcentcolon =     \begin{pmatrix} 
    \bar{B}^{11} & \dots  & \bar{B}^{1N}\\
    \vdots & \ddots        & \vdots\\
    \bar{B}^{N1} & \dots  & \bar{B}^{NN} 
    \end{pmatrix},
\ee
and denote by $\boldsymbol{B}$ and $\boldsymbol{\bar{B}}$ the bounded linear operators on $L^2([0,T],\R^N)$ induced by them. We also set
\be\label{eq:Lambda}
\Lambda\vcentcolon=\operatorname{diag}(\lambda^1,\ldots,\lambda^N)\in\R^{N\times N},
\ee
and
\be \label{eq:b} 
b:\Omega\times [0,T]\to\R^N,\quad    b_t := \big(b^{ii}_t \big)_{i\in\{1,\ldots,N\}}.
\ee
We can characterize the Nash equilibrium of the finite-player game as in Definition~\ref{def:Nash-finite-player} by a system of stochastic Fredholm equations as follows.
\begin{proposition}\label{prop:fredholm-foc}
Let Assumption~\ref{assum:A-B-C} hold and assume that the objective functionals $\al^{i,N} \mapsto J^{i,N}(\al^{i,N}; \al^{-i,N})$ are strictly concave in $\al^{i,N}\in\mathcal{A}$ for all $i\in\{1,\ldots, N\}$. Then, a set of strategies $\al^N=(\al^{1,N},\dots,\al^{N,N})\in\mathcal{A}^N$ is a Nash equilibrium if and only if it satisfies the $N$-dimensional coupled system of stochastic Fredholm equations of the second kind given by
\be \label{eq:fredholm-foc}
2 \Lambda \al^N_t =  b_t  - \int_0^t   B(t,s) \al^N_s ds -  \int_t^T \bar{B}(s,t)^\top  \E_t [\al^N_s] ds, \quad d\P \otimes dt \textrm{-a.e.~on } \Omega \times [0,T], 
\ee
where $B$, $\bar{B}$, $\Lambda$ and $b$ are defined in \eqref{eq:B-Bbar}, \eqref{eq:Lambda} and \eqref{eq:b}, respectively. 
\end{proposition}

The proof of Proposition~\ref{prop:fredholm-foc} will be given in Section~\ref{sec:proofs-finite}.

In order to explicitly derive the Nash equilibrium of the finite-player game, we make the following  assumption on $\boldsymbol{B}$, $\boldsymbol{\bar{B}}$ and $\Lambda$.
\begin{assumption}\label{assum:nonnegative-definite}
Recalling \eqref{eq:B-Bbar} and \eqref{eq:Lambda}, assume that there exists a constant $c_0>0$ such that  
\be
\langle f, \left(\boldsymbol{B}+\boldsymbol{\bar{B}}^*\right)f+2\Lambda f-c_0 f\rangle_{L^2,N}\geq 0,\quad \text{for all } f\in L^2([0,T],\R^N),
\ee
i.e. the operator
\be
\boldsymbol{B}+\boldsymbol{\bar{B}}^*+2\Lambda\boldsymbol{I}^N-c_0\boldsymbol{I}^N
\ee
is nonnegative definite, where $\boldsymbol{I}^N$ denotes the identity operator on $L^2([0,T],\R^N)$.
\end{assumption}
\begin{remark}
Assumption~\ref{assum:nonnegative-definite} equivalently states that the bilinear form on the Hilbert space $L^2([0,T],\R^N)$ defined by 
\be
(f,g)\mapsto \big\langle (\boldsymbol{B}+\boldsymbol{\bar{B}}^*+2\Lambda\boldsymbol{I}^N)f,g\big\rangle_{L^2,N},\quad f,g\in L^2([0,T],\R^N),
\ee
is coercive. Notice that in order to derive the a solution to \eqref{eq:fredholm-foc} we need to solve the the stochastic analog to deterministic equations of the form,   
\be
(\boldsymbol{B}+\boldsymbol{\bar{B}}^*+2\Lambda\boldsymbol{I}^N)f=h
\ee
for a fixed $h$, where $f,h\in L^2([0,T],\R^N)$. Assumption~\ref{assum:nonnegative-definite} is a natural coercivity condition to this class of equations, which similarly appears in the celebrated Lax-Milgram theorem (\citet{LaxMilgram+1955+167+190}).
\end{remark}
The following two definitions are needed to state the main theorem of this section.
\begin{definition} \label{def:truncated-kernel}
For a Volterra kernel $G\in L^2([0,T]^2,\R^{N\times N})$ and a fixed $t\in[0,T]$ define the truncated Volterra kernel 
\be
G_t(s,r)\vcentcolon=\mathds{1}_{\{r> t\}}G(s,r),
\ee
and let $\boldsymbol{G}_t$ denote the induced integral operator on $L^2([0,T],\R^N)$.
\end{definition}
\begin{definition}\label{def:column-wise}
Let $h\in L^2([0,T],\R^{N\times N})$ be a matrix-valued function in one variable and $\boldsymbol{D}$ be a linear operator from $L^2([0,T],\R^N)$ into itself. We define the column-wise application of $\boldsymbol{D}$ to $h$ as
\be
\boldsymbol{D}\diamond h :[0,T]\to \R^{N\times N},\quad 
\big(\boldsymbol{D}\diamond h\big)(r)\vcentcolon = \big((\boldsymbol{D}h_{\bullet 1})(r),\dots,(\boldsymbol{D}h_{\bullet N})(r)\big),
\ee
where for $i=1,\ldots,N$,
\be
h_{\bullet i}:[0,T]\to\R^N,\quad h_{\bullet i}(r)\vcentcolon =h(r)\vec{e_i},
\ee
is given by the $i$-th column of $h$. 
\end{definition}
We now state the main theorem of this section, which solves the finite-player game.
\begin{theorem}\label{thm:finite-player-equilibrium}
Under Assumptions~\ref{assum:A-B-C} and \ref{assum:nonnegative-definite}, there exists a unique Nash equilibrium $\hat{\al}^N = (\hat{\al}^{1,N},\ldots,\hat{\al}^{N,N}) \in \mathcal{A}^N$ of the finite-player game given by
\be\label{eq:finite-player-equilibrium}
\hat\al_t^N=\big((\boldsymbol{I}^N+\boldsymbol{E})^{-1}\gamma\big)(t),\quad 0\leq t \leq T,
\ee
where
\be
\gamma_t\vcentcolon=\frac{1}{2}\Lambda^{-1}\Big(b_t-\int_t^T \bar B(r,t)^\top \big(\boldsymbol{D}_t^{-1}\mathds{1}_{\{t<\cdot\}}\E_t[b_\cdot]\big)(r)dr\Big),
\ee
\be
E(t,s)\vcentcolon=-\frac{1}{2}\Lambda^{-1}\Big(\mathds{1}_{\{t>s\}}\int_t^T \bar B(r,t)^\top \big(\boldsymbol{D}_t^{-1}\diamond\mathds{1}_{\{t< \cdot\}}B(\cdot,s)\big)(r)dr-B(t,s)\Big),\hspace{3cm} \\[1ex]
\ee
\be
\boldsymbol{D}_t\vcentcolon=2\Lambda\boldsymbol{I}^N+\boldsymbol{B}_t+\boldsymbol{\bar{B}}^*_t,
\ee
where $B$, $\bar{B}$, $\Lambda$ and $b$ are defined in \eqref{eq:B-Bbar}, \eqref{eq:Lambda} and \eqref{eq:b} respectively, $\boldsymbol{E}$ denotes the integral operator induced by the Volterra kernel $E$, and $\boldsymbol{I}^N$ denotes the identity operator on $L^2([0,T],\R^N)$.
\end{theorem}
The proof of Theorem~\ref{thm:finite-player-equilibrium} will be given in Section~\ref{sec:proofs-finite}.

\begin{remark} Theorem~\ref{thm:finite-player-equilibrium} yields the explicit finite-player Nash equilibrium for a general class of linear-quadratic non-Markovian network games with heterogeneous interactions in terms of operator resolvents. It allows for progressively measurable idiosyncratic and common noise, while only requiring the minimal coercivity Assumption~\ref{assum:nonnegative-definite} on the interaction operators. Moreover, it generalizes Theorem~2.7 of \citet{abijaber2023equilibrium}, by allowing for heterogeneous player interactions. Finally, it allows to solve various classes of non-Markovian and Markovian network games from the literature as described in Section~\ref{sec:examples}.
\end{remark} 

\section{The Infinite-Player Game}\label{sec:infinite-player-game}
It is natural to consider the infinite-player game which is obtained from sending the number of players in the finite-player game introduced in Section~\ref{subsec:definition-finite-player-game} to infinity. We are particularly interested in infinite-player games, where the interaction of the players is incorporated via a graphon $W:[0,1]^2\to [0,1]$.  Therefore, we  consider in Section~\ref{subsec:definition-finite-player-game-graph} a non-Markovian finite-player game on a weighted graph, which is a special case of the finite-player game introduced in Section~\ref{subsec:definition-finite-player-game}. In Section~\ref{subsec:definition-infinite-player-game}, we then formulate the corresponding non-Markovian infinite-player game on a graphon $W$. Finally, in Section~\ref{subsec:results-infinite-player-game}, we solve the graphon game explicitly (see Theorem~\ref{thm:infinite-player-equilibrium}). Later, in Section~\ref{sec:convergence} we prove the convergence of the finite-player game introduced in Section~\ref{subsec:definition-finite-player-game-graph} to the graphon game. 

\subsection{Definition of the finite-player game on a weighted graph}\label{subsec:definition-finite-player-game-graph}
Let $N\in\N$ be a positive integer. Recalling Assumption~\ref{assum:A-B-C}, assume  that the operators $\boldsymbol{A}^{i}$,  $\boldsymbol{B}^{ij}$, $\boldsymbol{\bar{B}}^{ij}$, $\boldsymbol{C}^{ijk}$ and the processes $b^{ij}$ therein are chosen in such a way that the objective functional in \eqref{eq:J^iN} takes the form
\be
\begin{aligned}\label{eq:J_0^iN}
    J_{0}^{i,N}(\al^{i,N}; \al^{-i,N }) :=\E\Big[ &- \langle \al^{i,N}, (\wt{\boldsymbol A}+\lambda \boldsymbol{I}) \al^{i,N} \rangle_{L^2} -\frac{1}{N}\sum_{j=1}^N w_{ij}^N\langle \al^{i,N},\big(\wt{\boldsymbol B}+\wt{\boldsymbol B}^*\big) \al^{j,N} \rangle_{L^2}   \\ 
    &- \frac{1}{N^2}\sum_{j,k=1}^Nw_{ij}^Nw_{ik}^N\langle \al^{j,N},\wt{\boldsymbol{C}}\al^{k,N} \rangle_{L^2}
    + \langle b^{i,N}, \al^{i,N} \rangle_{L^2} \\&+\frac{1}{N}\sum_{j=1}^N w_{ij}^N\langle b^{*,N}, \al^{j,N}\rangle_{L^2} + c^{i,N} \Big], 
\end{aligned}\ee
so that the following special case of Assumption~\ref{assum:A-B-C} is satisfied. 
\begin{assumption}\label{assum:A-B-C-graph}
Assume that $\wt{\boldsymbol A}$, $\wt{\boldsymbol B}$, $\wt{\boldsymbol C}$ are admissible operators induced by Volterra kernels $\wt{A},\wt{B},\wt{C}\in\mathcal{G}$.
Assume that $b^{i,N}$~and $b^{*,N}$ are $\F$-progressively measurable processes in $L^2(\Omega\times [0,T],\R)$ incorporating idiosyncratic and common noise, respectively, and that $c^{i,N}$ are integrable $\mathcal{F}_T$-random variables. Moreover, let $\lambda>0$ be a fixed constant, and assume that weights 
$w_{ij}^N\in[0,1]$ which measures player $i$'s sensitivity to player $j$'s control satisfy $w_{ij}^N=w_{ji}^N$ and $w_{ii}^N=0$ for all $i,j=1,\ldots,N$.
\end{assumption}
For fixed $w^N\vcentcolon=(w_{ij}^N)\in[0,1]^{N\times N}$ we introduce the notation 
\be\label{eq:aggregate}
\bar\al^{i,N}\vcentcolon=\frac{1}{N}\sum_{j=1}^N w_{ij}^N\al^{j,N},\quad i=1,\ldots,N,
\ee
which allows us to rewrite the objective functional in \eqref{eq:J_0^iN} in concise form,
\be
\begin{aligned}\label{eq:J_0^iN-rewritten}
J_{0}^{i,N}(\al^{i,N}; \al^{-i,N }) =\E\Big[ &- \langle \al^{i,N}, (\wt{\boldsymbol A}+\lambda \boldsymbol{I}) \al^{i,N} \rangle_{L^2} -\langle \al^{i,N},\big(\wt{\boldsymbol B}+\wt{\boldsymbol B}^*\big) \bar\al^{i,N} \rangle_{L^2}   \\ 
&- \langle \bar\al^{i,N},\wt{\boldsymbol{C}}\bar\al^{i,N} \rangle_{L^2}
+ \langle b^{i,N}, \al^{i,N} \rangle_{L^2} +\langle b^{*,N}, \bar\al^{i,N}\rangle_{L^2} + c^{i,N} \Big].
\end{aligned}\ee
From \eqref{eq:J_0^iN-rewritten} is follows that the objective functional of each player is linear-quadratic in the player's control and in the local aggregate of the controls $\bar\al^{\cdot,N}$. 
\begin{remark}
Note that in \eqref{eq:J_0^iN-rewritten} we assume the operators $\wt{\boldsymbol A}$, $\wt{\boldsymbol B}$, $\wt{\boldsymbol C}$ and $\lam>0$ to be identical for all players and not depending on $N$, but allow for heterogeneous players' sensitivity $w^{\cdot,N}$, idiosyncratic noise $b^{\cdot,N}$, the common noise $b^{*,N}$ and $c^{\cdot,N}$ that may change with $N$.
\end{remark}
Assumption~\ref{assum:nonnegative-definite}, which is needed to solve the finite-player game by means of Theorem~\ref{thm:finite-player-equilibrium}, simplifies to the following assumption after recalling Remark \ref{rem:nonnegative-operators}.
\begin{assumption}\label{assum:nonnegative-definite-graph}
Assume that there exists a constant $c_0>0$ such that  
\be
\langle f, \lambda f +\wt{\boldsymbol{A}}f+ \frac{1}{N}w^N\cdot\wt{\boldsymbol{B}}f-c_0 f\rangle_{L^2,N}\geq 0,\quad \text{for all } f\in L^2([0,T],\R^N),
\ee
where the operators $\wt{\boldsymbol{A}}$ and $\wt{\boldsymbol{B}}$ on $L^2([0,T],\R)$ are applied component-wise.
\end{assumption}
Finally, the central Theorem~\ref{thm:finite-player-equilibrium} simplifies as follows.
\begin{corollary} \label{cor:finite-player-equilibrium}
    Under Assumptions~\ref{assum:A-B-C-graph} and \ref{assum:nonnegative-definite-graph}, there exists a unique Nash equilibrium $\hat{\al}^N=(\hat{\al}^{1,N},\ldots, \hat{\al}^{N,N})\in\mathcal{A}^N$ of the finite-player game on the weighted graph with adjacency matrix $w^N$. It is given explicitly by equation \eqref{eq:finite-player-equilibrium}, where 
    \be\label{eq:correspondence-weighted}
    B=\bar{B}=\operatorname{Id}^N\cdot \wt{A}+\frac{1}{N}w^N\cdot \wt{B},\qquad
    \Lambda = \lam\cdot \operatorname{Id}^N,\qquad b=(b^{1,N},\ldots,b^{N,N}),
    \ee
    and $\operatorname{Id}^N$ denotes the $N$-dimensional identity matrix.
\end{corollary}

\subsection{Definition of the infinite-player graphon game}\label{subsec:definition-infinite-player-game}
We will label the agents amid the unit interval by $u\in [0,1]$. Inspired by \citet{aurell2022stochastic} and \citet{tangpi2024optimal}, the infinite-player graphon game will be modeled by the following setup.
Let $\mathcal{B}_{[0,1]}$ be the Borel $\sigma$-algebra of $[0,1]$, and $\mu_{[0,1]}$ denote the Lebesgue measure on $[0,1]$. Let $(\Omega,\mathcal{F},\P)$ be the sample space and $([0,1],\mathcal{I},\mu)$ be a probability space that extends the Lebesgue measure space $([0,1],\mathcal{B}_{[0,1]},\mu_{[0,1]})$. We will consider a rich Fubini extension $([0,1]\times\Omega,\mathcal{I}\boxtimes \mathcal{F},\mu\boxtimes\P)$ of the standard product space $([0,1]\times\Omega,\mathcal{I}\otimes \mathcal{F},\mu\otimes\P)$. We refer to \citet{sun2006exact} for a detailed discussion of the theory of (rich) Fubini extensions and \citet{sun2009individual} for particular results on their existence. Namely, by Theorem~1 in \citet{sun2009individual}, it is possible to define $\mathcal{I}\boxtimes \mathcal{F}$-measurable processes $(b,c): [0,1]\times\Omega \to L^2([0,T],\R)\times\R$ such that the random variables $(b^u,c^u)_{u\in [0,1]}$ are essentially pairwise independent, and for each $u\in [0,1]$, the process $b^u=(b_t^u)_{0\leq t\leq T}$ is a stochastic process in $L^2(\Omega\times [0,T],\R)$ and $c^u$ is an integrable random variable on $(\Omega,\mathcal{F},\P)$. Here the pairs $(b^u,c^u)$ have to satisfy 
the technical condition that the map from $([0,1],\mathcal{I},\mu)$ to the space of Borel probability measures on $L^2([0,T],\R)\times\R$ which assigns to each $u\in [0,1]$ the distribution of $(b^u,c^u)$ is measurable. Finally, assume that there is an independent stochastic process $b^*$ in $L^2(\Omega\times [0,T],\R)$ incorporating common noise.
\begin{remark}
By Definition~2.2 in \citet{sun2006exact}, the Fubini property holds on the rich Fubini extension $([0,1]\times\Omega,\mathcal{I}\boxtimes \mathcal{F},\mu\boxtimes\P)$. More precisely, for any $\mu\boxtimes\P$-integrable function $f:[0,1]\times\Omega\to\R$ it holds that
\be
\int_{[0,1]\times\Omega}f(u,\omega)(\mu\boxtimes\P)(du,d\omega)=\int_0^1\E[f(u)]\mu(du)=\E\Big[\int_0^1f(u)\mu(du)\Big].
\ee
Also see Lemma~2.3 in \citet{sun2006exact} for a generalized Fubini property. 
We will often denote $\mu(du)=du$ for ease of notation and will tacitly employ the Fubini property without referring to this remark again.
\end{remark}
Throughout the rest of Section~\ref{sec:infinite-player-game}, denote by $\F\vcentcolon=(\mathcal{F}_t)_{0\leq t\leq T}$ the augmentation of the filtration generated by $((b^u)_{u\in [0,1]},(c^u)_{u\in [0,1]},b^*)$.
\begin{definition}
    A strategy profile is a family $(\al^u)_{u\in [0,1]}$ of $\F$-progressively measurable processes such that the map $(u,t,\omega)\to\al_t^u(\omega)$ is $\mathcal{I}\otimes\mathcal{B}([0,T])\otimes\mathcal{F}$-measurable.
Define the class of admissible strategy profiles $(\al^u)_{u\in [0,1]}$ as
\be
\mathcal{A}^\infty\vcentcolon=\left\{(\al^u)_{u\in [0,1]}:\int_0^T\E[(\al_t^u)^2]dt<\infty \ \text{for $\mu$-a.e. }u, \int_0^1\int_0^T\E[(\al_t^u)^2]dtdu<\infty\right\}.
\ee
\end{definition}
Slightly abusing notation, we will sometimes write $(\al^{v})_{v\neq u}\in\mathcal{A}^\infty$ for an admissible strategy profile $(\al^{v})_{v\in [0,1]}$ and $u\in [0,1]$.
Next, in order to incorporate interaction among the players, we fix a graphon $W$, i.e. a Borel-measurable and symmetric function $W:[0,1]\times [0,1]\to [0,1]$. Notice that in accordance with the theory from Section~\ref{subsec:function-spaces}, $W$ defines a self-adjoint bounded linear integral operator $\boldsymbol{W}:L^2([0,1],\R)\to L^2([0,1],\R)$ by
\be\label{eq:graphon-operator}
(\boldsymbol{W}f)(u)\vcentcolon=\int_0^1 W(u,v)f(v)dv,\quad f\in L^2([0,1],\R).
\ee
\begin{remark}\label{rem:spectral-decomposition}
As noted by \citet{lovasz2012large}, the operator $\boldsymbol{W}:L^2([0,1],\R)\to L^2([0,1],\R)$ is a  Hilbert-Schmidt operator and thus diagonizable. That is, $\boldsymbol{W}$ has countably many eigenvalues, all of which are real and can be ordered as $\vartheta_1\geq\vartheta_2\geq \ldots$. Moreover, there exists an orthonormal basis of $L^2([0,1],\R)$ of corresponding eigenfunctions $(\varphi_j)_{j\in\N}$, which implies that any function $f\in L^2([0,1],\R)$ can be decomposed as 
$$
f(u)=\sum_{j=1}^\infty \langle f,\varphi_j\rangle_{L^2([0,1],\R)}\varphi_j(u),\quad u\in [0,1].
$$
In particular, it follows that 
\be\label{eq:representation-Wf}
(\boldsymbol{W}f)(u)=\sum_{j=1}^\infty\vartheta_j\langle f,\varphi_j\rangle_{L^2([0,1],\R)}\varphi_j(u),\quad u\in [0,1].
\ee
\end{remark}
The following assumption is the graphon game analogue of Assumption~\ref{assum:A-B-C-graph}.

\begin{assumption}\label{assum:A-B-C-graphon}
Assume that $\wt{\boldsymbol A}$, $\wt{\boldsymbol B}$, $\wt{\boldsymbol C}$ are admissible operators induced by Volterra kernels $\wt{A},\wt{B},\wt{C}\in\mathcal{G}$ satisfying the technical integrability condition
\be\label{eq:sup-A+B-condition}
\sup_{0\leq t\leq T}\int_0^T|\wt{A}(t,s)|^2+|\wt{B}(t,s)|^2ds+\sup_{0\leq s\leq T}\int_0^T|\wt{A}(t,s)|^2+|\wt{B}(t,s)|^2dt<\infty.
\ee
For $u\in I$, assume that $b^u$ and $c^u$ are  essentially pairwise independent \mbox{$\F$-progressively} measurable processes in $L^2(\Omega\times [0,T],\R)$ and integrable $\mathcal{F}_T$-random variables, respectively, and assume that $b=(b_t^u)\in L^2( \Omega\times[0,T]\times [0,1],\R)$, where the map which assigns to each $u\in [0,1]$ the distribution of $(b^u,c^u)$ should be measurable. Moreover, let $b^*$ be an independent $\F$-progressively measurable process in $L^2(\Omega\times [0,T],\R)$ incorporating common noise. Finally, let $\lambda>0$ be a fixed constant, and assume that 
$W:[0,1]^2\to[0,1]$ is a graphon inducing the operator $\boldsymbol{W}$.
\end{assumption}
Recall the objective functionals of the finite-player game in \eqref{eq:J_0^iN-rewritten}. Under Assumption~\ref{assum:A-B-C-graphon}, the objective functionals of the corresponding infinite-player graphon game are defined as follows,
    \be
    \begin{aligned}\label{eq:J^uW}
    J^{u,W}(\al^{u}; (\al^{v})_{v\neq u}) \vcentcolon=\E\Big[ &- \langle \al^u, (\wt{\boldsymbol A}+\lambda \boldsymbol{I}) \al^u \rangle_{L^2} -\langle \al^u,\big(\wt{\boldsymbol B}+\wt{\boldsymbol B}^*\big) (\boldsymbol{W}\al)(u) \rangle_{L^2}   \\ 
    &- \langle (\boldsymbol{W}\al)(u),\smash{\wt{\boldsymbol{C}}}(\boldsymbol{W}\al)(u) \rangle_{L^2}
    + \langle b^u, \al^u \rangle_{L^2} \\&+\langle b^*, (\boldsymbol{W}\al)(u)\rangle_{L^2} + c^u \Big], \quad u\in [0,1]. 
\end{aligned}\ee
\begin{definition} \label{def:Nash-infinite-player}
An admissible strategy profile $(\hat{\al}^{u})_{u\in [0,1]}\in\mathcal{A}^\infty$ is called a graphon game Nash equilibrium  if for $\mu$-almost every $u\in [0,1]$ the control $\hat{\al}^u$ solves the optimization problem 
\be
\underset{\al^u}\max \ J^{u,W}(\al^{u}; (\hat{\al}^{v})_{v\neq u}),
\ee
where maximization takes place over the set of $\F$-progressively measurable processes $\al^u$ in $L^2(\Omega\times [0,T],\R)$. 
\end{definition}

\subsection{Main results for the graphon game}\label{subsec:results-infinite-player-game}
We can characterize the Nash equilibrium in the graphon game by a an infinite-dimensional stochastic Fredholm equation with forward and backwards components as follows.
\begin{proposition}\label{prop:fredholm-foc-infinite}
Let Assumption~\ref{assum:A-B-C-graphon} be satisfied and assume that $\wt{\boldsymbol A}$ is nonnegative definite. Then,  an admissible strategy profile $\al=(\al^{u})_{u\in [0,1]}\in\mathcal{A}^\infty$ is a graphon game Nash equilibrium if and only if it satisfies the  infinite-dimensional coupled system of stochastic Fredholm integral equation of the second kind  given by,
\be\begin{aligned} \label{eq:fredholm-foc-infinite} 
2\lam \al_t^u&=b_t^u - \int_0^t \wt{A}(t,s)\al^u_sds- \int_t^T \wt{A}(s,t)\E_t[\al^u_s]ds- \int_0^t\wt{B}(t,s)  \int_0^1 W(u,v)\al^v_sdvds\\
&- \int_t^T\wt{B}(s,t)  \int_0^1 W(u,v)\E_t[\al^v_s]dvds,\quad d\P \otimes dt\otimes d\mu \textrm{-a.e.~on } \Omega \times [0,T]\times [0,1].
\end{aligned}\ee
\end{proposition}
The proof of Proposition~\ref{prop:fredholm-foc-infinite} is similar to the proof of Proposition~\ref{prop:fredholm-foc} and deferred to the Appendix \ref{appendix}. 

The following assumption is the graphon game analogue of Assumption~\ref{assum:nonnegative-definite-graph}.
\begin{assumption}\label{assum:nonnegative-definite-graphon}
Assume that the linear operator $\wt{\boldsymbol A}$ on $L^2([0,T],\R)$ is nonnegative definite, and that there exists a constant $ c_W>0$ such that 
$$
\langle g, \lam g+\wt{\boldsymbol A} g + \wt{\boldsymbol B}(\boldsymbol W g)- c_W g\rangle_{L^2([0,T]\times [0,1],\R)}\geq 0,\quad\text{for all }g\in L^2([0,T]\times [0,1],\R).
$$
\end{assumption}

We now state the main theorem of this section, which solves the graphon game. We recall from Remark \ref{rem:spectral-decomposition} that $(\varphi_i)_{i\in\N}$ is an orthonormal basis of $L^2([0,1],\R)$ consisting of eigenfunctions of $\boldsymbol{W}$ together with a nonincreasing sequence of corresponding eigenvalues $(\vartheta_i)_{i\in\N}$ arising from the spectral decomposition of $\boldsymbol{W}$ from \eqref{eq:representation-Wf}.
\begin{theorem}\label{thm:infinite-player-equilibrium} Let Assumptions~\ref{assum:A-B-C-graphon} and \ref{assum:nonnegative-definite-graphon} be satisfied. Then there exists a unique graphon game Nash equilibrium $\hat{\al}\in\mathcal{A}^\infty$ of the infinite-player graphon game given by 
\be
\hat{\al}_t^u = \sum_{i=1}^\infty \varphi_i(u) \big((\boldsymbol{I} + \boldsymbol E^i)^{-1} \gamma^i \big) (t),\quad 0\leq t\leq T,\ 0\leq u\leq 1,
\ee
where for every $i\in\N$,
\be \label{eq:gamma^i}
\gamma_t^i \vcentcolon=   \wt{b}^i_t  -  \langle  \mathds{1}_{\{t< \cdot\}}  \wt{C}^i(\cdot,t),  ({\boldsymbol D}_t^i)^{-1}\mathds{1}_{\{t< \cdot\}} \E_t[\wt{b}^i_{\cdot}] \rangle_{L^2},
\ee
\be\label{eq:E^i}
E^i(t,s) \vcentcolon= - \mathds{1}_{\{t>s\}}  \big( \langle  \mathds{1}_{\{t< \cdot\}}   \wt{C}^i(\cdot,t), ({\boldsymbol D}_t^i)^{-1}   \mathds{1}_{\{t< \cdot\}}  \wt{C}^i(\cdot,s)  \rangle_{L^2}    -   \wt{C}^i(t,s)   \big),
\ee
\be
\wt{b}^i_t\vcentcolon=\int_0^1\varphi_i(u)b_t^udu,\quad
\wt{C}^i\vcentcolon=\frac{1}{2\lam}\big(\wt{A}+\vartheta_i\wt{B}\big),\quad
{\boldsymbol{D}}^i_t \vcentcolon= \boldsymbol{I}  + \wt{\boldsymbol{C}}^i_t + (\wt{\boldsymbol{C}}^i_t)^*.
\ee
\end{theorem}
The proof of Theorem~\ref{thm:infinite-player-equilibrium} will be given in Section~\ref{sec:proofs-infinite}.
\begin{remark}
Theorem~\ref{thm:infinite-player-equilibrium} provides a canonical solution to non-Markovian linear-quadratic graphon games, extending the previous literature on dynamic graphon games (\citet{aurell2022stochastic,gao2020linear,lacker2023label,tangpi2024optimal}) to non-Markovian models.  Moreover, while the solutions in these works on Markovian graphon games are often represented in terms of FBSDE systems, Theorem~\ref{thm:infinite-player-equilibrium} allows to solve a general class of both non-Markovian and Markovian graphon games explicitly in terms of operator resolvents. Although we have already solved in Theorem~\ref{thm:finite-player-equilibrium} the corresponding finite-player game, Theorem~\ref{thm:infinite-player-equilibrium} is still of great importance, since it allows to approximate $N$-player Nash equilbria for many large $N$ at once (see Section~\ref{sec:convergence}), and Assumption~\ref{assum:nonnegative-definite-graphon} is usually more tractable than Assumption~\ref{assum:nonnegative-definite-graph}. Finally, by solving the coupled system of stochastic Fredholm equations from Proposition~\ref{prop:fredholm-foc-infinite} as part of the proof, we also extend the solution techniques developed in \citet{abijaber2023equilibrium,abi2024optimal} from the finite-dimensional to the infinite-dimensional case. The latter is significantly more involved, as the theory of operator resolvents can no longer be applied directly. Instead, we first employ a spectral decomposition of the graphon to transform the system into a more manageable form, which we solve to then recover a solution to the original infinite-dimensional system.
\end{remark}

\section{Convergence Results}\label{sec:convergence}
In this section, we focus on the convergence of the finite-player game on a graph from Section~\ref{subsec:definition-finite-player-game-graph} to the infinite-player game on a graphon from Section~\ref{subsec:definition-infinite-player-game}. In previous work on graphon games, there have mainly been two approaches to this problem. First, one can start from a sequence of weighted graphs with adjacency matrices $(w^N)_{N\in\N}$ that converge to a graphon $W$ in a suitable sense, and show that the sequence of Nash equilbria of the corresponding finite-player games converges to the one of the corresponding graphon game. This approach is natural and has been employed for instance by \citet{bayraktar2023propagation,carmona2022stochastic,cui2021learning,gao2020linear}, \citet{lacker2023label} and \citet{tangpi2024optimal}. We will follow it in Section~\ref{subsection:convergence-general-games} and derive a convergence result for general graphons in Theorem~\ref{thm:convergence-given-graphs}. Second, one can fix a graphon $W$ and its corresponding infinite-player game from Section~\ref{subsec:definition-infinite-player-game} and sample from $W$ both weighted and simple graphs according to various common sampling procedures and thereby corresponding finite-player games on graphs as in Section~\ref{subsec:definition-finite-player-game-graph}, and show that the sampled game Nash equilibria converge to the graphon game Nash equilibrium. This approach takes the reverse perspective of the first approach and has been employed, among others, by \citet{aurell2022stochastic,carmona2022stochastic} and \citet{parise2019graphon}. We follow it in Section~\ref{subsection:convergence-sampled-games} and give a convergence result for blockwise Lipschitz graphons in Theorem~\ref{thm:convergence-sampled-graphs}.


\subsection{Preliminaries}\label{subsec:preliminaries}
We start by putting the finite-player game and the infinite-player game into the same probabilistic setting. This is done by assuming that for every $N\in\N$ the augmentations of the filtrations generated by $(b^{1,N},\ldots,b^{N,N},b^{*,N},c^{1,N},\ldots, c^{N,N})$ from Section~\ref{subsec:definition-finite-player-game-graph} are contained in the filtration $\F$ from Section~\ref{subsec:definition-infinite-player-game}.

First, we show that any $N$-player game can be equivalently reformulated as a graphon game. In the $N$-player game, a Nash equilibrium is an $N$-tuple of processes in $L^2(\Omega\times[0,T],\R)$, whereas in the graphon game, a Nash equilibrium is an element in $L^2(\Omega\times[0,T]\times [0,1],\R)$. To compare these two objects we introduce for every $N\in\N$ a uniform partition of $[0,1]$ given by
\be \label{partition}
\mathcal{P}^N=\{\mathcal{P}_1^N,\ldots,\mathcal{P}_N^N\}, \quad
\mathcal{P}_i^N:=\begin{cases}
    [\tfrac{i-1}{N},\tfrac{i}{N}),\quad\quad\text{for } 1\leq i\leq N-1,\\
    [\tfrac{N-1}{N},1],\hspace{8.55mm} \text{for } i= N.
\end{cases}
\ee
The idea is to pair each player $i$ in the $N$-player game with the interval $\mathcal{P}_i^N\subset [0,1]$. Namely, for a family $\al^N=(\al^{1,N},\ldots,\al^{N,N})$ of processes $\al^{i,N}\in L^2(\Omega\times[0,T],\R)$, define the corresponding step function strategy profile $\al^{N}_{\operatorname{step}}\in L^2(\Omega\times[0,T]\times [0,1],\R)$ by
\be\label{eq:step-function}
\al^{u,N}_{\operatorname{step}}:=\al^{i,N},\quad\forall u\in\mathcal{P}_i^N,\  i=1,\ldots, N.
\ee
Similarly, the partition $\mathcal{P}^N$ allows us to define for any simple edge-weighted graph with adjacency matrix $w^N\in [0,1]^{N\times N}$ a corresponding step graphon $W^N:[0,1]^2\to [0,1]$ given by
\be \label{eq:step-graphon}
W^N(u,v)\vcentcolon=w_{ij}^N,\quad \forall \,(u,v)\in\mathcal{P}_i^N\times\mathcal{P}_j^N,\ i,j=1,\ldots N.
\ee
\begin{proposition}\label{prop:correspondence-finite-infinite-game}
Let Assumptions~\ref{assum:A-B-C-graph} and \ref{assum:A-B-C-graphon} be satisfied and assume that $\wt{\boldsymbol A}$ is nonnegative definite. Then, a set of strategies $\al^N=(\al^{1,N},\dots,\al^{N,N})\in\mathcal{A}^N$ is a Nash equilibrium of the $N$-player game as in Section~\ref{subsec:definition-finite-player-game-graph} if and only if the corresponding step function strategy profile $\al^{N}_{\operatorname{step}}$ defined in \eqref{eq:step-function} is a Nash equilibrium of a graphon game as in Section~\ref{subsec:definition-infinite-player-game} with underlying step graphon $W^N$ corresponding to $w^N$ as in \eqref{eq:step-graphon} and step function noise process $b^{N}_{\operatorname{step}}$ corresponding to $(b^{1,N},\ldots, b^{N,N})$ as in \eqref{eq:step-function}.
\end{proposition}
The proof of Proposition~\ref{prop:correspondence-finite-infinite-game} will be given in Section~\ref{sec:proofs-convergence}.

\subsection{Convergence results for given graph sequences}\label{subsection:convergence-general-games}
In order to show convergence of equilibria for given graph sequences, we first introduce the cut norm (see Chapter~8.2 in \citet{lovasz2012large}). Let $\mathcal{W}$ denote the linear space of all bounded symmetric Borel-measurable functions $W:[0,1]^2\to\R$. For a function $W\in\mathcal{W}$, define its cut norm by
\be\label{eq:cut-norm}
\|W\|_\Box \vcentcolon= \sup_{S_1,S_2\subset [0,1]}\left|\int_{S_1}\int_{S_2} W(u,v)dudv\right|,
\ee
where the supremum is taken over all Borel-measurable subsets $S_1,S_2$. If one identifies functions that are almost everywhere equal, the cut norm is indeed a norm. This technical consideration, however, is not of relevance for our purposes, since we are mainly working with integral operators induced by functions in $\mathcal{W}$, which are equal for almost everywhere agreeing functions. 

The cut norm can be compared with various other norms on $\mathcal{W}$, such as the usual $L^1$ norm of a function $W\in\mathcal{W}$ or the operator norm of the integral operator $L^\infty([0,1],\R)\to L^1([0,1],\R)$ induced by $W$ (see Chapter~8.2.4 in \citet{lovasz2012large}). We note here the following interesting relationship between the cut norm of a graphon $W$ and the operator norm of its induced integral operator $\boldsymbol{W}$ on $L^2([0,1],\R)$.
\begin{definition}\label{def:operator norm}
Let $F:V\to V'$ be a linear operator between two normed real vector spaces $(V,\|\cdot\|_V)$ and $(V,\|\cdot\|_{V'})$. Then the operator norm of $F$ is defined as
\be 
\|F\|_{\operatorname{op}}\vcentcolon = \sup\big\{\|F(x)\|_{V'} : x\in V\text{ with }\|x\|_V\leq 1\big\}.
\ee
\end{definition}
\begin{remark}\label{rem:relation-cut-op}
    Let $W:[0,1]^2\to [0,1]$ be a graphon. Then it holds that
    \be
    \|W\|_\Box\leq \|\boldsymbol{W}\|_{\operatorname{op}}\leq  \sqrt{8\|W\|_\Box},
    \ee
    where $\boldsymbol{W}:L^2([0,1],\R)\to L^2([0,1],\R)$ (see Lemmas E.2 and E.6 in \citet{janson2010graphons}).
\end{remark}
The seminal works of \citet{lovasz2006limits} and \citet{borgs2008convergent,borgs2012convergent} employed the cut norm and introduced the related cut distance to characterize the convergence of dense graph sequences to graphons. Given a sequence of graphs with adjacency matrices $(w^N)_{N\in\N}$, we say that they converge in cut norm to a graphon $W$ if and only if the corresponding step graphons $W^N$ defined in \eqref{eq:step-graphon} satisfy $\|W-W^N\|_\Box\to 0$.

We can now state the main theorem of this section.
\begin{theorem}\label{thm:convergence-given-graphs}
Let $(w^N)_{N\in\N}$ be a sequence of interaction matrices $w^N$ in $[0,1]^{N\times N}$ such that for every $N\in\N$ Assumption~\ref{assum:A-B-C-graph} holds and Assumption~\ref{assum:nonnegative-definite-graph} is satisfied for a constant $c_0>0$ independent of $N$. For $N\in\N$, denote by $\hat{\al}^N\in\mathcal{A}^N$ the unique Nash equilibrium of the corresponding $N$-player game (which exists by Corollary~\ref{cor:finite-player-equilibrium}). Let $W$ be a graphon such that Assumptions~\ref{assum:A-B-C-graphon} and \ref{assum:nonnegative-definite-graphon} hold, and denote by $\hat\al\in\mathcal{A}^\infty$ the corresponding unique Nash equilibrium (which exists by Theorem~\ref{thm:infinite-player-equilibrium}). For $N\in\N$, denote by $W^N$ the step graphon corresponding to $w^N$. Moreover, assume that that there exists a bounded function $\psi:[0,\infty)\to[0,\infty)$ such that
\be\label{eq:noise-convergence-given}
\E\Big[\int_0^1\int_0^T(b_t^u-b^{u,N}_{t,\operatorname{step}})^2dtdu\Big]^\frac{1}{2}\leq\psi(N)\xrightarrow{N\to\infty}0.
\ee
Then there exists a constant $K>0$ such that for all $N\in\N$ it holds that
\be\label{eq:convergence-theorem}
\E\Big[\int_0^1\int_0^T(\hat\al_t^u-\hat\al^{u,N}_{t,\operatorname{step}})^2dtdu\Big]\leq K\Big(\psi(N)+\sqrt{\|W-W^N\|_\Box}\Big).
\ee
In particular, if the sequence of graphs corresponding to $(w^N)_{N\in\N}$ converges in cut norm to $W$, i.e. $\|W-W^N\|_\Box\to 0$, the left-hand side of \eqref{eq:convergence-theorem} converges to $0$.
\end{theorem}
The proof of Theorem~\ref{thm:convergence-given-graphs} will be given in Section~\ref{sec:proofs-convergence}.
\begin{remark}
Theorem~\ref{thm:convergence-given-graphs} extends the convergence results of \citet{bayraktar2023propagation,carmona2022stochastic,cui2021learning,gao2020linear,lacker2023label,tangpi2024optimal} and others to the non-Markovian case in the linear-quadratic setting, without requiring any assumptions on the underlying graphon. Moreover, compared to previous work, it allows for significantly more general idiosyncratic and common noise, which is only assumed to be progressively measurable.
\end{remark}

\subsection{Convergence results for sampled graph sequences}\label{subsection:convergence-sampled-games}
Following \citet{avella2018centrality},  we first introduce four common sampling procedures according to which (random) graphs with $N$ nodes can be sampled from a graphon $W$. Also see Chapter~10 of \citet{lovasz2012large}. Recall that the partition $\mathcal P^N$ was defined in \eqref{partition}.
\begin{definition}
Given a graphon $W$ and $N\in\N$, fix the latent variables $(u_1,\ldots, u_N)$ by choosing either
\begin{itemize}
\item deterministic latent variables: $u_i=\frac{1}{N}$, $1\leq i\leq N$,
\item stochastic latent variables:  $u_i$ is the $i$-th order statistic of $N$ independent uniform random points from $[0,1]$.
\end{itemize} 
Using the latent variables $(u_1,\ldots, u_N)$, define
\begin{itemize}
    \item the probability matrix $w^N$ given by  $w_{ij}^N\vcentcolon= W(u_i,u_j)$ for all $1\leq i,j\leq N$,
    \item the sampled graphon $W^N(u,v)\vcentcolon= \sum_{i=1}^N\sum_{j=1}^N w_{ij}^N \mathds{1}_{\mathcal{P}_i^N}(u)\mathds{1}_{\mathcal{P}_j^N}(v)$.
\end{itemize}
Given the probability matrix $w^N$ of a sampled graphon, define 
\begin{itemize}
\item the sampled matrix $s^N$ as the adjacency matrix
of a random simple graph obtained by taking $N$ isolated nodes $i\in\{1,\ldots , N\}$ and adding undirected edges between nodes $i$ and $j$ at random with probability $\kappa_Nw_{ij}^N$ for all $i>j$, where $(\kappa_N)_{N\geq 1}\subset (0,1]$ is a sequence of density parameters,
\item the associated random $0$-$1$-graphon $S^N(u,v)\vcentcolon= \sum_{i=1}^N\sum_{j=1}^N s_{ij}^{N} \mathds{1}_{\mathcal{P}_i^N}(u)\mathds{1}_{\mathcal{P}_j^N}(v)$.
\end{itemize}
This yields the following graphs corresponding to the sampling procedures:
\begin{enumerate}[label=(S{{\arabic*}})]
    \item\label{S1} the weighted graph corresponding to $w^N$ and deterministic latent variables,
    \item\label{S2} the random weighted graph corresponding to $w^N$ and stochastic latent variables,
    \item\label{S3} the random simple graph corresponding to $s^N$ and deterministic latent variables,
    \item\label{S4} the random simple graph corresponding to $s^N$ and stochastic latent variables.
\end{enumerate}
\label{def:sampling}
\end{definition}
\begin{remark}
We will always assume that the sampling is carried out on another probability space $(\Omega',\mathcal{F}',\Q)$ fully independently of all the randomness in the $N$-player and the graphon game modeled by the probability measure $\P$.
\end{remark}
\begin{remark}
Notice that in Chapter~10 of \citet{lovasz2012large}, the sampled graphs corresponding to the four sampling procedures are defined very similarly to Definition~\ref{def:sampling}. The only difference is that in the definition of \ref{S1} and \ref{S3}, instead of choosing deterministic latent variables $u_i=\tfrac{1}{N}$, independent uniform points $u_i$ from $\mathcal{P}_i^N$ are chosen for all $i\in\{1,\ldots ,N\}$. This is done to ensure that the value $W(u_i,u_j)$ is representative for the value of $W$ on $\mathcal{P}_i^N\times \mathcal{P}_j^N$.  However, we will assume in Assumption~\ref{assum:Lipschitz-graphon} that $W$ is blockwise Lipschitz continuous, which allows the use of deterministic latent variables as in \citet{avella2018centrality}. 
\end{remark}
In Definition~\ref{def:sampling} the expected number of edges per node in $s^N$ grows as $\kappa_NN$ when $N\to\infty$. In line with \citet{parise2019graphon}, we will assume later that $\tfrac{\log N}{\kappa_NN}\to 0$, which allows for graph sequences that become sparser and sparser for large $N$, i.e. the expected number of edges per node in $s^N$ grows sublinearly with $N$. The introduction of a density parameter $\kappa_N$ only impacts how a sampled $N$-player game is obtained from the graphon, without affecting the graphon game limit as the number of players approaches infinity. As a consequence, one has to slightly adjust the definition \eqref{eq:aggregate} of the local aggregate in the sampled $N$-player game obtained from sampling procedures \ref{S3}-\ref{S4}, in order to account for
the fact that the number of edges per node may now grow sublinearly. This leads to a sampled $N$-player game as in \eqref{eq:J_0^iN-rewritten} with the only difference being a modification of \eqref{eq:aggregate} as follows,
\be\label{eq:aggregate-sparse}
\bar\al^{i,N}\vcentcolon=\frac{1}{\kappa_N N}\sum_{j=1}^N s_{ij}^N\al^{j,N},\quad i=1,\ldots,N.
\ee
\begin{remark}\label{rem:correspondence-with-kappa_N}
    Note that Proposition~\ref{prop:correspondence-finite-infinite-game} still holds in the modified $N$-player game with $w^N$ replaced by $\kappa_N^{-1}s^N$ and $W^N$ replaced by $\kappa_N^{-1}S^N$. This is crucial to show convergence of the sampled $N$-player game to the graphon game in cases \ref{S3}-\ref{S4}.
\end{remark}
We now restate Assumption 1 and Theorem~1 from \citet{avella2018centrality}, which yield a precise estimate for the convergence of the sampled graphon operators $\boldsymbol{W}^N$ and $\kappa_N^{-1}\boldsymbol{S}^N$ to the original graphon operator $\boldsymbol{W}$ under a blockwise Lipschitz assumption on the underlying graphon $W$.
\begin{assumption}\label{assum:Lipschitz-graphon}
Assume that there exists a constant $L_0$ and a sequence of non-overlapping intervals $I_k=[a_{k-1},a_k)$ defined by $0=a_0<\ldots<a_{K_0+1}=1$ for a $K_0\in\N$ such that for any $k,l$, any set $I_{kl}=I_k\times I_l$ and pairs $(u,v),(u',v')\in I_{kl}$ it holds that
$$
|W(u,v)-W(u',v')|\leq L_0(|u-u'|+|v-v'|).
$$
\end{assumption}
As noted in Section~5 of \citet{avella2018centrality}, Assumption~\ref{assum:Lipschitz-graphon} has been employed for
graphon estimation, and is generally satisfied for most graphons of interest, such as general smooth graphons or piecewise constant graphons (see Section~4 of \citet{avella2018centrality}). For what follows, recall Definition~\ref{def:operator norm} of the operator norm. 
\begin{theorem}[\citet{avella2018centrality}, Theorem~1]\label{thm:avella}
For a graphon $W$ fulfilling Assumption~\ref{assum:Lipschitz-graphon}, it holds with $\Q$-probability $1-\dl'$ that 
\be\label{eq:rho(N)}
\|\boldsymbol{W}-\boldsymbol{W}^N\|_{\operatorname{op}}\leq 2\sqrt{(L_0^2-K_0^2)d_N^2+K_0d_N}=:\rho (N),
\ee
where $\dl'=0$ and $\smash{d_N=\tfrac{1}{N}}$ in the case of deterministic latent variables and $\dl'=\dl\in (Ne^{-N/5},e^{-1})$ and $d_N=\tfrac{1}{N}+(\tfrac{8\log (N/\dl)}{N+1})^{0.5}$ in the case of stochastic latent variables. Moreover, for sufficiently large $N$, it holds  with $\Q$-probability at least $1-\dl-\dl'$ that 
\be\label{eq:rho'(N)}
\|\boldsymbol{W}-\kappa_N^{-1}\boldsymbol{S}^N\|_{\operatorname{op}}\leq \sqrt{\frac{4\kappa_N^{-1}\log (2N/\dl)}{N}}+\rho (N)=:\rho'(N).
\ee
\end{theorem} 

Keeping in mind the notation of Theorem~\ref{thm:avella}, we can now state a the main theorem of this section, which establishes the convergence of the sampled finite-player games on graphs to the infinite-player graphon game in the following setup. 
\paragraph{Setup.}
Let Assumptions~\ref{assum:A-B-C-graph} and \ref{assum:A-B-C-graphon} hold and assume that there exists a constant $c_0>0$ independent of $N$ such that Assumption~\ref{assum:nonnegative-definite-graph} is satisfied for every matrix $w^N\in[0,1]^{N\times N}$ with zero diagonal entries and every $N\in\N$. Let Assumptions~\ref{assum:nonnegative-definite-graphon} and \ref{assum:Lipschitz-graphon} hold, and denote by $\hat\al\in\mathcal{A}^\infty$ the corresponding unique Nash equilibrium (which exists by Theorem~\ref{thm:infinite-player-equilibrium}). Fix one of the sampling procedures from Definition~\ref{def:sampling}. Then, in cases \ref{S1}-\ref{S2}, for every $N\in\N$, the sampled $N$-player game admits a unique Nash equilibrium $\hat\al^N\in\mathcal{A}^N$ corresponding to $w^N$ (by Corollary~\ref{cor:finite-player-equilibrium}). In cases \ref{S3}-\ref{S4}, for any $0<\dl<e^{-1}$, there exists an $N_\dl\in\N$ such that for all $N\geq N_\dl$ the sampled $N$-player game admits a unique Nash equilibrium $\hat\beta^N\in\mathcal{A}^N$ corresponding to  $\kappa_N^{-1} s^N$ with $\Q$-probability at least $1-\dl$ and $1-2\dl$, respectively (by Corollary~\ref{cor:finite-player-equilibrium} and the auxiliary Lemma~\ref{lemma:aux}).

\begin{lemma}\label{lemma:aux}
In the above setup, there is for any $0<\dl<e^{-1}$ an $N_\dl\in\N$ such that for all $N\geq N_\dl$ it holds that
\be \label{eq:aux-lemma}
\langle f, \lambda f +\wt{\boldsymbol{A}}f+ \frac{1}{N}\kappa_N^{-1}s^N\cdot\wt{\boldsymbol{B}}f\rangle_{L^2,N}\geq \frac{c_0}{2} \langle f, f\rangle_{L^2,N},\quad \text{for all }f\in L^2([0,T],\R^N),
\ee 
with $\Q$-probability at least $1-\dl$ in case \ref{S3} and $1-2\dl$ in case \ref{S4}.
\end{lemma}

\begin{theorem}\label{thm:convergence-sampled-graphs}
In the above setup, assume that $\tfrac{\log N}{\kappa_NN}\to 0$ as $N\to\infty$. Moreover, assume that there exists a bounded function $\pi:[0,\infty)\to[0,\infty)$ such that
\be\label{eq:noise-convergence}
\E\Big[\int_0^1\int_0^T(b_t^u-b^{u,N}_{t,\operatorname{step}})^2dtdu\Big]^\frac{1}{2}\leq\pi(N) \quad \textrm{for all } N\geq1, 
\ee
where $\pi(N)\rightarrow 0$ as $N\to\infty$. 
Then the following holds for the four different sampling procedures in Definition~\ref{def:sampling}. 
\begin{enumerate}
\item For deterministic latent variables and weighted sampled graphs \ref{S1}, there exists a constant $K_1>0$ such that  
\be\label{eq:convergence-theorem-1}
\E\Big[\int_0^1\int_0^T(\hat\al_t^u-\hat\al^{u,N}_{t,\operatorname{step}})^2dtdu\Big]\leq K_1\big(\pi(N)+\rho(N)\big)\quad \textrm{for all } N \in \mathbb{N}. 
\ee
In particular, the left-hand side of \eqref{eq:convergence-theorem-1} converges to $0$ as $N\to\infty$.
\item For stochastic latent variables and weighted sampled graphs \ref{S2}, there exists a constant $K_2 >0$ such that for every $0<\dl<e^{-1}$ it holds for all $N\in\N$ satisfying $Ne^{-N/5}<\dl$ with $\Q$-probability at least $1-\dl$ that
\be\label{eq:convergence-theorem-2}
\E\Big[\int_0^1\int_0^T(\hat\al_t^u-\hat\al^{u,N}_{t,\operatorname{step}})^2dtdu\Big]\leq K_2\big(\pi(N)+\rho(N)\big).
\ee
In particular, the left-hand side of \eqref{eq:convergence-theorem-2} converges $\Q$-a.s. to $0$ as $N\to\infty$.
\item For deterministic latent variables and simple sampled graphs \ref{S3}, there exists a constant $K_3>0$ such that for every $0<\dl<e^{-1}$ it holds for all $N\geq N_\dl$ with $\Q$-probability at least $1-\dl$ that
\be\label{eq:convergence-theorem-3}
\E\Big[\int_0^1\int_0^T(\hat\al_t^u-\hat\beta^{u,N}_{t,\operatorname{step}})^2dtdu\Big]\leq K_3\big(\pi(N)+\rho'(N)\big).
\ee
In particular, the left-hand side of \eqref{eq:convergence-theorem-3} converges $\Q$-a.s. to $0$ as $N\to\infty$.
\item For stochastic latent variables and simple sampled graphs \ref{S4}, there exists a constant $K_4>0$ such that for every $0<\dl<e^{-1}$ it holds for all $N\geq N_\dl$ with $\Q$-probability at least $1-2\dl$ that
\be\label{eq:convergence-theorem-4}
\E\Big[\int_0^1\int_0^T(\hat\al_t^u-\hat\beta^{u,N}_{t,\operatorname{step}})^2dtdu\Big]\leq K_4\big(\pi(N)+\rho'(N)\big). 
\ee
In particular, the left-hand side of \eqref{eq:convergence-theorem-4} converges $\Q$-a.s. to $0$ as $N\to\infty$.
\end{enumerate}
\end{theorem}
Theorem~\ref{thm:convergence-sampled-graphs} and the auxiliary Lemma~\ref{lemma:aux} will be proved in Section~\ref{sec:proofs-convergence}.
\begin{remark}Theorem~\ref{thm:convergence-sampled-graphs} extends the convergence results of \citet{aurell2022stochastic,carmona2022stochastic} and \citet{parise2019graphon}, who consider graphs sampled according to sampling procedures \ref{S2} and \ref{S4}, to the non-Markovian case, while assuming a linear-quadratic setting and blockwise Lipschitz continuity of the underlying graphon. Moreover,  under the same assumptions, it also addresses the convergence of network games sampled according to procedures \ref{S1} and \ref{S3}, which have not been addressed in the literature on dynamic stochastic graphon games even in the Markovian case.
\end{remark}
\begin{remark}
Notice that cases \ref{S3}-\ref{S4} in Theorem~\ref{thm:convergence-sampled-graphs} allow for density parameter sequences $(\kappa_N)_{N\geq 1}$ that model significantly sparser regimes than the ones addressed in the literature on dynamic stochastic graphon games so far, where density parameters have yet only been considered by \citet{tangpi2024optimal}, who employ the more restrictive assumption that $\kappa_N^2 N\to \infty$ as $N\to\infty$. For instance, Theorem~\ref{thm:convergence-sampled-graphs} permits sequences such as $\kappa_N=(\log N)^2/N$ corresponding to sampled simple graphs whose expected number of edges per node grows as $(\log N)^2$ when $N\to\infty$. 
\end{remark}
\section{Illustrative Examples} \label{sec:examples}
In this section, we showcase the versatility and applicability of our framework from Section~\ref{sec:finite-player-game} by illustrating how it can be applied to solve various dynamic stochastic games with heterogeneous interactions, in particular in cases which have been considered intractable in the literature before. 
In order to ease notation we fix the number of agents $N$, and omit the superscript $N$ from $\al^{i,N}, \al^{-i,N}$, $\al^N$ and $J^{i,N}$ defined in \eqref{eq:notation-alpha} and \eqref{eq:J^iN} throughout this section.

\subsection{Stochastic Volterra games with heterogeneous interactions} \label{subsec:Volterra-game}
We start by showing that generic stochastic Volterra linear-quadratic games with heterogeneous interactions are included in the finite-player model from Section~\ref{subsec:definition-finite-player-game}. 

Define the following controlled Volterra state variables,
\begin{align}
X^i_t &= P^i_t + \sum_{j=1}^N\int_0^t G_1^{ij} (t,s)\al^j_s ds, \quad i=1,\ldots, N, \label{eq:volterraXi} \\
Z^i_t &= R^i_t + \sum_{j=1}^N\int_0^t G_2^{ij} (t,s)\al^j_s ds, \quad i=1,\ldots,N, \label{eq:volterraZi}
\end{align}
for some progressively measurable processes $P^i,R^i$ satisfying $\smash{\int_0^T \E[(P^i_t)^2] dt<\infty}$ and $\smash{\int_0^T \E[(R^i_t)^2] dt<\infty}$ and admissible Volterra kernels $G_1^{ij},G_2^{ij}\in \mathcal{G}$ (recall Definition~\ref{def:admissible-volterra-operator}). 
Then the $\R^2$-valued state variable given by $Y^i\vcentcolon= (X^i, Z^i)^\top$ captures the dynamics \eqref{eq:volterraXi} and \eqref{eq:volterraZi} and can be written as 
\be\label{eq:volterraYi}
    Y^i_t = d^i_t + \int_0^t G^i(t,s) \al_sds,
\ee
with 
\be\label{eq:volterraGi}
d^i_t = \begin{pmatrix}
    P^i_t \\
    R^i_t 
\end{pmatrix} \quad \mbox{and} \quad
G^i(t,s) =\begin{pmatrix}
G^{i1}_1(t,s) & \dots & G^{iN}_1(t,s)\\
G^{i1}_2(t,s) & \dots & G^{iN}_2(t,s) 
\end{pmatrix}.
\ee
Consider the following objective functional to be maximized by the $i$-th player,  
\be\label{eq:volterraJ^i}
     J^i_{Vol}(\al^i; \al^{-i}) =\mathbb E\left[ \int_0^T f_{Vol}^{i}(Y^i_t , \al^i_t) dt + g_{Vol}^i(Y^i_T) \right],\quad i=1,\dots,N,
\ee
where the running and terminal criterion have a linear-quadratic dependence in $(Y^i,\al^i)$ of the form
\begin{align}
f_{Vol}^{i}( y, \al) &= - p^i \al^2 - y^\top Q^i y +  \al (q^i)^\top y,  \label{eq:volterraf}  \\
g_{Vol}^{i}( y ) &= - y^\top S^i y + y^\top s^i,\label{eq:volterrag}
\end{align}
such that $p^i>0$, $Q^i,S^i \in \mathbb R^{2\times 2}$,  $q^i \in \R^2$ and $s^i$ are square-integrable $\mathcal F_T$-measurable $\R^2$-valued random variables  for all $i=1,\ldots, N$. 

Since the dynamics of each $Y^i$ in \eqref{eq:volterraYi} are linear in the controls $(\al^1,\ldots,\al^N)$, and $f_{Vol}^i$ and $g^i_{Vol}$ in \eqref{eq:volterraf} and \eqref{eq:volterrag} are linear-quadratic in $(Y^i,\al^i)$, it is clear that $ J^i_{Vol}$ is a linear-quadratic functional  in $(\al^1,\ldots,\al^N)$ as in \eqref{eq:J^iN}. This is summarized in the following lemma, which shows that the generic Volterra game with dynamics \eqref{eq:volterraXi} and \eqref{eq:volterraZi} fits into our framework developed in Section~\ref{subsec:definition-finite-player-game}, and also yields sufficient conditions on the Volterra game coefficients introduced in \eqref{eq:volterraGi}, \eqref{eq:volterraf} and \eqref{eq:volterrag} so that Assumption~\ref{assum:nonnegative-definite} holds. Let $(\vec{e_i})_{i=1}^N$ denote the standard basis of $\R^N$. Recalling \eqref{eq:volterraGi}, define the $\R^2$-valued Volterra kernel 
\be G^{ij}\vcentcolon=(G_1^{ij},G_2^{ij})^\top,\quad \text{for all }i,j=1,\ldots, N.\ee
\begin{lemma} \label{lemma:Volterra-game}     
The objective functionals \eqref{eq:volterraJ^i} for the stochastic Volterra game can be written in the form of \eqref{eq:J^iN-concise-form} with the following coefficients for $i=1,\dots,N$:
\begin{align}
 F^i(t,s) \label{eq:volterraFi}
 & =  \mathds{1}_{\{t>s\}} \int_{s\vee t}^T G^i(r,t)^\top \big(Q^i+(Q^i)^\top\big)\,  G^i(r,s) dr \\[1ex]
& \hspace{13pt}+ \mathds{1}_{\{t>s\}}\,G^i(T,t)^\top \big(S^i+(S^i)^\top\big)\, G^i(T,s) - \vec{e_i}(q^i)^\top G^i(t,s),\\[1ex]
    b^i_t  \label{eq:volterrabi}   &= -   \int_t^T G^i(r,t)^\top  \big( Q^i + (Q^i)^\top\big)   \E_t [d^i_r]  dr   
 + \vec{e_i}(d^i_t)^\top q^i    \\[1ex]
&\hspace{13pt}+G^i(T,t)^\top \Big( \E_t [s^i] -  \big( S^i + (S^i)^\top\big)  \E_t [d^i_T]   \Big), \\[1ex]
    c^i\label{eq:volterraci}& =   - \int_0^T (d_t^i)^\top   Q^i d_t^idt  -(d_T^i)^\top S^i  d_T^i   +   (d^i_T)^\top s^i , \\[1ex]
     \lambda^i\label{eq:volterralambdai} &= p^i. 
 \end{align}
If for all $i=1,\ldots,N$ there exists a $c_0>0$ such that
\be\label{eq:Pi-nonnegative-definite}
\langle f,  \boldsymbol{\Gamma}f+2\Pi f -c_0 f\rangle_{L^2,N}\geq 0,\quad \text{for all } f\in L^2([0,T],\R^N),
\ee
where $\Pi\vcentcolon =\operatorname{diag}(p^1,\ldots,p^N)\in\R^{N\times N}$ and $\boldsymbol{\Gamma}$ is the integral operator induced by the kernel $\Gamma\in L^2([0,T]^2,\R^{N\times N})$ given by 
\be\begin{aligned}
\Gamma_{ij}(t,s)&=\int_{s\vee t}^T G^{ii}(r,t)^\top \big(Q^i+(Q^i)^\top\big)\,  G^{ij}(r,s) dr + G^{ii}(T,t)^\top \big(S^i+(S^i)^\top\big)\, G^{ij}(T,s)  \\[1ex]
& - (q^i)^\top G^{ij}(t,s)-\dl_{ij}(q^i)^\top G^{ii}(s,t),\quad\text{for }(t,s)\in[0,T]^2,\quad i,j=1,\dots, N,
\end{aligned}\ee
then Assumption~\ref{assum:nonnegative-definite} is satisfied.
\end{lemma}

\begin{remark}
Note that by definition of the Volterra game, the processes $b^i$ defined in \eqref{eq:volterrabi} and random variables $c^i$ defined in \eqref{eq:volterraci} satisfy \eqref{eq:assumption-b-c}. Moreover, since all entries $F^i_{jk}$ of $F^i$ defined in \eqref{eq:volterraFi} are Volterra kernels in $\mathcal{G}$ and all $\lam^i=p^i$ defined in \eqref{eq:volterralambdai} are positive, Assumption~\ref{assum:A-B-C} is satisfied as well.  In particular, whenever the objective functionals in \eqref{eq:volterraJ^i} are strictly concave, the characterization of the Nash equilibrium via the system of stochastic Fredholm equations from Proposition~\ref{prop:fredholm-foc} holds. If additionally, condition \eqref{eq:Pi-nonnegative-definite} holds, the Volterra game can be solved explicitly by means of Theorem~\ref{thm:finite-player-equilibrium}.
\end{remark}

\begin{proof}[Proof of Lemma~\ref{lemma:Volterra-game}]
The first part of the lemma follows from a a straightforward application of Fubini's theorem and the tower property of conditional expectation.
The second part follows from recalling Remark \ref{rem:J^iN-concise-form}, reading off the entries of $B$ and $\bar B$ in \eqref{eq:volterraFi}, and defining $\Gamma$ to be the sum of $B$ and $\bar{B}^*$.
\end{proof}

Lemma~\ref{lemma:Volterra-game} demonstrates that the Volterra game fits into our framework developed in Section~\ref{subsec:definition-finite-player-game} and can be solved through Theorem~\ref{thm:finite-player-equilibrium}. However, there are various classes of network games where the state processes $X^i$ and $Z^i$ are not given explicitly as in \eqref{eq:volterraXi} and \eqref{eq:volterraZi}, but rather defined in differential form with drift terms involving $X^i$ and $Z^i$ again.
Therefore, before addressing more concrete examples in the next subsections, we first demonstrate that the dynamics \eqref{eq:volterraXi}, \eqref{eq:volterraZi} include any stochastic Volterra equation for the state variables $X^i$, where the drift has linear dependence in $X^i$ itself and in a weighted average $Z^i$ of the other states. For these purposes we introduce an interaction matrix $w\in\R^{N\times N}_+$ whose entries $w_{ij}\geq 0$ function as weights quantifying how much player $i$ takes into account the state of player $j$. In contrast to the matrices $w^N$ from Sections \ref{sec:infinite-player-game} and \ref{sec:convergence}, we will not always scale $w$ by $1/N$. Namely, in what follows, we assume the state variables $Z^i$ in \eqref{eq:volterraZi} to be given by
\be
Z^i\vcentcolon = \sum_{j=1}^N w_{ij}X^j,\quad \text{for }i=1,\ldots, N.
\ee
In order to show the correspondence of the games, we will employ resolvents of Volterra kernels, see Chapter~9.3 of \citet{gripenberg1990volterra} for a thorough treatment.
\begin{lemma}\label{lemma:volterra-linear}
Let $\smash{w\in\R^{N\times N}_+}$. For $i=1,\ldots, N$, let  $M^i$ be a progressively measurable processes satisfying $\int_0^T \E[(M^i_t)^2] dt  <\infty$ and  $H^i,K^i:[0,T]^2 \to \mathbb R$ be Volterra kernels in $\mathcal{G}$. Then the $N$-dimensional linear system of Volterra equations 
\be
\begin{aligned}\label{eq:volterra-linear}
        &X^i_t  = M^i_t + \int_0^t H^i(t,s) X^i_s ds + \int_0^t K^i(t,s) Z^i_s ds, \\
        &Z^i_t=  \sum_{j=1}^N w_{ij}X^j_t,
        \quad i=1,\ldots, N,\ t\in[0,T],
\end{aligned}  
\ee
admits a unique solution given by 
\be\label{eq:volterra-solution}
X_t =  M_t - \int_0^t R^{-(H+Kw)}(t,s)  M_sds,
\ee
where $X\vcentcolon = (X^1,\ldots,X^N)^\top$, $M\vcentcolon = (M^1,\ldots,M^N)^\top$, $H\vcentcolon=\operatorname{diag}(H^1,\ldots,H^N)$, $K\vcentcolon=\operatorname{diag}(K^1,\ldots,K^N)$, and $R^{-(H+Kw)}$ denotes the resolvent of $-(H+Kw)$.
\end{lemma}
\begin{proof}
Notice that the system \eqref{eq:volterra-linear} can be rewritten as follows,
\be\label{eq:volterra-rewritten}
M_t=X_t-\int_0^t \big(H(t,s)+K(t,s)w\big)X_sds.
\ee
Now since $H^i,K^i\in\mathcal{G}$ for all $i=1,\dots,N$, it follows that the matrix-valued Volterra kernel $-(H+Kw)$ admits a Volterra resolvent $R^{-(H+Kw)}$ (see Corollary 9.3.16 in \citet{gripenberg1990volterra}). Due to $\smash{\int_0^T \E[(M^i_t)^2] dt  <\infty}$ for all $i$,  \eqref{eq:volterra-rewritten} admits a unique solution given by \eqref{eq:volterra-solution} (see Theorem~9.3.6 in \citet{gripenberg1990volterra}).
\end{proof}

\begin{remark}\label{rem:volterra-linear}
A particular case of interest is the application of Lemma~\ref{lemma:volterra-linear} to controlled Volterra processes with $M^i$ in \eqref{eq:volterra-linear} given by
\be \label{eq:volterra-control-remark}
M^i_t =  N^i_t+ \sum_{j=1}^N\int_0^t L^{ij} (t,s)\al^j_s ds, \quad i=1,\ldots, N. 
\ee  
Here $N^i$ are progressively measurable processes satisfying $\int_0^T \E[(N^i_t)^2] dt  <\infty$, $L^{ij}$ are Volterra kernels in $\mathcal G$  and $\al=(\al^1,\dots,\al^N)$ are the controls. Plugging \eqref{eq:volterra-control-remark} into \eqref{eq:volterra-solution} shows that the dynamics \eqref{eq:volterra-linear} can be represented as, 
\be
X_t =  N_t -\int_0^tR^{-(H+Kw)}(t,s)  N_sds + \int_0^t \Big(L(t,s)-\big(R^{-(H+Kw)}\star L\big)(t,s)\Big)\al_sds. 
\ee
Here for any $G, H \in  L^2\left([0,T]^2,\mathbb R^{N\times N}\right)$ the $\star$-product is given by,  
\be
(G \star H)(t,s) = \int_0^T G(t,u) H(u,s)du, \quad  (t,s) \in [0,T]^2,
\ee
which is a well-defined kernel in $L^2([0,T]^2,\mathbb R^{N\times N})$ again (see Definition~9.2.3 and Theorem~9.2.4 in \citet{gripenberg1990volterra}).
This shows that the dynamics \eqref{eq:volterraXi}, \eqref{eq:volterraZi} include controlled stochastic Volterra equations for the state variables $X^i$, where the drift has linear dependence in $X^i$ itself and the weighted average $Z^i$ of the other states.
\end{remark}

\subsection{Heterogeneous inter-bank lending and borrowing models} \label{subsec:systemic-risk}
In their seminal paper \cite{carmona2015mean}, Carmona~et~al. introduced an inter-bank borrowing and lending model in which the log-monetary reserves of $N$ banks are represented by a system of controlled diffusion processes that are coupled in a homogeneous way, and each bank has as control its rate of borrowing/lending to a central bank. In \cite{carmona2018systemic}, Carmona~et~al. extended their framework to a more realistic model which additionally allows delay in the controls and thereby the clearance of debt obligations, compelling the banks to account for their prior lending and borrowing activities. In \cite{sun2022mean}, Sun studied another extension of \citet{carmona2015mean} to a setting with heterogeneity among (groups of) banks that have relative concerns. Clearly, this feature is crucial as well to model inter-bank borrowing and lending as accurately as possible.

In this subsection we show how our framework accommodates the models on systemic risk studied in \citet{carmona2015mean,carmona2018systemic} and \citet{sun2022mean} and extends them into various directions. Most importantly, we allow for both delay in the controls and heterogeneity of the banks at the same time and provide explicit operator formulas for the associated Nash equilibrium.


First, we present the following inter-bank lending and borrowing model which unifies and generalizes the models in the aforementioned references. The log-monetary reserves $X^i$ of $N\in\N$ banks are modeled as follows,
\be\label{eq:systemicXi}
dX^i_t = \left( \int_0^t \al^i_{t-s} \nu^i(ds)  +h^i(t) \right) dt + dV^i_t,\quad X^i_0=\xi^i,\quad i=1,\ldots, N,
\ee
with progressively measurable processes $h^i$ and $V^i$, signed measures $\nu^i$ on $[0,T]$ of locally bounded variation, square-integrable random variables $\xi^i$, and $\al^i_t=0$ for $t<0$. The control $\al^i$ of the $i$-th bank corresponds to the rate of lending or borrowing from a central bank. The delay in the control reflects the repayments after a fixed time. The canonical example is given by $\al^i_t - \al^i_{t-\tau }$ which corresponds to the case $\nu^i = \dl_0 - \dl_{\tau}$ for some fixed time $\tau \geq 0$. In this case, if the $i$-th bank borrowed from the central bank an amount $\al_t^idt$ at some time $t$, then $X^i_t$ increases by $\al_t^i dt$. In addition, since the amount needs  to be paid back to the central bank at a later time $t + \tau $, the log-monetary reserve  $X^i_{t+\tau}$  decreases by $\al^i_{(t+\tau) - \tau} dt=\al^i_t dt$, which explains the form of the drift $\al^i_t - \al^i_{t-\tau}$ at time $t$. 

In order to incorporate heterogeneous banks with relative concerns, consider an interaction matrix $w^{sys}\in\R^{N\times N}_+$ with nonnegative entries $w^{sys}_{ij}\geq 0$ and the individual weighted averages acting as benchmarks given by
\be \label{eq:systemicZi}
Z^i \vcentcolon= \sum_{i=1}^N w^{sys}_{ij} X^j,\quad i=1,\ldots,N.
\ee
For $i=1,\ldots,N$, the $i$-th bank chooses the strategy $\al^i$ to minimize the following cost functional involving $Z^i$, 
\be \label{eq:systemicJ^i}
J^i_{sys}(\al^i; \al^{-i}) =\mathbb E\left[ \int_0^T f^i_{sys}(X^i_t, Z^i_t , \al^i_t) dt + g^i_{sys}(X^i_T, Z^i_T) \right],
\ee
with the running and terminal costs given by
\be\begin{aligned}
f^i_{sys}(X^i, Z^i , \al^i) &= \frac{1}{2}(\al^i)^2 - \kappa^i \al^i(Z^i - X^i) + \frac{\varepsilon^i}{2}(Z^i -X^i )^2, \\
g^i_{sys}(X^i, Z^i) &= \frac{c^i}{2}(Z^i - X^i)^2.
\end{aligned}\ee
The running cost of borrowing/lending is given by $(\al^i)^2/2$, the parameter $\kappa^i\geq 0$ controls the incentive to borrow or lend depending on the difference with the weighted capitalization level of the others, the quadratic term in $(Z^i - X^i)^2$ in the running and terminal costs penalize deviation from the weighted average with $\eps^i>0$, $c^i\geq 0$. The condition  $(\kappa^i)^2< \eps^i$ ensures the strict convexity of the functions $f^i_{sys}$.

\noindent \textbf{The correspondence with our Volterra game framework.} Setting 
\be\label{eq:systemicGii}
G_1^{ii}(t,s):= \mathds{1}_{\{t>s\}}\nu^i([0,t-s]), \quad  0\leq t,s\leq T,\quad i=1,\ldots,N\ee
writing \eqref{eq:systemicXi} in integral form, and applying Lemma~3.10 from \citet{abijaber2023equilibrium},  we obtain the following Volterra representation for $X^i$ in \eqref{eq:systemicXi}, 
\be\label{eq:systemic-representation}
   X_t^i =  \xi^i + \int_0^{t} h^i(s)ds + V_t^i +  \int_0^t G_1^{ii}(t,s) \al^i_s ds,\quad i=1,\ldots,N.
\ee

\begin{tcolorbox}
The Volterra representation \eqref{eq:systemic-representation} demonstrates that the model corresponds to the stochastic Volterra game from Section~\ref{subsec:Volterra-game} where the controlled variables $X^i$  in \eqref{eq:volterraXi} and $Z^i$ in \eqref{eq:volterraZi} have coefficients 
\be\begin{aligned}
P_t^i &= \xi^i + \int_0^{t} h^i(s)ds + V_t^i,\quad &&G_1^{ij}(t,s)=\begin{cases}
\mathds{1}_{\{t>s\}}\nu^i([0,t-s]),\quad\text{for }i=j,
\\0,\hspace{3.37cm}\text{for }i\neq j,
\end{cases} \vspace{-1cm}\\
 R_t^i&=\sum_{j=1}^Nw^{sys}_{ij}P_t^j,\quad && G_2^{ij}(t,s)= w^{sys}_{ij}G_1^{jj}(t,s). 
\end{aligned}\ee
The minimization of the cost functional in \eqref{eq:systemicJ^i} corresponds to maximizing the objective functionals in \eqref{eq:volterraJ^i} given in terms of $f^i_{Vol}$ as in \eqref{eq:volterraf} and $g^i_{Vol}$ as in \eqref{eq:volterrag} with parameters for $i=1,\dots, N$ given by
\be
    p^i = \frac 1 2,  \quad 
    Q^i =  \frac{\eps^i}{2} \begin{pmatrix}
        1  & -1  \\
        -1   & 1
     \end{pmatrix},  \quad 
    q^i = \kappa^i\begin{pmatrix}
        -1  \\
        1
     \end{pmatrix},  \quad 
    S^i =   \frac{c^i}{\eps^i} Q^i, \quad 
    s^i   = 0. 
\ee
\end{tcolorbox}

In the following remarks we explain how our general approach unifies and extends existing systemic risk models that appear in the literature.   
\begin{remark}\label{rem:systemic-correspondence}
Notice that setting $w^{sys}$ to be the matrix with $1/N$ in all entries, choosing for all $i$ the state processes and the cost functionals to be the same with $h^i=0$ and the noise processes $V^i$ to be combinations of idiosyncratic and common Brownian noises, and defining for all~$i$ the measure $\nu^i(ds)=\mathds{1}_{\{s\leq\tau\}}\theta(ds)$, recovers the mean-field framework from \citet{carmona2018systemic}, where $\theta$ denotes the signed measure on $[0,\tau]$ therein. If additionally, the measures $\nu^i$ are chosen as $\dl_0$ for all $i$, the original model from \citet{carmona2015mean} without delay in the controls is obtained. Note that in \citet{carmona2015mean}, an additional interacting
term $\smash{\frac{1}{N}\sum_{j=1}^N(X_t^j-X_t^i)}$ in the drift of $X^i$ in \eqref{eq:systemicXi} representing the borrowing/lending rate of bank $i$ from/to bank $j$ was considered, which can be incorporated into our model as well via an application of Lemma~\ref{lemma:volterra-linear} and Remark \ref{rem:volterra-linear} with $w=w^{sys}$. Our general approach does not only yield explicit operator formulas for the Nash equilibria in the frameworks of \citet{carmona2015mean,carmona2018systemic}, but more importantly allows the network of banks to be represented by a weighted (and possibly directed) graph with adjacency matrix $w^{sys}$. 
\end{remark}

\begin{remark}
We can also recover the model from \citet{sun2022mean}, where $d\in\N$ groups of $N_k\in\N$ homogeneous banks with log-monetary reserves $X^{(k)i}$ for $i=1,\ldots,N_k$ and $k=1,\ldots,d$ were considered. This is done by setting $\nu^i=\dl_0$ for all $i$ again, defining the processes $V^i$ to be combinations of idiosyncratic Brownian noises and common Brownian noises within and among the groups of banks, and finally choosing the weights in the matrix $w^{sys}$ appropriately to turn the state process $Z^{(k)i}$ associated with bank $i\in\{1,\ldots,N_k\}$ belonging to group $k\in\{1,\ldots ,d\}$ into a weighted sum of the group capitalization average $\smash{\frac{1}{N_k}\sum_{j=1}^{N_k} X^{(k)j}}$ and the global capitalization average $\smash{\frac{1}{N}\sum_{k=1}^d\sum_{j=1}^{N_k} X^{(k)j}}$. In \citet{sun2022mean}, the Nash equilibrium in the case of two homogeneous groups of banks was derived semiexplicitly in terms of coupled Riccati equations.
Our main results do not only provide novel explicit operator formulas for the Nash equilibrium in the case of $d\in\N$ groups, but also allow with no additional effort several non-trivial, realistic extensions. That is, compared to \citet{sun2022mean}, we additionally allow for more general heterogeneity as expounded above, for controls with delay in the sense of \citet{carmona2018systemic}, for an interacting term in the drift of $X^i$ in \eqref{eq:systemicXi} by using Lemma~\ref{lemma:volterra-linear} and Remark \ref{rem:volterra-linear}, and for more general combinations of idiosyncratic and common noises given by a progressively measurable process.
\end{remark}
In Example \ref{ex:systemic} of Appendix \ref{appendix}, we study in greater detail the case of one repayment after a fixed time $0\leq\tau\leq T$ and verify Assumption~\ref{assum:nonnegative-definite} under various conditions on the interaction matrix $w^{sys}$ and the model's coefficients $\kappa^i,\eps^i,c^i$. 
\subsection{Stochastic differential network games}
We show that the Volterra game framework introduced in Section~\ref{subsec:Volterra-game} accommodates linear-quadratic stochastic differential network games as studied by \citet{aurell2022stochastic}, who consider the following $N$-player game. For $i=1,\dots,N$, player $i$ chooses a control $\al^i\in\mathcal{A}$ in order to minimize the objective functional
\be\label{eq:networkJ^i}
      J_{net}^i(\al^i; \al^{-i}) \vcentcolon=\E\left[ \int_0^T  f_{net}^{i}(X^i_t , \al^i_t, Z^i_t) dt + g_{net}^i(X^i_T, Z^i_T) \right],\quad i=1,\dots,N,\hspace{5mm}
\ee
where the player states $(X^i)_{i=1}^N$ are subject to the dynamics
\be
\begin{aligned}
dX_t^i&=\big( a_{net}(i)X_t^i+ b_{net}(i)\al_t^i+ c_{net}(i)Z_t^i\big)dt+d\bar B_t^i,\quad X_0^i=\xi^i\label{eq:networkXiZi}\\
 Z_t^i&=\frac{1}{N}\sum_{j=1}^N {w}^{net}_{ij}X_t^j,\quad  i=1,\dots, N,\ \ t\in[0,T],
\end{aligned}
\ee
with coefficients $a_{net}(i), b_{net}(i), c_{net}(i)\in\R$, independent Brownian motions $\bar B_t^i$ and independent  random variables $ \xi^i$  for $i=1,\dots, N$, as well as a symmetric weight matrix $w^{net}\in[0,1]^{N\times N}$. The functions $ f_{net}^{i}$ and $g_{net}^{i}$ are given by 
\be
f_{net}^{i}=\frac{1}{2}x^\top C^{net}_f(i) x,\quad g_{net}^{i}=\frac{1}{2}y^\top  C^{net}_g(i) y,\quad\text{for } x\in\R^3,y\in\R^2,
\ee
where $ C^{net}_f(\cdot)\in\R^{3\times 3}$ and $ C^{net}_g(\cdot)\in\R^{2\times 2}$ are symmetric matrices.

\noindent \textbf{The correspondence with our Volterra game framework.}
First, using the notation \eqref{eq:notation-alpha}, notice that the dynamics \eqref{eq:networkXiZi} can be summarized by the following linear system of SDEs,
\be\label{eq:networkX}
dX_t=\left( {\boldsymbol{a}} X_t+{\boldsymbol{b}}\al_t\right) dt+d\bar B_t,\quad X_0= \xi,
\ee
where
\be
\begin{aligned}\label{eq:network-matrices}
{\boldsymbol{a}}&\vcentcolon = \operatorname{diag}\big( a_{net}(1),\ldots, a_{net}(N)\big)+\frac{1}{N}\operatorname{diag}\big( c_{net}(1),\ldots,c_{net}(N)\big)w^{net},\\
{\boldsymbol{b}}&\vcentcolon = \operatorname{diag}\big( b_{net}(1),\ldots, b_{net}(N)\big),\quad \bar B\vcentcolon=\big(\bar B^1,\ldots, \bar B^N\big)^\top,\quad \xi\vcentcolon=\big( \xi^1,\ldots,  \xi^N\big)^\top.
\end{aligned}
\ee
Using \eqref{eq:networkX} and \eqref{eq:network-matrices}, it follows that \eqref{eq:networkXiZi} can be written in closed form as follows,
\be\label{eq:networkX-closed-form}
X_t=e^{{\boldsymbol{a}}t}\xi+\int_0^te^{{\boldsymbol{a}}(t-s)}{\boldsymbol{b}}\al_sds+\int_0^te^{{\boldsymbol{a}}(t-s)}d\bar B_s,\quad t\in[0,T].
\ee
\begin{tcolorbox}
The closed form \eqref{eq:networkX-closed-form} reveals that the stochastic differential network game studied by \citet{aurell2022stochastic} corresponds to the stochastic Volterra game from Section~\ref{subsec:Volterra-game} with controlled variables $X^i$ in~\eqref{eq:volterraXi} and $Z^i$ in~\eqref{eq:volterraZi} and coefficients given by
\be
\begin{aligned}
&P^i_t=\big(e^{{\boldsymbol{a}}t}\xi+\int_0^te^{{\boldsymbol{a}}(t-s)}d\bar B_s\big)_i\ ,\quad
&&G_1^{ij}(t,s)=\big(\mathds{1}_{\{t>s\}}e^{{\boldsymbol{a}}(t-s)}{\boldsymbol{b}}\big)_{ij},\\
&R^i_t=\Big(\frac{1}{N} w^{net}\big(e^{{\boldsymbol{a}}t}\xi+\int_0^te^{{\boldsymbol{a}}(t-s)}d\bar B_s\big)\Big)_i,\hspace{-3mm} 
&&G_2^{ij}(t,s)=\Big(\frac{1}{N} w^{net}\big(\mathds{1}_{\{t>s\}}e^{{\boldsymbol{a}}(t-s)}{\boldsymbol{b}}\big)\Big)_{ij}.
\end{aligned}
\ee

Minimizing the objective functionals in \eqref{eq:systemicJ^i} corresponds to maximizing the objective functionals in \eqref{eq:volterraJ^i} given in terms of $f^i_{Vol}$ as in \eqref{eq:volterraf} and $g^i_{Vol}$ as in \eqref{eq:volterrag} with parameters for $i=1,\dots, N$ given by
\be
\begin{aligned}
    p^i &= \frac{1}{2}[ C^{net}_f(i)]_{22} ,   \
    &&Q^i =  \frac{1}{2}{\begin{pmatrix} 
    [ C^{net}_f(i)]_{11} & [ C^{net}_f(i)]_{13} \\ [ C^{net}_f(i)]_{31} & [ C^{net}_f(i)]_{33}                                   \end{pmatrix}},  \quad 
    q^i = -\begin{pmatrix} 
    [ C^{net}_f(i)]_{12} \\ [ C^{net}_f(i)]_{23}   \end{pmatrix},  \\
    S^i &=  \frac{1}{2}[ C^{net}_g(i)],  
    &&s^i = 0. 
\end{aligned}\ee
\end{tcolorbox}
\begin{remark}\label{rem:improvements}  \citet{aurell2022stochastic} characterized the Nash equilibrium of the above finite-player network game as a system linear system of FBSDEs without the derivation of a solution. Our main results do not only yield explicit operator formulas for the equilibrium, but also allow without additional effort various extensions of the game. First, our framework allows the introduction of a direct dependence of player $i$'s state $X^i$ in \eqref{eq:networkXiZi} on the other players' controls $\al^j$, which can be incorporated by updating the coefficient matrix $\boldsymbol{b}$ in \eqref{eq:networkX} to a non-diagonal matrix. Second, we can implement non-symmetric heterogeneous interactions of the players in the sense that player $i$ does not care equally much about player $j$'s state as vice versa, by employing a non-symmetric interaction matrix $w^{net}$ in \eqref{eq:networkXiZi}. Third, in addition to the idiosyncratic Brownian noise we can include a common Brownian noise $\bar B^*$ by replacing $\bar B^i_t$ by $\bar B^i_t+\bar B^*_t$ in \eqref{eq:networkXiZi} and updating the processes $P^i_t$ and $R^i_t$ accordingly. Finally, we also allow the use of a more general independent or correlated noise.
\end{remark}

\subsection{Stochastic differential games on simple graphs}
We show the correspondence of our Volterra game framework from Section~\ref{subsec:Volterra-game} with stochastic differential games on simple graphs, studied by \citet{lacker2022case} and extended by \citet{hu2024finite}. In these references the following games were considered.

Let $\mathfrak{S} = (V_{sim},E_{sim})$ be a connected simple graph, where
$V_{sim} = \{v_1, \ldots , v_N\}$ denotes the set of vertices, and each edge $e\in E_{sim}$ is an unordered pair $e = (v, v')$ for some $v, v'\in V$. Two vertices $v, v'\in V_{sim}$ are said to be adjacent $v\sim v'$ if and only if $(v, v')\in E_{sim}$. The degree $d_v$ of the vertex $v\in V_{sim}$ is the number of edges that are connected to $v$.
Define the symmetric normalized Laplacian $L^{\mathfrak{S}}\in\R^{N
\times N}$ of $\mathfrak{S}$ as
\be\label{eq:laplacian}
L_{ij}^{\mathfrak{S}}\vcentcolon=\begin{cases}
    1,\hspace{24mm}\text{if }i=j,\\
    -(d_{v_i}d_{v_j})^{-\frac{1}{2}},\hspace{5mm}\text{if }i\neq j\text{ and }v_i\sim v_j,\\
    0,\hspace{24mm}\text{else.}
\end{cases}
\ee
For $i=1,\dots,N$, each player $i$ affects through her control $\al^i\in\mathcal{A}$ her individual state process ${X}^i$ with dynamics given by
\be
\begin{aligned}
\label{eq:graphs-dynamics}
d{X}_t^i&=\Big[\bar{a}( {Z}_t^i -{X}_t^i  )+\al_t^i\Big]dt+{\bar{\sigma}}d\bar{B}_t^i,\quad {X}_0^i=\bar{\xi}^i,\\
{Z}_t^i&=\frac{1}{\sqrt{d_{v_i}}} \sum_{j:v_j \sim v_i}  \frac{1}{\sqrt{d_{v_j}}}{X}_t^j, \quad i=1,\ldots, N, \  t\in [0,T],
\end{aligned}
\ee
for some initial states $\bar{\xi}^i$, independent Brownian motions $\bar{B}_t^i$,  speed of mean reversion ${a}_{sim}\geq 0$, and volatility ${\bar{\sigma}}>0$. Player $i$ seeks to minimize the cost functional
\be\label{eq:graphs-Ji}
{J}_{sim}^i(\al^i;\al^{-i}):=\E\Big[\int_0^T {f}_{sim}^i({X}_t^i,{Z}_t^i,\al_t^i)dt+g_{sim}^i({X}_T^i,{Z}_T^i)    \Big],\quad i=1,\ldots, N,
\ee
with the running and terminal costs given by
\be\begin{aligned}
f_{sim}^i({X}^i,{Z}^i, \al^i) &= \frac{(\al^i)^2}{2} - {q}_{sim} \al^i({Z}^i-{X}^i) + \frac{{\varepsilon}_{sim}}{2}({Z}^i-{X}^i)^2, \\
g_{sim}^i({X}^i,{Z}^i) &= \frac{c_{sim}}{2}({Z}^i-{X}^i)^2,
\end{aligned}\ee
where $q_{sim},{\eps}_{sim},c_{sim}\geq 0$, and the condition $q_{sim}^2\leq {\eps}_{sim}$ ensures convexity of the $f_{sim}^i$. When $\bar{a}=q_{sim}=\eps_{sim}=0$ and $\mathfrak{S}$ is a vertex-transitive graph, the framework recovers the model from \citet{lacker2022case}, and when $\mathfrak{S}$ is the complete graph with $N$ vertices, it recovers the systemic risk model from \citet{carmona2015mean}, which has already been addressed in Section~\ref{subsec:systemic-risk}.

\noindent \textbf{The correspondence with our Volterra game framework.}
First, recalling the notation \eqref{eq:notation-alpha} and the definition of the symmetric normalized Laplacian $L^{\mathfrak{S}}$ in \eqref{eq:laplacian}, the dynamics \eqref{eq:graphs-dynamics} can be expressed by the following $N$-dimensional system of SDEs,
\be\label{eq:graphsX-SDE}
dX_t=\big( -\bar{a}L^{\mathfrak{S}} X_t+\al_t\big) dt+\bar{\sigma} d \bar B_t,\quad {X}_0= \bar{\xi},\quad t\in[0,T],
\ee
where
$$
{X}_t\vcentcolon=( {X}_t^1,\ldots,  {X}_t^N)^\top,\quad
\bar B_t\vcentcolon=( \bar B_t^1,\ldots,  \bar B_t^N)^\top,\quad
{\bar{\xi}}\vcentcolon=( {\bar{\xi}}^1,\ldots,  {\bar{\xi}}^N)^\top.
$$
The representation \eqref{eq:graphsX-SDE} allows us to write \eqref{eq:graphs-dynamics} in closed form as follows,
\be\label{eq:graphsX-closed-form}
{X}_t=e^{-\bar{a}L^{\mathfrak{S}}t}{\bar{\xi}} +\int_0^te^{-\bar{a}L^{\mathfrak{S}}(t-s)}\al_sds+\int_0^te^{-\bar{a}L^{\mathfrak{S}}(t-s)}\bar{\sigma}d\bar B_s,\quad t\in[0,T].
\ee
\begin{tcolorbox}
The closed form \eqref{eq:graphsX-closed-form} shows that stochastic differential games on graphs as studied by \citet{lacker2022case} and \citet{hu2024finite} correspond to stochastic Volterra games from Section~\ref{subsec:Volterra-game} with controlled variables $X^i$ in~\eqref{eq:volterraXi} and $Z^i$ in~\eqref{eq:volterraZi} and coefficients given by
\be
\begin{aligned}
&P^i_t=\big(e^{-\bar{a}L^{\mathfrak{S}}t}{\bar{\xi}}+\int_0^te^{-\bar{a}L^{\mathfrak{S}}(t-s)}\bar{\sigma}d\bar B_s\big)_i\, ,\ \ 
&&G_1^{ij}(t,s)=\big(\mathds{1}_{\{t>s\}}e^{-\bar{a}L^{\mathfrak{S}}(t-s)}\big)_{ij},\\
&R^i_t=\Big(\big(\operatorname{Id}^N-L^{\mathfrak{S}}\big)P_t\Big)_{i} \ ,
&&G_2^{ij}(t,s)=\Big(\big(\operatorname{Id}^N-L^{\mathfrak{S}}\big)G_1(t,s)\Big)_{ij}.
\end{aligned}
\ee
Minimizing the cost functionals in \eqref{eq:graphs-Ji} corresponds to maximizing the objective functionals in \eqref{eq:volterraJ^i} given in terms of $f^i_{Vol}$ as in \eqref{eq:volterraf} and $g^i_{Vol}$ as in \eqref{eq:volterrag} with parameters for $i=1,\dots, N$ given by
\be
    p^i = \frac 1 2,  \,  
    Q^i =  \frac{{\eps}_{sim}}{2} \begin{pmatrix}
        1  & -1  \\
        -1   & 1
     \end{pmatrix},  \,  
    q^i = q_{sim}\begin{pmatrix}
        -1  \\
        1
     \end{pmatrix},  \,  
    S^i =   \frac{c_{sim}}{2} \begin{pmatrix}
        1  & -1  \\
        -1   & 1
     \end{pmatrix},  
    s^i   = 0.
\ee
\end{tcolorbox}
\begin{remark}
Notice that \citet{lacker2022case} and \citet{hu2024finite} studied Markovian Nash equilibria, that is, the players choose controls from the set of (full-information) Markovian controls, i.e. the set of Borel-measurable functions $\al:[0,T]\times\R^N\to\R$ satisfying a suitable integrability condition.
In our framework, players choose controls from the set $\mathcal{A}$ of progressively measurable processes in $L^2([0,T]\times\Omega,\R)$, leading to the concept of open-loop Nash equilibria. Those can be derived explicitly via Theorem~\ref{thm:finite-player-equilibrium}, which complements the results of the two aforementioned papers. Moreover, our framework allows for various extensions of the game, such as the ones discussed in Remark \ref{rem:improvements}.
\end{remark}

\section{Proofs of Section~\ref{sec:finite-player-game}}\label{sec:proofs-finite}
Since the number of agents $N\in\N$ is fixed, we omit the superscript $N$ from $\al^{i,N}, \al^{-i,N}$, $\al^N$ and $J^{i,N}$ defined in \eqref{eq:notation-alpha} and \eqref{eq:J^iN} throughout this section. We first establish the strict concavity property of the map $\al^{i} \mapsto J^i(\al^{i};\al^{-i})$, which is crucial for the derivation of the Nash equilibrium.
\begin{lemma} \label{lemma:concave-finite} 
Let $i \in \{1,\ldots, N\}$. Then under Assumptions~\ref{assum:A-B-C} and \ref{assum:nonnegative-definite}, for any $\al^{-i} \in \mathcal{A}^{N-1}$ fixed, the functional $\al^{i} \mapsto J^{i}(\al^{i};\al^{-i})$ in~\eqref{eq:J^iN} is strictly concave in $\al^{i} \in \mathcal{A}$.
\end{lemma}

\begin{proof}
Fix $i \in \{1,\ldots, N\}$ and $\al^{-i} \in \mathcal{A}^{N-1}$. 
Then, for any $\al^i, \beta^i \in \mathcal{A}$ such that $\al^i \neq \beta^i$, $d\P \otimes dt$-a.e.~on $\Omega \times [0,T]$, and for all $\varepsilon \in (0,1)$, a direct computation yields 
\begin{equation}
    \begin{aligned}\label{eq:concave-finite}
        & J^{i}(\varepsilon \al^{i} + (1-\varepsilon) \beta^{i} ; \al^{-i}) - \varepsilon J^{i}(\al^{i};  \al^{-i}) - (1-\varepsilon) J^{i}(\beta^{i};  \al^{-i}) \\[0.5ex]
        & = \varepsilon (1-\varepsilon) \E\left[ \langle \al^i - \beta^i , \boldsymbol{A}^i (\al^i - \beta^i) \rangle_{L^2} \right] \\[0.5ex]    
        & = \varepsilon (1-\varepsilon)  \E\left[ \langle \al^i - \beta^i ,\boldsymbol{B}^{ii} (\al^i - \beta^i) +\lambda^i(\al^i- \beta^i)\rangle_{L^2} \right],
    \end{aligned}
\end{equation}
where we used \eqref{eq:A^i} in the last equality. Recall that we denote by $(\vec{e_i})_{i=1}^N$ the standard basis of $\R^N$. Now, by defining 
\be 
f:\Omega\times [0,T]\to\R^N,\quad f(t)\vcentcolon =\vec{e_i} (\al^i_t - \beta^i_t),
\ee
and using \eqref{prod-id} and Assumption~\ref{assum:nonnegative-definite}, it follows from \eqref{eq:concave-finite} that
\be
\begin{aligned}
        & J^{i}(\varepsilon \al^{i} + (1-\varepsilon) \beta^{i} ; \al^{-i}) - \varepsilon J^{i}(\al^{i};  \al^{-i}) - (1-\varepsilon) J^{i}(\beta^{i};  \al^{-i}) \\[0.5ex]
        &=\varepsilon (1-\varepsilon)  
        \E\Big[\langle f, \frac{1}{2}(\boldsymbol{B}+\boldsymbol{\bar{B}}^*)f+\Lambda f\rangle_{L^2,N}\Big]\\[0.5ex]
        &\geq \varepsilon (1-\varepsilon) \E\Big[\frac{  c_0}{2}\langle f, f\rangle_{L^2,N}\Big]\\[0.5ex]
        &=\varepsilon (1-\varepsilon) \frac{  c_0}{2} \E\Big[\| \al^i- \beta^i\|_{L^2}^2\Big]\\[0.5ex]
        &>0
\end{aligned}
\ee
where we used that $\eps,  c_0>0$ and the fact that $\| \al^i- \beta^i\|_{L^2}^2 >0$, $\P$-a.s..
\end{proof}

Fix $i \in \{1,...,N\}$ and $\al^{-i} \in \mathcal{A}^{N-1}$. Lemma~\ref{lemma:concave-finite} states that the objective  functional $\al^i \mapsto J^i(\al^i; \al^{-i})$ of player $i$'s best response to all other players' fixed strategies $\al^{-i}$ is strictly concave. Hence, it admits a unique maximizer characterized by the critical point at which the G\^ateaux derivative
\be \label{eq:gateaux}
    \langle  \nabla  J^i(\al^i; \al^{-i}) , \beta^i  \rangle = \lim_{\varepsilon \to 0} \frac{ J^i(\al^i + \varepsilon \beta^i; \al^{-i}) -  J^i(\al^i; \al^{-i})}{\varepsilon},
\ee
vanishes for all $\beta^i\in \mathcal{A}$. Therefore, a control $\al^i \in \mathcal A$ maximizes \eqref{eq:J^iN} if and only if $\al^i$ satisfies the first order condition 
\be \label{eq:foc} 
 \langle \nabla J^i(\al^i; \al^{-i}) , \beta^i  \rangle  = 0, \quad \text{for all } \beta^i \in \mathcal{A},
\ee
(see Propositions~1.2 and 2.1 in Chapter~II of \citet{ekeland1999convex}).
This allows the derivation of a system of stochastic Fredholm equations which characterizes the Nash equilibrium.
\begin{proof}[Proof of Proposition~\ref{prop:fredholm-foc}]
Fix $i \in \{1,...,N\}$ and $\al^{-i} \in \mathcal{A}^{N-1}$. Plugging \eqref{eq:J^iN} into \eqref{eq:gateaux}, a direct computation and Fubini's theorem yield that for any $\al^i,\beta^i\in\mathcal{A}$,
\be
\begin{aligned}\label{eq:foc2}
&\langle \nabla  J^i(\al^i; \al^{-i}) , \beta^i  \rangle   \\
&=\E\Big[ \langle \beta^i, -\big(\boldsymbol{A}^i+(\boldsymbol{A}^i)^*\big) \al^i 
- \Big(\sum_{j\neq i} \big(\boldsymbol B^{ij}+(\boldsymbol{\bar{B}}^{ji})^*\big) \al^j  \Big)+ b^{ii} \rangle_{L^2}  \Big] 
\\
&=\mathbb E\Big[ \langle \beta^i,  \,-2\lam^i \al^i - \Big(\sum_{j=1}^N \big(\boldsymbol B^{ij}+(\boldsymbol{\bar{B}}^{ji})^*\big) \al^j\Big)  + b^{ii} \rangle_{L^2}    \Big]
\\
&= \int_0^T \mathbb E\left[ \beta^i_t \left(   \,-2\lam^i \al^i_t  - \Big(\sum_{j=1}^N \big((\boldsymbol B^{ij}+(\boldsymbol{\bar{B}}^{ji})^*) \al^j\big)(t)\Big) + b^{ii}_t  \right) \right] dt.
\end{aligned} 
\ee
By conditioning on $\mathcal F_t$ and using the  tower property we get from \eqref{eq:foc} and \eqref{eq:foc2} the following first order condition, 
\be \label{eq:foc3} 
2\lam^i \al^i_t  = b^{ii}_t -\sum_{j=1}^N \big((\boldsymbol B^{ij}+(\boldsymbol{\bar{B}}^{ji})^*) \E_t[\al^j]\big)(t),
\quad d\P \otimes dt \textrm{-a.e.~on } \Omega \times [0,T].
\ee
It follows that a strategy profile $\al\in \mathcal{A}^N$ is a Nash equilibrium if and only if the following holds for all $i \in \{1,\ldots,N\}$, 
\begin{equation} \label{eq:foc4}
    2 \lambda^i \al_t^i =  \,  b^{ii}_t - \sum_{j=1}^N \Big( \int_0^t   B^{ij}(t,s) \al^j_s ds +  \int_t^T   \bar B^{ji}(s,t)  \E_t[\al^j_s] ds\Big),  \ \ d\P \otimes dt \textrm{-a.e.~on } \Omega \times [0,T].
\end{equation}
Using the notation \eqref{eq:notation-alpha}, \eqref{eq:B-Bbar}, \eqref{eq:Lambda} and \eqref{eq:b}, the equations \eqref{eq:foc4} for all the players $i\in\{1,\ldots,N\}$ can be summarized as an $N$-dimensional coupled system of stochastic Fredholm equations as in \eqref{eq:fredholm-foc}.
\end{proof}

In order to prove Theorem~\ref{thm:finite-player-equilibrium} we will need the following technical lemma. 
\begin{lemma}\label{lemma:D_t}
Under Assumptions~\ref{assum:A-B-C} and \ref{assum:nonnegative-definite}, there exists a constant $c_0>0$ such that for every $t\in[0,T]$ it holds that 
\be
\langle f,2\Lambda f + (\boldsymbol{B}_t+\boldsymbol{\bar{B}}^*_t)f \rangle_{L^2,N}\geq   c_0 \langle f, f\rangle_{L^2,N}\quad\text{for any }f\in L^2([0,T],\R^N). 
\ee
\end{lemma}
\begin{proof}
It follows from Assumption  \ref{assum:nonnegative-definite} that there exists a $  c_0>0$ such that
\be\label{eq:B-Bbar-inequality}
\langle f, (\boldsymbol{B}+\boldsymbol{\bar{B}}^*)f\rangle_{L^2,N}\geq 
\langle f,   c_0 f-2\Lambda f\rangle_{L^2,N}\quad \text{for any }f\in L^2([0,T],\R^N). 
\ee
Without loss of generality, assume that $  c_0<2\lam^i$ for all $i=1,\dots,N$. Fix $t\in[0,T]$ and $f\in L^2([0,T],\R^N)$. Define $f_t(s)\vcentcolon=\mathds{1}_{\{s> t\}}f(s)$. Then it follows that
\be
\begin{aligned}
\langle f,2\Lambda f + (\boldsymbol{B}_t+\boldsymbol{\bar{B}}^*_t)f \rangle_{L^2,N}
&=2\langle f,\Lambda f \rangle_{L^2,N}+\langle f_t,(\boldsymbol{B}+\boldsymbol{\bar{B}}^*)f_t \rangle_{L^2,N}\\
&\geq \sum_{i=1}^N\left(2\lam^i\langle f^i, f^i \rangle_{L^2}+(  c_0-2\lam^i)\langle f_t^i, f_t^i \rangle_{L^2}\right)\\
&\geq \sum_{i=1}^N\left(2\lam^i\langle f^i, f^i \rangle_{L^2}+(  c_0-2\lam^i)\langle f^i, f^i \rangle_{L^2}\right)\\
&=  c_0\langle f, f \rangle_{L^2,N}
\end{aligned}
\ee
where, the first equality holds since $B$ and $\bar B$ are Volterra kernels
, the first inequality holds due to \eqref{eq:B-Bbar-inequality}, and the second inequality follows from $  c_0<2\lam^i$ for all $i=1,\dots,N$. This completes the proof.
\end{proof}

\begin{proof}[Proof of Theorem~\ref{thm:finite-player-equilibrium}] The key idea is to apply Proposition~4.1 of \citet{abi2024optimal} to the system of stochastic Fredholm equations in \eqref{eq:fredholm-foc} with 
\be\label{eq:coefficient-choices}
K=B,
\quad L=\bar B,
\quad f= b,
\quad \bar \Lambda = 2\Lambda.
\ee
Notice that under Assumption~\ref{assum:nonnegative-definite} the integral operators on $L^2([0,T],\R^N)$ induced by $B$ in and $\bar B$ are not necessarily nonnegative definite. However, recalling Definition~\ref{def:truncated-kernel}, for the proof of Proposition~4.1 of \citet{abi2024optimal} it is only required that there exists a $  c_0>0$ such that for every $t\in[0,T]$ the bounded linear operator given by
$\boldsymbol{D}_t\vcentcolon=\bar\Lambda\boldsymbol{I}^N+\boldsymbol{K}_t+\boldsymbol{L}^*_t$ 
satisfies
\be\label{eq:condition D_t}
\langle f,\boldsymbol{D}_t f\rangle_{L^2,N}\geq   c_0 \langle f, f\rangle_{L^2,N}\quad\text{for any }f\in L^2([0,T],\R^N). 
\ee
Now, recalling \eqref{eq:coefficient-choices}, the inequality \eqref{eq:condition D_t} follows from Lemma~\ref{lemma:D_t}. 
Thus, Lemma~\ref{lemma:concave-finite}, Proposition~\ref{prop:fredholm-foc} and Proposition~4.1 of \citet{abi2024optimal} yield the existence of the unique Nash equilibrium given in Theorem~\ref{thm:finite-player-equilibrium}.
\end{proof}

\section{Proofs of Section~\ref{sec:infinite-player-game}}\label{sec:proofs-infinite}
This section is dedicated to the proofs of Section~\ref{sec:infinite-player-game}, excluding the proof of Proposition~\ref{prop:fredholm-foc-infinite} which is deferred to the Appendix \ref{appendix}. In order to prove Theorem~\ref{thm:infinite-player-equilibrium}, we first note that Proposition~\ref{prop:fredholm-foc} simplifies in the setting of Section~\ref{subsec:definition-finite-player-game-graph}, hence the precise statement is given below in Proposition~\ref{prop:fredholm-foc-graph}. Also note that the nonnegative definiteness of the operator $\wt{\boldsymbol A}$ from Assumption~\ref{assum:A-B-C-graph} ensures the strict concavity of the objective functionals $J_0^{i,N}$ defined in \eqref{eq:J_0^iN}.
\begin{proposition}\label{prop:fredholm-foc-graph}
Let Assumption~\ref{assum:A-B-C-graph} be satisfied and assume that $\wt{\boldsymbol A}$ is nonnegative definite. Then, a set of strategies $\al^N=(\al^{1,N},\dots,\al^{N,N})\in\mathcal{A}^N$ is a Nash equilibrium if and only if it satisfies the $N$-dimensional coupled system of stochastic Fredholm equations of the second kind given by
\be \begin{aligned} \label{eq:fredholm-foc-graph}
2 \lambda \al_t^{i,N} &=  b^{i,N}_t  - \int_0^t   \wt{A}(t,s) \al^{i,N}_s ds - \hspace{-0.5mm} \int_t^T \wt{A}(s,t)  \E_t[\al^{i,N}_s] ds- \hspace{-0.5mm}\int_0^t   \wt{B}(t,s) \frac{1}{N}\sum_{j=1}^Nw_{ij}^N\al^{j,N}_s ds  \\&- \int_t^T \wt{B}(s,t)   \frac{1}{N}\sum_{j=1}^Nw_{ij}^N\E_t[\al^{j,N}_s] ds, \quad d\P \otimes dt \textrm{-a.e.~on } \Omega \times [0,T],\ i=1,\ldots, N.
\end{aligned}\ee 
\end{proposition}

Now we are ready to prove Theorem~\ref{thm:infinite-player-equilibrium}. 
\begin{proof}[Proof of Theorem~\ref{thm:infinite-player-equilibrium}]
The key idea is to make use of the spectral decomposition \eqref{eq:representation-Wf} of $\boldsymbol{W}$. Namely, Remark \ref{rem:spectral-decomposition} yields the existence of an orthonormal basis $(\varphi_i)_{i\in\N}$ of $L^2([0,1],\R)$ consisting of eigenfunctions of the operator $\boldsymbol{W}$ together with a nonincreasing sequence of corresponding eigenvalues $(\vartheta_i)_{i\in\N}$ such that for any function $f\in L^2([0,1],\R)$ the representation \eqref{eq:representation-Wf} applies.
Plugging \eqref{eq:representation-Wf} into \eqref{eq:fredholm-foc-infinite} yields the infinite-dimensional system of stochastic Fredholm equations given by,
\be\label{eq:foc-infinite5}
\begin{aligned}
2\lam \al_t^u=b_t^u &- \int_0^t \wt{A}(t,s)\al^u_sds- \int_t^T \wt{A}(s,t)\E_t[\al^u_s]ds\\
&- \int_0^t\wt{B}(t,s)   \sum_{j=1}^\infty\vartheta_j\langle\varphi_j,\al_s\rangle_{L^2([0,1],\R)}\varphi_j(u)ds\\
&- \int_t^T\wt{B}(s,t)   \sum_{j=1}^\infty\vartheta_j\langle\varphi_j,\E_t[\al_s]\rangle_{L^2([0,1],\R)}\varphi_j(u)ds,
\end{aligned}\ee
with equality  understood in the almost sure sense as before. For every $i\in\N$, define the $\F$-progressively measurable processes $\wt{\al}^i$ and $\wt{b}^i$ by
\be\begin{aligned} \label{eq:alpha-tilde-b-tilde}
\wt{\al}^i_t\vcentcolon=&\langle\varphi_i,\al_t\rangle_{L^2([0,1],\R)}=\int_0^1\varphi_i(u)\al_t^udu,\quad &0\leq t\leq T,\\
\wt{b}^i_t\vcentcolon=&\langle\varphi_i,b_t\rangle_{L^2([0,1],\R)}=\int_0^1\varphi_i(u)b_t^udu,\quad &0\leq t\leq T.
\end{aligned}\ee
Then $\wt{b}^i\in L^2(\Omega\times [0,T],\R)$ for every $i\in \N$ by an application of the Cauchy–Schwarz inequality and due to the two facts that $b \in L^2(\Omega\times [0,T]\times [0,1],\R)$ by assumption and $\varphi_i\in L^2([0,1],\R)$ for every $i\in \N$. Now, applying for every $i\in\N$ to both sides of the equation \eqref{eq:foc-infinite5} the operator $\langle\varphi_i,\cdot\rangle_{L^2([0,1],\R)}$ and using Fubini's theorem and \eqref{eq:alpha-tilde-b-tilde} yields the following infinite-dimensional system of stochastic Fredholm equations in one variable,
\be\label{eq:foc-infinite-reduced}
\begin{aligned}
2\lam \wt{\al}^i_t=\wt{b}^i_t &- \int_0^t \wt{A}(t,s)\wt{\al}^i_sds- \int_t^T \wt{A}(s,t)\E_t[\wt{\al}^i_s]ds \\
&-\int_0^t\wt{B}(t,s)  \sum_{j=1}^\infty\vartheta_j\langle\varphi_i,\varphi_j\rangle_{L^2([0,1],\R)}\wt{\al}^j_sds\\
&- \int_t^T\wt{B}(s,t) \sum_{j=1}^\infty\vartheta_j\langle\varphi_i,\varphi_j\rangle_{L^2([0,1],\R)}\E_t[\wt{\al}^j_s]ds\\
=\wt{b}^i_t &- \int_0^t \big(\wt{A}(t,s)+\vartheta_i\wt{B}(t,s)\big)\wt{\al}^i_sds\\&- \int_t^T \big(\wt{A}(s,t)+\vartheta_i\wt{B}(s,t)\big)\E_t[\wt{\al}^i_s]ds, \qquad i\in\N,
\end{aligned}\ee
where the second equality of \eqref{eq:foc-infinite-reduced} follows from the fact that $(\varphi_i)_{i\in\N}$ is an orthonormal basis of $L^2([0,1],\R)$. Next, since the infinite-dimensional system of stochastic Fredholm equations in  \eqref{eq:foc-infinite-reduced} is fully decoupled, we can solve for each $i\in\N$ the corresponding equation for $\wt{\al}^i$ individually. For this, we define the Volterra kernels
\be\label{eq:C^i}
\wt{C}^i\vcentcolon=\frac{1}{2\lam}\big(\wt{A}+\vartheta_i\wt{B}\big)\in\mathcal{G},\quad\text{for every }i\in\N.
\ee
To solve \eqref{eq:foc-infinite-reduced}, we also need the following technical lemma which will be proved at the end of this section. For that, recall Definition~\ref{def:truncated-kernel} of truncated Volterra kernels.
\begin{lemma}\label{lemma:C^i-nonnegative}
Under Assumptions~\ref{assum:A-B-C-graphon} and \ref{assum:nonnegative-definite-graphon}, the kernels $\wt{C}^i$ defined in \eqref{eq:C^i} satisfy
$$
\langle f,f+\wt{\boldsymbol{C}}^if+(\wt{\boldsymbol{C}}^i)^*f\rangle_{L^2}\geq \frac{ c_W}{\lam}\langle f,f\rangle_{L^2},\quad \text{for every }f\in L^2([0,T],\R),\ i\in\N.
$$
If $c_W\leq\lam$, the corresponding truncated kernels $\wt{C}_t^i$ satisfy for every $t\in [0,T]$, 
$$
\langle f,f+\wt{\boldsymbol{C}}_t^if+(\wt{\boldsymbol{C}}_t^i)^*f\rangle_{L^2}\geq \frac{ c_W}{\lam}\langle f,f\rangle_{L^2},\quad \text{for every }f\in L^2([0,T],\R),\ i\in\N.
$$
\end{lemma}
By the same argument as the one given in the proof of Theorem~\ref{thm:finite-player-equilibrium}, we can apply Proposition~5.1 of \citet{abijaber2023equilibrium} thanks to Lemma~\ref{lemma:C^i-nonnegative}, when choosing $c_W\leq\lam$ in Assumption~\ref{assum:nonnegative-definite-graphon}. The proposition implies that for every $i\in \N$ the stochastic Fredholm equation \eqref{eq:foc-infinite-reduced} admits a unique $\F$-progressively measurable solution $\hat{\al}^i$ in $L^2(\Omega\times [0,T],\R)$ given by
\be \label{eq:alpha_t^i-hat}
    \hat{\al}_t^i \vcentcolon =   \big((\boldsymbol{I} + \boldsymbol E^i)^{-1} \gamma^i \big) (t), \quad 0\leq t \leq T, 
\ee
where
\be 
\gamma_t^i \vcentcolon=   \wt{b}^i_t  -  \langle  \mathds{1}_{\{t< \cdot\}}  \wt{C}^i(\cdot,t),  ({\boldsymbol D}_t^i)^{-1}\mathds{1}_{\{t< \cdot\}} \E_t[\wt{b}^i_{\cdot}] \rangle_{L^2}
\ee
\be
E^i(t,s) \vcentcolon= - \mathds{1}_{\{t>s\}}  \big( \langle  \mathds{1}_{\{t< \cdot\}}   \wt{C}^i(\cdot,t), ({\boldsymbol D}_t^i)^{-1}   \mathds{1}_{\{t< \cdot\}}  \wt{C}^i(\cdot,s)  \rangle_{L^2}    -   \wt{C}^i(t,s)   \big),
\ee
\be \label{d-opt} 
 {\boldsymbol{D}}^i_t \vcentcolon= \boldsymbol{I}  + \wt{\boldsymbol{C}}^i_t + (\wt{\boldsymbol{C}}^i_t)^*,
\ee
and $\boldsymbol E^i$ is the integral operator induced by the kernel $E^i$. Finally, define a strategy profile $\hat{\al}$, our candidate for the graphon game Nash equilibrium, by 
\be\label{eq:alpha-hat}
\hat{\al}:\Omega\times [0,T]\times [0,1]\to\R,\quad \hat{\al}_t^u\vcentcolon=\sum_{i=1}^\infty \varphi_i(u)\hat{\al}^i_t.
\ee
It will be shown in Lemma~\ref{lemma:alpha-hat-admissible} that $\hat\al$ is admissible, i.e. that $\hat{\al}\in\mathcal{A}^\infty$. In order to verify that $\hat{\al}$ is indeed a Nash equilibrium, we plug \eqref{eq:alpha-hat} into the rewritten stochastic Fredholm equation in \eqref{eq:foc-infinite5}, which yields
\be\label{eq:foc-plugged-in}
\begin{aligned}
2\lam \sum_{i=1}^\infty \varphi_i(u)\hat{\al}^i_t=b_t^u &- \int_0^t \wt{A}(t,s)\sum_{i=1}^\infty \varphi_i(u)\hat{\al}^i_sds- \int_t^T \wt{A}(s,t)\E_t\big[\sum_{i=1}^\infty \varphi_i(u)\hat{\al}^i_s\big]ds\\
&- \int_0^t\wt{B}(t,s)  \sum_{j=1}^\infty\vartheta_j\varphi_j(u)\int_0^1\varphi_j(v)\sum_{i=1}^\infty \varphi_i(v)\hat{\al}^i_sdvds\\
&- \int_t^T\wt{B}(s,t)  \sum_{j=1}^\infty\vartheta_j\varphi_j(u)\int_0^1 \varphi_j(v)\E_t\big[\sum_{i=1}^\infty \varphi_i(v)\hat{\al}^i_s\big]dvds.
\end{aligned}\ee
By making use of the fact that $(\varphi_i)_{i\in\N}$ is an orthonormal basis of $L^2([0,1],\R)$ and recalling the definition of the processes $\smash{\wt{b}^i}$ defined in \eqref{eq:alpha-tilde-b-tilde}, we get that equation \eqref{eq:foc-plugged-in} is equivalent to
\be\label{eq:foc-plugged-in-rewritten}
\begin{aligned}
\sum_{i=1}^\infty \varphi_i(u)2\lam\hat{\al}^i_t
&=\sum_{i=1}^\infty \varphi_i(u)\wt{b}^i_t 
- \sum_{i=1}^\infty \varphi_i(u)\hspace{-0.5mm}\int_0^t \wt{A}(t,s)\hat{\al}^i_sds
- \sum_{i=1}^\infty \varphi_i(u)\hspace{-0.5mm}\int_t^T \wt{A}(s,t)\E_t[\hat{\al}^i_s]ds\\
&- \sum_{i=1}^\infty\varphi_i(u)\int_0^t\wt{B}(t,s)   \vartheta_i\hat{\al}^i_sds
- \sum_{i=1}^\infty\varphi_i(u)\int_t^T\wt{B}(s,t)  \vartheta_i\E_t[\hat{\al}^i_s]ds,
\end{aligned}\ee
where we renamed the summation index in the last two terms of \eqref{eq:foc-plugged-in-rewritten}. Now equation \eqref{eq:foc-plugged-in-rewritten} is satisfied, as it coincides with the weighted infinite sum of the stochastic Fredholm equations in \eqref{eq:foc-infinite-reduced} summed over all $i\in\N$ with weights $\varphi_i(u)$. It follows that the strategy profile $\hat{\al}$ defined in \eqref{eq:alpha-hat} solves the rewritten stochastic Fredholm equation in \eqref{eq:foc-infinite5} and therefore also the stochastic Fredholm equation in \eqref{eq:fredholm-foc-infinite} itself.
Thus, by Proposition~\ref{prop:fredholm-foc-infinite} and Lemma~\ref{lemma:alpha-hat-admissible}, $\hat{\al}$ is a graphon game Nash equilibrium in the sense of Definition~\ref{def:Nash-infinite-player}.

For the uniqueness of $\hat{\al}$, notice that any graphon game Nash equilibrium has to satisfy the stochastic Fredholm equation \eqref{eq:fredholm-foc-infinite} and thus \eqref{eq:foc-infinite5} as well. It follows that its Fourier coefficients defined in \eqref{eq:alpha-tilde-b-tilde} have to satisfy \eqref{eq:foc-infinite-reduced}. Moreover, they have to be $\F$-progressively measurable and in $L^2(\Omega\times [0,T],\R)$, because the graphon game Nash equilibrium has to be $\F$-progressively measurable and in $L^2(\Omega\times [0,T]\times [0,1],\R)$.
Now by Proposition~4.1 in \citet{abijaber2023equilibrium}, equation \eqref{eq:foc-infinite-reduced} has for every $i\in\N$ a unique $\F$-progressively measurable solution $\hat{\al}^i$ in $L^2(\Omega\times [0,T],\R)$ given by \eqref{eq:alpha_t^i-hat}. Therefore, the Fourier coefficients are unique and thus they uniquely determine the graphon game Nash equilibrium via equation \eqref{eq:alpha-hat}, since $(\varphi_i)_{i\in\N}$ is an orthonormal basis of $L^2([0,1],\R)$.
\end{proof}

\begin{lemma} \label{lemma:alpha-hat-admissible}
Under Assumptions~\ref{assum:A-B-C-graphon} and \ref{assum:nonnegative-definite-graphon}, the strategy profile $\hat{\al}$ defined in \eqref{eq:alpha-hat} is admissible, i.e. $\hat{\al}\in\mathcal{A}^\infty$.
\end{lemma}
Since the proof of Lemma~\ref{lemma:alpha-hat-admissible} mainly involves standard techniques based on Grönwall’s lemma, it is deferred to Appendix \ref{appendix}.

\begin{proof}[Proof of Lemma~\ref{lemma:C^i-nonnegative}]
Fix $i\in\N$ and $f\in L^2([0,T],\R)$. Define the product function $g\in L^2([0,T]\times[0,1],\R)$ by
\be \label{eq:definition-g}
g(t,u):= f(t)\varphi_i(u),\quad (t,u)\in [0,T]\times [0,1].
\ee
By Assumption~\ref{assum:nonnegative-definite-graphon}, it holds that 
\be\begin{aligned}\label{eq:Ci-nonnegative}
0&\leq \int_0^T\int_0^1 (\lam- c_W)g(t,u)^2 + g(t,u) \big((\wt{\boldsymbol A}g(u))(t)+(\wt{\boldsymbol B}(\boldsymbol {W}g)(u))(t)\big)dudt\\
&=\int_0^T\int_0^1 (\lam- c_W)f(t)^2\varphi_i(u)^2+ f(t)\varphi_i(u) \big(\varphi_i(u)(\wt{\boldsymbol A}f)(t)+\vartheta_i\varphi_i(u)(\wt{\boldsymbol B}f)(t)\big)dudt\\
&=\int_0^T f(t) \big((\lam- c_W)f(t)+(\wt{\boldsymbol A}f)(t)+\vartheta_i(\wt{\boldsymbol B}f)(t)\big)dt,
\end{aligned}\ee
where we used \eqref{eq:definition-g} for the first equality and the fact that $(\varphi_i)_{i\in\N}$ is an orthonormal basis of $L^2([0,1],\R)$ for the second equality. 
Dividing \eqref{eq:Ci-nonnegative} by $2\lam$ yields that 
$$
\langle f,\tfrac{1}{2}f+\wt{\boldsymbol{C}}^if\rangle_{L^2}\geq \frac{ c_W}{2\lam}\langle f,f\rangle_{L^2},
$$
and therefore by Remark \ref{rem:nonnegative-operators} that
\be \label{eq:D^i-coercive}
\langle f,f+\wt{\boldsymbol{C}}^if+(\wt{\boldsymbol{C}}^i)^*f\rangle_{L^2}\geq \frac{ c_W}{\lam}\langle f,f\rangle_{L^2}.
\ee
For the second part of the lemma, fix $t\in [0,T]$ and let $f_t(s)\vcentcolon=\mathds{1}_{\{s> t\}}f(s)$. Then 
\be\begin{aligned}\label{eq:D_t^i-coercive}
\langle f,\wt{\boldsymbol{C}}_t^if+(\wt{\boldsymbol{C}}_t^i)^*f\rangle_{L^2}&=
\langle f_t,\wt{\boldsymbol{C}}^if_t+(\wt{\boldsymbol{C}}^i)^*f_t\rangle_{L^2}\\
&\geq \big(\frac{c_w}{\lam}-1\big)\langle f_t, f_t\rangle_{L^2}\\
&\geq \big(\frac{c_w}{\lam}-1\big)\langle f, f\rangle_{L^2},
\end{aligned}\ee
where the equality in \eqref{eq:D_t^i-coercive} follows from Definition~\ref{def:truncated-kernel} and by definition of $f_t$, the first inequality follows from \eqref{eq:D^i-coercive} and the second inequality follows from $c_w\leq\lam$. Rearranging \eqref{eq:D_t^i-coercive} completes the proof.
\end{proof}
\section{Proofs of Section~\ref{sec:convergence}}\label{sec:proofs-convergence}
We first prove the results of Section~\ref{subsec:preliminaries}, and start by we showing that any finite-player game on a weighted graph can be equivalently reformulated as a graphon game.
\begin{proof}[Proof of Proposition~\ref{prop:correspondence-finite-infinite-game}]
Suppose that $\hat\al\in\mathcal{A}^\infty$ is a Nash equilibrium of the graphon game with underlying step graphon $W^N$ and family of noise processes $b^{N}_{\operatorname{step}}$. Then, by Definition~\ref{def:Nash-infinite-player}, it maximizes the objective functional
 \be
    \begin{aligned}\label{eq:J^uWn}
    J^{u,W^N}(\al^{u}; (\al^{v})_{v\neq u}) =\E\Big[ &- \langle \al^u, (\wt{\boldsymbol A}+\lambda \boldsymbol{I}) \al^u \rangle_{L^2} -\langle \al^u,\big(\wt{\boldsymbol B}+\wt{\boldsymbol B}^*\big) (\boldsymbol{W}^N\al)(u) \rangle_{L^2}   \\ 
    &- \langle (\boldsymbol{W}^N\al)(u),\smash{\wt{\boldsymbol{C}}}(\boldsymbol{W}^N\al)(u) \rangle_{L^2}
    + \langle b^{u,N}_{\operatorname{step}}, \al^u \rangle_{L^2} \\&+\langle b^*, (\boldsymbol{W}^N\al)(u)\rangle_{L^2} + c^u \Big]
\end{aligned}\ee
for $\mu$-a.e. $u\in I$.
Now, since $W^N$ and $b^{N}_{\operatorname{step}}$ are step functions with respect to the partition $\mathcal{P}^N$, the second, third, fourth and fifth term on the right-hand side of \eqref{eq:J^uWn} are also step functions with respect to $\mathcal{P}^N$. It follows that the maximizer $\hat\al$ of the functionals \eqref{eq:J^uWn} is a step function with respect to $\mathcal{P}^N$ as well, since the random variable $c^u$ does not affect the maximization and the individual maximizers $\hat\al^u$ of \eqref{eq:J^uWn} for $u\in I$  are unique due to its concavity.
Define $\hat\al^{i,N}$ to be the value of $\hat\al$ on $\mathcal{P}_i^N$ for all $i=1,\ldots, N$. Then the stochastic Fredholm equation \eqref{eq:fredholm-foc-infinite}  satisfied by $\hat{\al}$ from Proposition~\ref{prop:fredholm-foc-infinite} is equivalent to the following,
\be\begin{aligned} \label{eq:fredholm-step-graphon-finite} 
2\lam \hat\al_t^{i,N}&=b_t^{i,N} - \hspace{-0.5mm}\int_0^t \wt{A}(t,s)\hat\al_s^{i,N}ds-\hspace{-0.5mm} \int_t^T \wt{A}(s,t)\E_t[\hat\al_s^{i,N}]ds- \hspace{-0.5mm}\int_0^t\wt{B}(t,s)  \frac{1}{N}\sum_{j=1}^Nw_{ij}^N\hat\al_s^{j,N}ds\\
&- \int_t^T\wt{B}(s,t)  \frac{1}{N}\sum_{j=1}^Nw_{ij}^N\E_t[\hat\al_s^{j,N}]ds,\ \ d\P \otimes dt \textrm{-a.e.~on } \Omega \times [0,T],\ i\in\{1,\ldots,N\},
\end{aligned}\ee
where $b^{i,N}$ is the value of $b^N_{\operatorname{step}}$ on $\mathcal{P}_i^N$ for all $i=1,\ldots, N$. Now by Proposition~\ref{prop:fredholm-foc-graph}, $\hat\al^N=(\hat\al^{1,N},\ldots, \hat\al^{N,N})$ is a Nash equilibrium of the finite-player game, as it solves equation \eqref{eq:fredholm-step-graphon-finite}. 
Using Propositions~\ref{prop:fredholm-foc-graph} and \ref{prop:fredholm-foc-infinite} again, the reverse direction follows immediately. 
\end{proof}

Next, we derive a continuity result for the graphon game, which states that a  small change of the underlying graphon $W$ and the noise processes only results in a small change of its corresponding Nash equilibrium.

\begin{proposition}\label{prop:graphon-game-continuity}
Consider a graphon game with underlying graphon $W$ and family of noise processes $b$,  such that Assumptions~\ref{assum:A-B-C-graphon} and \ref{assum:nonnegative-definite-graphon} are satisfied. Let $\hat\al\in\mathcal{A}^\infty$ denote the corresponding unique Nash equilibrium (which exists by Theorem~\ref{thm:infinite-player-equilibrium}). Consider another graphon game with underlying graphon $W'$ and family of noise processes $b'$ and assume that it admits a Nash equilibrium $\hat\beta\in\mathcal{A}^\infty$. Then
\be\begin{aligned} \label{eq:continuity-bound} 
&\E\Big[\int_0^1\int_0^T(\hat\al_t^u-\hat\beta_t^u)^2dtdu\Big] \\&\leq\frac{1}{2 c_W}\|\wt{\boldsymbol{B}}\|_{\operatorname{op}}\|\boldsymbol{W}-\boldsymbol{W}'\|_{\operatorname{op}}\Big(\E\Big[\int_0^1\int_0^T(\hat\al_t^u)^2dtdu\Big]+3\E\Big[\int_0^1\int_0^T(\hat\beta_t^u)^2dtdu\Big]\Big)
\\&+\frac{1}{2 c_W} \E\Big[\int_0^1\hspace{-1mm}\int_0^T(b_t^u-b^{'u}_t)^2dtdu\Big]^\frac{1}{2}\Big(\E\Big[\int_0^1\hspace{-1mm}\int_0^T(\hat\al_t^u)^2dtdu\Big]^\frac{1}{2}+\E\Big[\int_0^1\hspace{-1mm}\int_0^T(\hat\beta_t^u)^2dtdu\Big]^\frac{1}{2}\Big),
\end{aligned}\ee
where $c_W>0$ is the constant from Assumption~\ref{assum:nonnegative-definite-graphon}.
\end{proposition}
In order to prove Proposition~\ref{prop:graphon-game-continuity}, the following lemma is required. 
\begin{lemma}\label{lemma:(A+A*)+(B+B*)W-nonnegative}
Under Assumption~\ref{assum:nonnegative-definite-graphon}, it holds for all $g\in L^2([0,T]\times [0,1],\R)$ that 
$$
\int_0^1\hspace{-0.5mm}\int_0^T g(t,u)\Big(2(\lam- c_W)g(t,u)+\big((\wt{\boldsymbol{A}}+\wt{\boldsymbol{A}}^*)g(u)\big)(t)+\big((\wt{\boldsymbol{B}}+\wt{\boldsymbol{B}}^*)(\boldsymbol{W}g)(u)\big)(t)\Big)dtdu\geq 0.
$$
\end{lemma}
\begin{proof}[Proof of Lemma~\ref{lemma:(A+A*)+(B+B*)W-nonnegative}]
The proof follows from recalling \eqref{prod-id}, taking the adjoint of the operator on the left-hand side of the inequality in Assumption~\ref{assum:nonnegative-definite-graphon}, and noting that
\be\begin{aligned}
&\int_0^1\int_0^T g(t,u)\big(\wt{\boldsymbol{B}}(\boldsymbol{W}g)(u)\big)(t)dtdu\\&=
\int_0^1\int_0^T\int_0^1\int_0^T g(t,u)\wt{B}(t,s)W(u,v)g(s,v)dsdvdtdu\\
&=\int_0^1\int_0^T g(s,v)\big(\wt{\boldsymbol{B}}^*(\boldsymbol{W}g)(v)\big)(s)dsdv,
\end{aligned}\ee
where used Fubini's theorem and the fact that $\boldsymbol{W}=\boldsymbol{W}^*$.
\end{proof}

\begin{proof}[Proof of Proposition~\ref{prop:graphon-game-continuity}]
We start by subtracting from each other the two characterizing stochastic Fredholm equations from Proposition~\ref{prop:fredholm-foc-infinite}, which yields the following equation for the difference of the two Nash equilibria,
\be\begin{aligned} \label{eq:fredholm-difference} 
2\lam (\hat\al_t^u-\hat\beta_t^u)&=(b_t^u-b^{'u}_t)- \int_0^t \wt{A}(t,s)(\hat\al_s^u-\hat\beta_s^u)ds- \int_t^T \wt{A}(s,t)\E_t[\hat\al_s^u-\hat\beta_s^u]ds\\&- \int_0^t\wt{B}(t,s)  \int_0^1 W(u,v)\hat\al^v_sdvds+ \int_0^t\wt{B}(t,s)  \int_0^1 W'(u,v)\hat\beta^v_sdvds \\
&- \hspace{-0.5mm}\int_t^T\wt{B}(s,t)  \hspace{-0.5mm}\int_0^1 W(u,v)\E_t[\hat\al^v_s]dvds
 + \hspace{-0.5mm}\int_t^T\wt{B}(s,t)  \hspace{-0.5mm}\int_0^1 W'(u,v)\E_t[\hat\beta^v_s]dvds.
\end{aligned}\ee
Rearranging \eqref{eq:fredholm-difference} and multiplying by  $(\hat\al_t^u-\hat\beta_t^u)$ yields,
\be\begin{aligned} \label{eq:fredholm-difference-squared} 
2\lam (\hat\al_t^u-\hat\beta_t^u)^2& +  (\hat\al_t^u-\hat\beta_t^u)\Big(\int_0^t \wt{A}(t,s)(\hat\al_s^u-\hat\beta_s^u)ds+ \int_t^T \wt{A}(s,t)\E_t[\hat\al_s^u-\hat\beta_s^u]ds\\&+ \int_0^t\wt{B}(t,s)  \int_0^1 W(u,v)\hat\al^v_sdvds- \int_0^t\wt{B}(t,s)  \int_0^1 W'(u,v)\hat\beta^v_sdvds \\
&+ \hspace{-1.1mm} \int_t^T\wt{B}(s,t)  \hspace{-1.1mm}\int_0^1 W(u,v)\E_t[\hat\al^v_s]dvds
 - \hspace{-1.1mm} \int_t^T\wt{B}(s,t)  \hspace{-1.1mm}\int_0^1 W'(u,v)\E_t[\hat\beta^v_s]dvds\Big)\\
 &= (\hat\al_t^u-\hat\beta_t^u)(b_t^u-b^{'u}_t)
\end{aligned}\ee
Next,  taking expectations, using the tower property of conditional expectation, integrating with respect to $t$, applying Fubini's theorem, and inserting the auxiliary term $\E\big[\int_0^T(\hat\al_t^u-\hat\beta_t^u) \big((\wt{\boldsymbol{B}}+\wt{\boldsymbol{B}}^*)(\boldsymbol{W}\hat\beta)(u)\big)(t)dt\big]$ and its additive inverse implies,
\be\begin{aligned} \label{eq:fredholm-difference-integrated} 
2\lam\E\Big[\int_0^T(\hat\al_t^u-\hat\beta_t^u)^2dt\Big]& +\E\Big[\int_0^T(\hat\al_t^u-\hat\beta_t^u) \big((\wt{\boldsymbol{A}}+\wt{\boldsymbol{A}}^*)(\hat\al^u-\hat\beta^u)\big)(t)dt\Big]\\&
+\E\Big[\int_0^T(\hat\al_t^u-\hat\beta_t^u) \big((\wt{\boldsymbol{B}}+\wt{\boldsymbol{B}}^*)(\boldsymbol{W}(\hat\al-\hat\beta))(u)\big)(t)dt\Big]\\&
+\E\Big[\int_0^T(\hat\al_t^u-\hat\beta_t^u) \big((\wt{\boldsymbol{B}}+\wt{\boldsymbol{B}}^*)((\boldsymbol{W}-\boldsymbol{W}')\hat\beta)(u)\big)(t)dt\Big]\\&
= \E\Big[\int_0^T(\hat\al_t^u-\hat\beta_t^u)(b_t^u-b^{'u}_t)dt\Big].
\end{aligned}\ee
Integration with respect to $u$ and Fubini's theorem allow us to apply Lemma~\ref{lemma:(A+A*)+(B+B*)W-nonnegative} to the first three terms on the left-hand side of \eqref{eq:fredholm-difference-integrated} to get
\be\begin{aligned} \label{eq:fredholm-difference-integrated2} 
&2 c_W\E\Big[\int_0^1\int_0^T(\hat\al_t^u-\hat\beta_t^u)^2dtdu\Big]
\\&+\E\Big[\int_0^1\int_0^T(\hat\al_t^u-\hat\beta_t^u) \big((\wt{\boldsymbol{B}}+\wt{\boldsymbol{B}}^*)((\boldsymbol{W}-\boldsymbol{W}')\hat\beta)(u)\big)(t)dtdu\Big]\\&
\leq \E\Big[\int_0^1\int_0^T(\hat\al_t^u-\hat\beta_t^u)(b_t^u-b^{'u}_t)dtdu\Big].
\end{aligned}\ee
Now, splitting up the second term on the left-hand side of \eqref{eq:fredholm-difference-integrated2} and applying the Cauchy-Schwarz inequality with underlying Hilbert space $L^2(\Omega\times [0,1]\times [0,T],\R)$ twice yields
\be\begin{aligned} \label{eq:fredholm-difference-rearranged2} 
&2 c_W\E\Big[\int_0^1\int_0^T(\hat\al_t^u-\hat\beta_t^u)^2dtdu\Big] \\&\leq\E\Big[\int_0^1\int_0^T(\hat\al_t^u)^2dtdu\Big]^\frac{1}{2}\E\Big[\int_0^1\int_0^T\big(\big((\wt{\boldsymbol{B}}+\wt{\boldsymbol{B}}^*)((\boldsymbol{W}-\boldsymbol{W}')\hat\beta)(u)\big)(t)\big)^2dtdu\Big]^\frac{1}{2}
\\&
\quad +\E\Big[\int_0^1\int_0^T(\hat\beta_t^u)^2dtdu\Big]^\frac{1}{2}\E\Big[\int_0^1\int_0^T\big(\big((\wt{\boldsymbol{B}}+\wt{\boldsymbol{B}}^*)((\boldsymbol{W}-\boldsymbol{W}')\hat\beta)(u)\big)(t)\big)^2dtdu\Big]^\frac{1}{2}
\\&
\quad + \E\Big[\int_0^1\int_0^T(\hat\al_t^u-\hat\beta_t^u)(b_t^u-b^{'u}_t)dtdu\Big].
\end{aligned}\ee
Using the inequality $\|(\wt{\boldsymbol{B}}+\wt{\boldsymbol{B}}^*)f\|_{L^2}^2\leq \|\wt{\boldsymbol{B}}+\wt{\boldsymbol{B}}^*\|_{\operatorname{op}}^2\|f\|_{L^2}^2$ for all $f\in L^2([0,T],\R)$, and the fact that $\|\wt{\boldsymbol{B}}\|_{\operatorname{op}}=\|\wt{\boldsymbol{B}}^*\|_{\operatorname{op}}$ implies
\be\begin{aligned} \label{eq:fredholm-difference-rearranged3} 
&2 c_W\E\Big[\int_0^1\int_0^T(\hat\al_t^u-\hat\beta_t^u)^2dtdu\Big] \\&\leq2\|\wt{\boldsymbol{B}}\|_{\operatorname{op}}\E\Big[\int_0^1\int_0^T(\hat\al_t^u)^2dtdu\Big]^\frac{1}{2}\E\Big[\int_0^1\int_0^T\big(((\boldsymbol{W}-\boldsymbol{W}')\hat\beta_t)(u)\big)^2dtdu\Big]^\frac{1}{2}
\\&
\quad +2\|\wt{\boldsymbol{B}}\|_{\operatorname{op}}\E\Big[\int_0^1\int_0^T(\hat\beta_t^u)^2dtdu\Big]^\frac{1}{2}\E\Big[\int_0^1\int_0^T\big(((\boldsymbol{W}-\boldsymbol{W}')\hat\beta_t)(u)\big)^2dtdu\Big]^\frac{1}{2}
\\&
\quad + \E\Big[\int_0^1\int_0^T(\hat\al_t^u-\hat\beta_t^u)(b_t^u-b^{'u}_t)dtdu\Big].
\end{aligned}\ee
Applying Fubini's theorem and then the same argument for the operator $(\boldsymbol{W}-\boldsymbol{W}')$ on $L^2([0,1],\R)$ shows that
\be\begin{aligned} \label{eq:fredholm-difference-bounded} 
&2 c_W\E\Big[\int_0^1\int_0^T(\hat\al_t^u-\hat\beta_t^u)^2dtdu\Big] \\&\leq 2\|\wt{\boldsymbol{B}}\|_{\operatorname{op}}\|\boldsymbol{W}-\boldsymbol{W}'\|_{\operatorname{op}}\Big(\E\Big[\int_0^1\int_0^T(\hat\al_t^u)^2dtdu\Big]^\frac{1}{2}\E\Big[\int_0^1\int_0^T(\hat\beta_t^u)^2dtdu\Big]^\frac{1}{2}
\\&
 \quad+\E\Big[\int_0^1\int_0^T(\hat\beta_t^u)^2dtdu\Big]\Big)
+ \E\Big[\int_0^1\int_0^T(\hat\al_t^u-\hat\beta_t^u)(b_t^u-b^{'u}_t)dtdu\Big],
\end{aligned}\ee
so that an application of Young's inequality and division by $2 c_W$ yield
\be\begin{aligned} \label{eq:fredholm-difference-bounded2} 
&\E\Big[\int_0^1\int_0^T(\hat\al_t^u-\hat\beta_t^u)^2dtdu\Big] \\&\leq\frac{1}{2 c_W}\|\wt{\boldsymbol{B}}\|_{\operatorname{op}}\|\boldsymbol{W}-\boldsymbol{W}'\|_{\operatorname{op}}\Big(\E\Big[\int_0^1\int_0^T(\hat\al_t^u)^2dtdu\Big]+3\E\Big[\int_0^1\int_0^T(\hat\beta_t^u)^2dtdu\Big]\Big)
\\&\quad +\frac{1}{2 c_W} \E\Big[\int_0^1\int_0^T(\hat\al_t^u-\hat\beta_t^u)(b_t^u-b^{'u}_t)dtdu\Big].
\end{aligned}\ee
Finally, separating $\hat\al_t^u$ and $\hat\beta_t^u$ in the last term on the right-hand side of \eqref{eq:fredholm-difference-bounded2} and applying the Cauchy-Schwarz inequality twice completes the proof.
\end{proof}

We proceed with the proofs of Section~\ref{subsection:convergence-sampled-games}. The proof Theorem~\ref{thm:convergence-given-graphs} from Section~\ref{subsection:convergence-general-games} is given at the end of this section. We first prove the auxiliary Lemma~\ref{lemma:aux}.

\begin{proof}[Proof of Lemma~\ref{lemma:aux}]
In the setup of Theorem~\ref{thm:convergence-sampled-graphs}, there exists a $c_0>0$ independent of $N$ such that Assumption~\ref{assum:nonnegative-definite-graph} is satisfied for every matrix $w^N\in[0,1]^{N\times N}$ with zero diagonal entries and every $N\in\N$. Therefore, it holds that 
\be
\langle f, \lambda f +\wt{\boldsymbol{A}}f+ \frac{1}{N}w^N\cdot\wt{\boldsymbol{B}}f-c_0 f\rangle_{L^2,N}\geq 0,\quad \text{for all } f\in L^2([0,T],\R^N).
\ee   
It follows for all $f\in L^2([0,T],\R^N)$ that  
\be\label{eq:aux-lemma1}
\langle f, \lambda f +\wt{\boldsymbol{A}}f+ \frac{1}{N}\kappa_N^{-1}s^N\cdot\wt{\boldsymbol{B}}f\rangle_{L^2,N}\geq c_0\langle f,f\rangle_{L^2,N}-\langle f,\frac{1}{N}(w^N-\kappa_N^{-1}s^N)\cdot\wt{\boldsymbol{B}}f\rangle_{L^2,N}.
\ee   
Next, let $\|\cdot\|_2$ denote the spectral norm of a matrix. Then by the Cauchy-Schwarz inequality, the consistency of $\|\cdot\|_2$ with the Euclidean norm,  and the fact that $\smash{\|\wt{\boldsymbol{B}}f\|_{L^2,N}\leq \|\wt{\boldsymbol{B}}\|_{\operatorname{op}}\|f\|_{L^2,N}}$ for all $f\in L^2([0,T],\R^N)$, it holds that
\be\begin{aligned}\label{eq:aux-lemma2}
\Big|\langle f,\frac{1}{N}(w^N-\kappa_N^{-1}s^N)\cdot\wt{\boldsymbol{B}}f\rangle_{L^2,N}\Big|
&\leq \| f\|_{L^2,N}\cdot\frac{1}{N}\|w^N-\kappa_N^{-1}s^N\|_2\cdot\| \wt{\boldsymbol{B}}f\|_{L^2,N}\\
&\leq \frac{1}{N}\|w^N-\kappa_N^{-1}s^N\|_2\cdot\|\wt{\boldsymbol{B}}\|_{\operatorname{op}}\cdot\| f\|^2_{L^2,N}.
\end{aligned}\ee
Now it follows from the proof of Theorem~1 in \citet{avella2018centrality} that for any $0<\dl<e^{-1}$ there is an $\smash{\wt{N}_\dl\in\N}$ such that for all $\smash{N\geq \wt{N}_\dl}$ it holds that 
\be\label{eq:aux-lemma3}
\frac{1}{N}\|w^N-\kappa_N^{-1}s^N\|_2\leq\sqrt{\frac{\kappa_N^{-1}\log(2N/\dl)}{N}},
\ee
with $\Q$-probability at least $1-\dl$ and $1-2\dl$ in case \ref{S3} and case \ref{S4}, respectively. Since the right-hand side of \eqref{eq:aux-lemma3} converges to $0$ as $N\to\infty$ by assumption, it follows from 
\eqref{eq:aux-lemma1} and \eqref{eq:aux-lemma2}   that there is an $\smash{N_\dl\geq \wt{N}_\dl}$ such that inequality \eqref{eq:aux-lemma} is satisfied with $\Q$-probability at least $1-\dl$ in case \ref{S3} and $1-2\dl$ in case \ref{S4} for all $N\geq N_\dl$.
\end{proof}

The following lemma gives a uniform bound on the sampled $N$-player game Nash equilibria, and is needed for the proof of Theorem~\ref{thm:convergence-sampled-graphs} as well.

\begin{lemma}\label{lemma:uniform-bound}
Under the assumptions of Theorem~\ref{thm:convergence-sampled-graphs}, the sampled $N$-player game Nash equilibria $\hat{\al}^N$ are uniformly bounded in $N\in\N$. Namely, in cases \ref{S1}-\ref{S2}, there exists a constant $C_{1,2}>0$ such that
$$ 
\sup_{N\in\N} \E\Big[\int_0^1\int_0^T(\hat\al^{u,N}_{t,\operatorname{step}})^2dtdu\Big]\leq C_{1,2}. 
$$
In cases \ref{S3}-\ref{S4}, there exists a constant $C_{3,4}>0$ such that for every $0<\dl<e^{-1}$  
$$ 
\sup_{N\geq N_\dl} \E\Big[\int_0^1\int_0^T(\hat\beta^{u,N}_{t,\operatorname{step}})^2dtdu\Big]\leq C_{3,4},
$$
with $\Q$-probability at least $1-\dl$ in case \ref{S3} and $1-2\dl$  in case \ref{S4}.
\end{lemma}
\begin{proof}[Proof of Lemma~\ref{lemma:uniform-bound}]
First, fix one of the cases \ref{S1}-\ref{S2} from Definition~\ref{def:sampling} and denote by $w^N$ the sampled probability matrix.
Then, since Assumptions~\ref{assum:A-B-C-graph} and \ref{assum:nonnegative-definite-graph} are satisfied, the sampled $N$-player game admits a unique Nash equilibrium $\hat{\al}^N$ given by Corollary~\ref{cor:finite-player-equilibrium}. In order to uniformly bound the Nash equilibria $\hat{\al}^N$ over all $N\in\N$, recall that by Proposition~\ref{prop:fredholm-foc-graph} $\hat\al^N$ satisfies
\be \begin{aligned} \label{eq:uniform-bound-1}
2 \lambda \hat{\al}_t^{i,N} &+ \int_0^t   \wt{A}(t,s) \hat{\al}^{i,N}_s ds +  \int_t^T \wt{A}(s,t)  \E_t[\hat{\al}^{i,N}_s] ds+ \int_0^t   \wt{B}(t,s) \frac{1}{N}\sum_{j=1}^Nw_{ij}^N\hat{\al}^{j,N}_s ds \\  & + \int_t^T \wt{B}(s,t)   \frac{1}{N}\sum_{j=1}^Nw_{ij}^N\E_t[\hat{\al}^{j,N}_s] ds = b^{i,N}_t,  \quad d\P \otimes dt \text{-a.e.~on } \Omega \times [0,T],
\end{aligned}\ee 
for all $i\in\{1,\ldots, N\}$. Next, using vector notation in \eqref{eq:uniform-bound-1}, multiplying by $(\hat\al_t^N)^\top$, taking expectations, using the tower property of conditional expectation, integrating with respect to $t$, and applying Fubini's theorem yields

\be\begin{aligned}\label{eq:uniform-bound-2}
2\lam\E\Big[\int_0^T\|\hat\al_t^N\|^2dt\Big]& +\E\Big[\int_0^T(\hat\al_t^N)^\top \big((\wt{\boldsymbol{A}}+\wt{\boldsymbol{A}}^*)\hat\al^N\big)(t)dt\Big]\\&
+\E\Big[\int_0^T(\hat\al_t^N)^\top \frac{1}{N}w^N\big((\wt{\boldsymbol{B}}+\wt{\boldsymbol{B}}^*)\hat\al^N\big)(t)dt\Big]\\&
= \E\Big[\int_0^T(\hat\al_t^N)^\top b_t^Ndt\Big].
\end{aligned}\ee
Since Assumption~\ref{assum:nonnegative-definite-graph} holds for a $c_0>0$ independent of $N$, it follows  that
\be\label{eq:uniform-bound-3}
2c_0\E\Big[\int_0^T\|\hat\al_t^N\|^2dt\Big]\leq \E\Big[\int_0^T(\hat\al_t^N)^\top b_t^Ndt\Big].
\ee
Dividing by $N$ and using the notation from \eqref{eq:step-function} yields 
\be\label{eq:uniform-bound-4}
2c_0\E\Big[\int_0^T\int_0^1(\hat\al_{t,\operatorname{step}}^{u,N})^2dudt\Big]\leq \E\Big[\int_0^T\int_0^1\hat\al_{t,\operatorname{step}}^{u,N} b_{t,\operatorname{step}}^{u,N}dudt\Big].
\ee
Therefore, the Cauchy-Schwarz inequality implies that
\be\label{eq:uniform-bound-5}
2c_0\E\Big[\int_0^T\int_0^1(\hat\al_{t,\operatorname{step}}^{u,N})^2dudt\Big]\leq \E\Big[\int_0^T\int_0^1(\hat\al_{t,\operatorname{step}}^{u,N})^2dudt\Big]^\frac{1}{2}\E\Big[\int_0^T\int_0^1(b_{t,\operatorname{step}}^{u,N})^2dudt\Big]^\frac{1}{2}.
\ee
It follows that
\be\label{eq:uniform-bound-6}
4c_0^2\E\Big[\int_0^T\int_0^1(\hat\al_{t,\operatorname{step}}^{u,N})^2dudt\Big]\leq \E\Big[\int_0^T\int_0^1(b_{t,\operatorname{step}}^{u,N})^2dudt\Big],
\ee
which yields a uniform bound in $N\in\N$ on the left-hand side of \eqref{eq:uniform-bound-6}, since the right-hand side of \eqref{eq:uniform-bound-6} converges by assumption of Theorem~\ref{thm:convergence-sampled-graphs}. Setting
$$
C_{1,2}:=\frac{c_0^{-2}}{4}\sup_{N\in\N} \E\Big[\int_0^1\int_0^T(b^{u,N}_{t,\operatorname{step}})^2dtdu\Big]
$$
completes the proof for cases \ref{S1}-\ref{S2}.

Second, fix one of the sampling procedures \ref{S3}-\ref{S4} from Definition~\ref{def:sampling} and denote by $s^N$ the adjacency matrix of the sampled simple graph. We can not proceed as for the cases \ref{S1}-\ref{S2} by directly exploiting Assumption~\ref{assum:nonnegative-definite-graph}, since the entries of $\kappa_N^{-1}s^N$ might be larger than $1$, especially for sparse density parameter sequences. However, by Lemma~\ref{lemma:aux}, for every $0<\dl<e^{-1}$ Assumption~\ref{assum:nonnegative-definite-graph} with $w^N$ replaced by $\kappa_N^{-1}s^N$ and $c_0$ replaced by $c_0/2$ is satisfied with $\Q$-probability at least $1-\dl$ in case \ref{S3} and $1-2\dl$ in case \ref{S4} for all $N\geq N_\dl$.
Therefore, also recalling Remark \ref{rem:correspondence-with-kappa_N},
we can proceed analogously to cases \ref{S1}-\ref{S2} above with $w^N$ replaced by $\kappa_N^{-1}s^N$ in \eqref{eq:uniform-bound-1} and $c_0$ replaced by $c_0/2$. Setting
$$
C_{3,4}:=c_0^{-2}\sup_{N\in\N} \E\Big[\int_0^1\int_0^T(b^{u,N}_{t,\operatorname{step}})^2dtdu\Big]
$$
then completes the proof for cases \ref{S3}-\ref{S4}.
\end{proof}

We can now prove the main theorem of Section~\ref{subsection:convergence-sampled-games}.
\begin{proof}[Proof of Theorem~\ref{thm:convergence-sampled-graphs}]
Let the assumptions of Theorem~\ref{thm:convergence-sampled-graphs} be satisfied. 
First, fix one of the sampling procedures \ref{S1}-\ref{S2} from Definition~\ref{def:sampling} and denote by $w^N$ the sampled probability matrix. Then,  under the correspondence of Proposition~\ref{prop:correspondence-finite-infinite-game}, the unique Nash equilibrium $\hat{\al}^N\in\mathcal{A}^N$ corresponds to the step function strategy profile $\hat{\al}_{\operatorname{step}}^N\in\mathcal{A}^\infty$ defined in \ref{eq:step-function}, which is a Nash equilibrium of a graphon game with underlying step graphon $W^N$ corresponding to $w^N$ as in \eqref{eq:step-graphon} and step function noise process $b^{N}_{\operatorname{step}}$ corresponding to $(b^{1,N},\ldots, b^{N,N})$ as in \eqref{eq:step-function}.
Proposition~\ref{prop:graphon-game-continuity} implies that
\be\begin{aligned} \label{eq:convergence-1} 
&\E\Big[\int_0^1\int_0^T(\hat\al_t^u-\hat\al_{t,\operatorname{step}}^{u,N})^2dtdu\Big] \\&\leq\frac{1}{2 c_W}\|\wt{\boldsymbol{B}}\|_{\operatorname{op}}\|\boldsymbol{W}-\boldsymbol{W}^N\|_{\operatorname{op}}\Big(\E\Big[\int_0^1\int_0^T(\hat\al_t^u)^2dtdu\Big]+3\E\Big[\int_0^1\int_0^T(\hat\al_{t,\operatorname{step}}^{u,N})^2dtdu\Big]\Big)
\\&+\frac{1}{2 c_W} \E\Big[\int_0^1\int_0^T(b_t^u-b^{u,N}_{t,\operatorname{step}})^2dtdu\Big]^\frac{1}{2}\Big(\E\Big[\int_0^1\int_0^T(\hat\al_t^u)^2dtdu\Big]^\frac{1}{2}\\&+\E\Big[\int_0^1\int_0^T(\hat\al_{t,\operatorname{step}}^{u,N})^2dtdu\Big]^\frac{1}{2}\Big).
\end{aligned}\ee
Now inequality \eqref{eq:convergence-1}, Theorem~\ref{thm:avella}, Lemma~\ref{lemma:uniform-bound}, and \eqref{eq:noise-convergence} imply statements \eqref{eq:convergence-theorem-1} for \ref{S1} and \eqref{eq:convergence-theorem-2} for  \ref{S2}, where for $ \smash{C_{\hat\al}:=\E[\int_0^1\int_0^T(\hat\al_t^u)^2dtdu]<\infty}$ the two constants $K_1$, $K_2<\infty$ therein are given by 
$$
K_1=K_2=\frac{1}{2 c_W}\max\big\{\|\wt{\boldsymbol{B}}\|_{\operatorname{op}}(C_{\hat\al}+3C_{1,2}),\ \sqrt{C_{\hat\al}}+\sqrt{C_{1,2}}\big\}.
$$
In particular, it follows in case \ref{S1} that 
$$
\E\Big[\int_0^1\int_0^T(\hat\al_t^u-\hat\al^{u,N}_{t,\operatorname{step}})^2dtdu\Big]\xrightarrow{N\to\infty}0.
$$
In case \ref{S2}, assume that there exists an $\eps>0$ such that
\be\label{eq:2delta}
\bar\dl:=\Q\left(\limsup_{N\to\infty}\Big\{\E\Big[\int_0^1\int_0^T(\hat\al_t^u-\hat\al^{u,N}_{t,\operatorname{step}})^2dtdu\Big]>\eps\Big\}\right)>0.
\ee
Choose $0<\dl<\min\{\bar\dl,e^{-1}\}$. By \eqref{eq:convergence-theorem-2}, since $\rho(N)$ in \eqref{eq:rho(N)} and $\pi(N)$ in \eqref{eq:noise-convergence} converge to 0 as $N\to\infty$, there exists an $N(\eps)\in\N$ such that
\be
\Q\left(\Big\{\E\Big[\int_0^1\int_0^T(\hat\al_t^u-\hat\al^{u,N}_{t,\operatorname{step}})^2dtdu\Big]\leq\eps\Big\}\right)\geq 1-\dl
\ee
for all $N\geq N(\eps)$, which contradicts \eqref{eq:2delta}, and thus yields
$$\E\Big[\int_0^1\int_0^T(\hat\al_t^u-\hat\al^{u,N}_{t,\operatorname{step}})^2dtdu\Big]\xrightarrow{N\to\infty}0,\quad \Q\text{-almost surely.}
$$
Second, fix one of the sampling procedures \ref{S3}-\ref{S4} from Definition~\ref{def:sampling} and denote by $s^N$ the adjacency matrix of the sampled simple graph. Then, for any $0<\dl<e^{-1}$ there exists an $N_\dl\in\N$ such that for all $N\geq N_\dl$ the sampled $N$-player game defined in \eqref{eq:J_0^iN-rewritten} with modification \eqref{eq:aggregate-sparse} admits a unique Nash equilibrium $\hat\beta^N\in\mathcal{A}^N$ corresponding to  $\kappa_N^{-1} s^N$ with $\Q$-probability at least $1-\dl$ in case \ref{S3} and $1-2\dl$ in case \ref{S4} (by Corollary~\ref{cor:finite-player-equilibrium} and Lemma~\ref{lemma:aux}). Under the correspondence of Proposition~\ref{prop:correspondence-finite-infinite-game} and Remark \ref{rem:correspondence-with-kappa_N}, $\hat{\beta}^N$  corresponds to the step function strategy profile $\smash{\hat{\beta}_{\operatorname{step}}^N\in\mathcal{A}^\infty}$ defined in \ref{eq:step-function}, which is a Nash equilibrium of a graphon game with underlying step graphon $\smash{\kappa_N^{-1}S^N}$ corresponding to $\smash{\kappa_N^{-1}s^N}$ as in \eqref{eq:step-graphon} and step function noise process $b^{N}_{\operatorname{step}}$ corresponding to $(b^{1,N},\ldots, b^{N,N})$ as in \eqref{eq:step-function}.
Now Proposition~\ref{prop:graphon-game-continuity} implies that
\be\begin{aligned} \label{eq:convergence-2} 
&\E\Big[\int_0^1\int_0^T(\hat\al_t^u-\hat\al_{t,\operatorname{step}}^{u,N})^2dtdu\Big] \\&\leq\frac{1}{2 c_W}\|\wt{\boldsymbol{B}}\|_{\operatorname{op}}\|\boldsymbol{W}-\kappa_N^{-1}\boldsymbol{S}^N\|_{\operatorname{op}}\Big(\E\Big[\int_0^1\int_0^T(\hat\al_t^u)^2dtdu\Big]+3\E\Big[\int_0^1\int_0^T(\hat\al_{t,\operatorname{step}}^{u,N})^2dtdu\Big]\Big)
\\&+\frac{1}{2 c_W} \E\Big[\int_0^1\int_0^T(b_t^u-b^{u,N}_{t,\operatorname{step}})^2dtdu\Big]^\frac{1}{2}\Big(\E\Big[\int_0^1\int_0^T(\hat\al_t^u)^2dtdu\Big]^\frac{1}{2}\\&+\E\Big[\int_0^1\int_0^T(\hat\al_{t,\operatorname{step}}^{u,N})^2dtdu\Big]^\frac{1}{2}\Big).
\end{aligned}\ee
Inequality \eqref{eq:convergence-2}, Theorem~\ref{thm:avella}, Lemma~\ref{lemma:uniform-bound}, and \eqref{eq:noise-convergence} imply statements \eqref{eq:convergence-theorem-3} and \eqref{eq:convergence-theorem-4}, where the two constants $K_3$, $K_4<\infty$ therein are given by 
$$
K_3=K_4=\frac{1}{2 c_W}\max\big\{\|\wt{\boldsymbol{B}}\|_{\operatorname{op}}(C_{\hat\al}+3C_{3,4}),\ \sqrt{C_{\hat\al}}+\sqrt{C_{3,4}}\big\}.
$$
Finally, we want to prove almost sure convergence in the cases \ref{S3}-\ref{S4}. Assume that there exists an $\eps>0$ such that
\be\label{eq:delta-bar}
\bar\dl:=\Q\left(\limsup_{N\to\infty}\Big\{\E\Big[\int_0^1\int_0^T(\hat\al_t^u-\hat\beta^{u,N}_{t,\operatorname{step}})^2dtdu\Big]>\eps\Big\}\right)>0.
\ee
Choose $0<\dl<\min\{\bar\dl/2,e^{-1}\}$. By \eqref{eq:convergence-theorem-3} and \eqref{eq:convergence-theorem-4}, respectively, since $\rho'(N)$ in \eqref{eq:rho'(N)} and $\pi(N)$ in \eqref{eq:noise-convergence} converge to 0 as $N\to\infty$, because $\tfrac{\log N}{\kappa_NN}$ does, there exists an $N(\eps)\in\N$ such that
\be
\Q\left(\Big\{\E\Big[\int_0^1\int_0^T(\hat\al_t^u-\hat\beta^{u,N}_{t,\operatorname{step}})^2dtdu\Big]\leq\eps\Big\}\right)\geq 1-2\dl>1-\bar\dl
\ee
for all $N\geq N(\eps)$, which contradicts \eqref{eq:2delta} and thus yields
$$\E\Big[\int_0^1\int_0^T(\hat\al_t^u-\hat\beta^{u,N}_{t,\operatorname{step}})^2dtdu\Big]\xrightarrow{N\to\infty}0,\quad \Q\text{-almost surely.}
$$
\end{proof}
Finally, we proceed with the proof of the results from Section~\ref{subsection:convergence-general-games}. 
\begin{proof}[Proof of Theorem~\ref{thm:convergence-given-graphs}]
    We first note that the Nash equilibria $\hat{\al}^N$ of the finite-player games are uniformly bounded in $N\in\N$. Namely, there exists a constant $C_0<\infty$ such that
\be \label{eq:uniform-bound-C0}
\sup_{N\in\N} \E\Big[\int_0^1\int_0^T(\hat\al^{u,N}_{t,\operatorname{step}})^2dtdu\Big]\leq C_0. 
\ee
This can be seen in the exact same way as in the first part of the proof of Lemma~\ref{lemma:uniform-bound}, since Assumption~\ref{assum:nonnegative-definite-graph} is satisfied for a $c_0>0$ independent of $N$. 

Next,  under the correspondence of Proposition~\ref{prop:correspondence-finite-infinite-game}, the unique Nash equilibrium $\hat{\al}^N\in\mathcal{A}^N$ corresponds to the step function strategy profile $\hat{\al}_{\operatorname{step}}^N\in\mathcal{A}^\infty$ defined in \ref{eq:step-function}, which is a Nash equilibrium of a graphon game with underlying step graphon $W^N$ corresponding to $w^N$ as in \eqref{eq:step-graphon} and step function noise process $b^{N}_{\operatorname{step}}$ corresponding to $(b^{1,N},\ldots, b^{N,N})$ as in \eqref{eq:step-function}.
Proposition~\ref{prop:graphon-game-continuity} implies that
inequality \eqref{eq:convergence-1} holds in the setting of Theorem~\ref{thm:convergence-given-graphs} as well. Therefore, \eqref{eq:noise-convergence-given} and Remark \ref{rem:relation-cut-op} imply inequality  \eqref{eq:convergence-theorem}, where for $ \smash{C_{\hat\al}:=\E[\int_0^1\int_0^T(\hat\al_t^u)^2dtdu]<\infty}$ the constant $K<\infty$ therein is given by 
$$
K=\frac{1}{2 c_W}\max\big\{\sqrt{8}\,\|\wt{\boldsymbol{B}}\|_{\operatorname{op}}(C_{\hat\al}+3C_0),\ \sqrt{C_{\hat\al}}+\sqrt{C_0}\big\}.
$$
\end{proof}


\appendix
\section{Appendix}\label{appendix}
\subsection{Example on systemic risk}
In the first part of the Appendix, we study an example which complements our results from Section~\ref{subsec:systemic-risk}. It focuses on a canonical special case of the model introduced therein, where there is exactly one repayment after a fixed time $0\leq\tau\leq T$. We verify Assumption~\ref{assum:nonnegative-definite} under different conditions on the interaction matrix $w^{sys}$ and the coefficients $\kappa^i,\eps^i,c^i$ defining the cost functionals.
\begin{example}\label{ex:systemic}
Fix $0\leq \tau\leq T$ and assume that $\nu^i=\dl_0-\dl_{\tau}$ for all $i=1,\ldots, N$. Recalling \eqref{eq:systemicGii}, it follows for all $i,j=1,\ldots, N$ that
\be
G_1^{ij}(t,s)=\dl_{ij}\mathds{1}_{\{0<t-s<\tau\}}\quad\text{and}\quad G_2^{ij}(t,s)=w^{sys}_{ij} \mathds{1}_{\{0<t-s<\tau\}},
\ee
and therefore that $G^{ij}(t,s)=(\dl_{ij},w^{sys}_{ij})^\top \mathds{1}_{\{0<t-s<\tau\}}$. Thus, for $i,j=1,\dots, N$, the kernel $\Gamma_{ij}$ in Lemma~\ref{lemma:Volterra-game}  is given by 
\be
\begin{aligned}
\Gamma_{ij}(t,s)=&\int_{s\vee t}^T \eps^i (1-w^{sys}_{ii})(\dl_{ij}-w^{sys}_{ij})\mathds{1}_{\{0<u-t<\tau\}} \mathds{1}_{\{0<u-s<\tau\}} du \label{eq:systemic-example-gamma}\\[0.5ex]
& + c^i (1-w^{sys}_{ii})(\dl_{ij}-w^{sys}_{ij})\mathds{1}_{\{0<T-t<\tau\}} \mathds{1}_{\{0<T-s<\tau\}}\\[1.5ex]
& +\kappa^i(\dl_{ij}-w^{sys}_{ij})\mathds{1}_{\{0<t-s<\tau\}}+\kappa^i\dl_{ij}(1-w^{sys}_{ii})\mathds{1}_{\{0<s-t<\tau\}}.
\end{aligned}
\ee
We first focus on the first two terms of \eqref{eq:systemic-example-gamma}, which we denote by $\Gamma^1=(\Gamma^\mathds{1}_{ij})_{i,j=1}^N$ and $\Gamma^2=(\Gamma^2_{ij})_{i,j=1}^N$, respectively. Let $\Delta w^{sys}\vcentcolon=\operatorname{diag}(w^{sys}_{11},\ldots,w^{sys}_{NN})$ denote the diagonal matrix that has as entries the diagonal entries of $w^{sys}$. Define $\Delta c\vcentcolon=\operatorname{diag}(c^1,\ldots,c^N)$, $\Delta \eps\vcentcolon=\operatorname{diag}(\eps^1,\ldots,\eps^N)$ and $\Delta \kappa\vcentcolon=\operatorname{diag}(\kappa^1,\ldots,\kappa^N)$. Let $f\in L^2([0,T],\R^N)$. Then it follows from Fubini's theorem that
\be\begin{aligned}\label{eq:systemic-example-gamma1}
    &\langle f, \boldsymbol{\Gamma^1}f\rangle_{L^2,N}\\&=\int_0^T\int_0^Tf(t)^\top \int_{s\vee t}^T \Delta\eps (\operatorname{Id}^N-\Delta w^{sys})(\operatorname{Id}^N-w^{sys})\mathds{1}_{\{0<u-t<\tau\}} \mathds{1}_{\{0<u-s<\tau\}} du f(s)dsdt\\
    &=\int_0^T\left(\int_{u-\tau}^u f(t)dt\right)^\top  \Delta\eps (\operatorname{Id}^N-\Delta w^{sys})(\operatorname{Id}^N-w^{sys}) \left(\int_{u-\tau}^u f(s)ds\right) du,
\end{aligned}\ee
and analogously that 
\be\begin{aligned}\label{eq:systemic-example-gamma2}
    &\langle f, \boldsymbol{\Gamma^2}f\rangle_{L^2,N}\\&=\int_0^T\int_0^Tf(t)^\top \Delta c (\operatorname{Id}^N-\Delta w^{sys})(\operatorname{Id}^N-w^{sys})\mathds{1}_{\{0<T-t<\tau\}} \mathds{1}_{\{0<T-s<\tau\}} f(s)dsdt\\
    &=\left(\int_{T-\tau}^T  f(t)dt\right)^\top  \Delta c (\operatorname{Id}^N-\Delta w^{sys})(\operatorname{Id}^N-w^{sys}) \left(\int_{T-\tau}^T f(s)ds\right).
\end{aligned}
\ee
Now the two terms \eqref{eq:systemic-example-gamma1} and \eqref{eq:systemic-example-gamma2} are nonnegative, whenever the two matrices $\Delta\eps(\operatorname{Id}^N-\Delta w^{sys})(\operatorname{Id}^N-w^{sys})$ and $\Delta c(\operatorname{Id}^N-\Delta w^{sys})(\operatorname{Id}^N-w^{sys})$ are both nonnegative definite.
Finally, we focus on the third and fourth term of \eqref{eq:systemic-example-gamma}, which we combine and denote by $\Gamma^3=(\Gamma^3_{ij})_{i,j=1}^N$. We get for $f\in L^2([0,T],\R^N)$ the following bound,
\be\begin{aligned}\label{eq:systemic-example-gamma3}
\langle f, \boldsymbol{\Gamma^3}f\rangle_{L^2,N}&\geq - \int_0^T\int_0^T\big|f(t)^\top \Delta\kappa\big(2\operatorname{Id}^N-w^{sys}-\Delta w^{sys}\big) \mathds{1}_{\{0<s-t<\tau\}}f(s)\big|dsdt \\[1ex]
&\geq -\int_0^T\int_0^T\mathds{1}_{\{t>s\}}\big\|f(t)\big\|\,\big\| \Delta\kappa\big(2\operatorname{Id}^N-w^{sys}-\Delta w^{sys}\big) \big\| \big\|f(s)\big\|dsdt \\[1ex]
&=-\frac{1}{2} \int_0^T\int_0^T\big\|f(t)\big\|\,\big\| \Delta\kappa\big(2\operatorname{Id}^N-w^{sys}-\Delta w^{sys}\big) \big\| \big\|f(s)\big\|dsdt \\[1ex]
&\geq-\frac{1}{4} \int_0^T\int_0^T\Big(\big\|f(t)\big\|^2+\big\|f(s)\big\|^2\Big)\,\big\| \Delta\kappa\big(2\operatorname{Id}^N-w^{sys}-\Delta w^{sys}\big) \big\| dsdt\\[1ex]
&=-\frac{T}{2}\, \big\| \Delta\kappa\big(2\operatorname{Id}^N-w^{sys}-\Delta w^{sys}\big) \big\|\, \langle f, f\rangle_{L^2,N},
\end{aligned}\ee
where the second inequality follows from the submultiplicativity of the Frobenius norm, the first equality follows from Remark \ref{rem:nonnegative-operators} and the third inequality follows from Young's inequality. Therefore, recalling that $p^i=1/2$ for all $i=1,\ldots, N$ and due to \eqref{eq:systemic-example-gamma3}, there exists a $c_0>0$ with
\be
\langle f,  \boldsymbol{\Gamma^3}f+2\Pi f -c_0 f\rangle_{L^2,N}\geq 0,\quad \text{for all } f\in L^2([0,T],\R^N),
\ee
whenever it holds that 
\be\label{eq:kappa-T-condition}
\frac{T}{2}\big\| \Delta\kappa\big(2\operatorname{Id}^N-w^{sys}-\Delta w^{sys}\big) \big\|<1.
\ee
Altogether, it follows from Lemma~\ref{lemma:Volterra-game}  that Assumption~\ref{assum:nonnegative-definite} is satisfied, whenever the matrices $\Delta\eps(\operatorname{Id}^N-\Delta w^{sys})(\operatorname{Id}^N-w^{sys})$ and $\Delta c(\operatorname{Id}^N-\Delta w^{sys})(\operatorname{Id}^N-w^{sys})$ are nonnegative definite and \eqref{eq:kappa-T-condition} holds. In this case, Theorem~\ref{thm:finite-player-equilibrium} can be applied and yields explicit operator formulas for the Nash equilibrium. In particular, this allows for arbitrary interaction matrices $w^{sys}\in\R_+^{N\times N}$, if the coefficients $\kappa^i,\eps^i,c^i$ defining the cost functional in \eqref{eq:systemicJ^i} are chosen suitably. In the case where $w^{sys}$ is symmetric, the aforementioned matrices are nonnegative definite whenever the largest eigenvalue of $w^{sys}$ is smaller or equal to 1. 
\end{example}
\begin{remark} Note that in the case without repayments corresponding to $\tau>T$, the analysis of the term $\langle f, \boldsymbol{\Gamma^3}f\rangle_{L^2,N}$ in Example \ref{ex:systemic} considerably simplifies. Namely, in this case, by a similar argument to the ones given for $\boldsymbol{\Gamma^1}$ and $\boldsymbol{\Gamma^2}$, it holds that $\langle f, \boldsymbol{\Gamma^3}f\rangle_{L^2,N}\geq 0$ for all $f\in L^2([0,T],\R^N)$ whenever the matrix $\Delta\kappa(2\operatorname{Id}^N-w^{sys}-\Delta w^{sys})$ is nonnegative definite, yielding sufficient conditions for Assumption~\ref{assum:nonnegative-definite} to be satisfied that are independent of the time horizon $T$. In particular,  if $w^{sys}$ is symmetric with its largest eigenvalue being smaller or equal to 1, Assumption~\ref{assum:nonnegative-definite} is always satisfied without additional assumptions on the coefficients $\kappa^i,\eps^i,c^i$.
\end{remark}

\subsection{Additional proofs of Section~\ref{sec:infinite-player-game}}
In the second part of the Appendix, we give for the sake of completeness several proofs of statements from Section~\ref{sec:infinite-player-game}.
We start with the following lemma, which is needed to prove Proposition~\ref{prop:fredholm-foc-infinite}.
\begin{lemma}\label{lemma:concave-infinite}
Let Assumption~\ref{assum:A-B-C-graphon} hold and assume that $\wt{\boldsymbol A}$ is nonnegative definite. Fix $u\in [0,1]$ and a family of strategies $(\al^{v})_{v\neq u}\in\mathcal{A}^\infty$. Then, the objective functional $\al^u\mapsto J^{u,W}(\al^{u}; (\al^{v})_{v\neq u})$ in \eqref{eq:J^uW} is strictly concave on the set of $\F$-progressively measurable processes in $L^2(\Omega\times [0,T],\R)$.
\end{lemma}
\begin{proof}
    Fix $u\in [0,1]$ and $(\al^{v})_{v\neq u}\in\mathcal{A}^\infty$. Then, for any $\F$-progressively measurable $\al^u, \beta^u \in L^2(\Omega\times [0,T],\R)$ such that $\al^u \neq \beta^u$, $d\P \otimes dt$-a.e.~on $\Omega \times [0,T]$, and for all $\varepsilon \in (0,1)$, a direct computation yields 
\begin{equation}
    \begin{aligned}
        & J^{u,W}(\varepsilon \al^{u} + (1-\varepsilon) \beta^{u} ; (\al^{v})_{v\neq u}) - \varepsilon J^{u,W}(\al^{u};  (\al^{v})_{v\neq u}) - (1-\varepsilon) J^{u,W}(\beta^{u};  (\al^{v})_{v\neq u}) \\[0.5ex]
        & = \varepsilon (1-\varepsilon) \, \E\left[ \langle \al^u - \beta^u ,\wt{\boldsymbol{A}} (\al^u - \beta^u) +\lambda(\al^u- \beta^u)\rangle_{L^2} \right]\\[0.5ex]
        &\geq \varepsilon (1-\varepsilon) \lam\, \E\left[\| \al^u- \beta^u\|_{L^2}^2 \right]\\[1ex]
        & >0,
    \end{aligned}
\end{equation}
where we used that $\wt{\boldsymbol A}$ is nonnegative definite by assumption for the first inequality, and the facts that $\eps,\lam>0$ and $\| \al^u- \beta^u\|_{L^2}^2 >0$, $\P$-a.s. for the second inequality.
\end{proof}
\begin{proof}[Proof of Proposition~\ref{prop:fredholm-foc-infinite}]
By Lemma~\ref{lemma:concave-infinite}, for each $u\in [0,1]$ and fixed $(\al^{v})_{v\neq u}\in\mathcal{A}^\infty$, the functional $\al^u\mapsto J^{u,W}(\al^{u}; (\al^{v})_{v\neq u})$ admits a unique maximizer characterized by the critical point at which its G\^ateaux derivative vanishes in all directions. Analogous to the proof of Proposition~\ref{prop:fredholm-foc} this allows the derivation of the following first order condition, for all $\F$-progressively measurable processes $\beta^u$ in $L^2(\Omega\times [0,T],\R)$,
\be \label{eq:foc-infinite}
0= \E\left[\langle\beta^u, -2\lam \al^u - (\wt{\boldsymbol A}+\wt{\boldsymbol A}^*)\al^u - \big(\wt{\boldsymbol B}+\wt{\boldsymbol B}^*\big) (\boldsymbol{W}\al)(u) +b^u\rangle_{L^2} \right].
\ee
By conditioning on $\mathcal F_t$ and using the tower property of conditional expectation we get from \eqref{eq:foc-infinite} the following first order condition given by a stochastic Fredholm equation in one variable, 
\be \label{eq:foc-infinite2}
2\lam \al^u_t = b_t^u - \big((\wt{\boldsymbol A}+\wt{\boldsymbol A}^*)\E_t[\al^u]\big)(t) - \big((\wt{\boldsymbol B}+\wt{\boldsymbol B}^*)(\boldsymbol{W} \E_t[\al])(u)\big)(t),\quad  d\P \otimes dt \textrm{-a.e.}
\ee
Summarizing the resulting equations in \eqref{eq:foc-infinite2} over all players $u\in [0,1]$ and recalling Definition~\ref{def:Nash-infinite-player} yields the the infinite-dimensional coupled system of stochastic Fredholm equations in \eqref{eq:fredholm-foc-infinite} and completes the proof.
\end{proof}
Next, we provide a rigorous proof of Lemma~\ref{lemma:alpha-hat-admissible}, employing a standard approach that leverages Grönwall’s Lemma.
\begin{proof}[Proof of Lemma~\ref{lemma:alpha-hat-admissible}]
Equation \eqref{eq:alpha-hat}, the fact that $(\varphi_i)_{i\in\N}$ is an orthonormal basis of $L^2([0,1],\R)$ and Tonelli's theorem imply
\be
\begin{aligned}\label{eq:admissible-equation}
\int_0^1\int_0^T\E[(\hat\al_t^u)^2]dtdu&= \int_0^T\E\Big[\int_0^1\big(\sum_{i=1}^\infty \varphi_i(u)\hat{\al}^i_t\big)^2du\Big]dt\\
&= \int_0^T\E\big[\sum_{i=1}^\infty (\hat{\al}^i_t)^2\big]dt\\
&=\sum_{i=1}^\infty \int_0^T\E\big[ (\hat{\al}^i_t)^2\big]dt.
\end{aligned}
\ee
In order to bound the infinite series in the last term of equation \eqref{eq:admissible-equation}, we will first bound its terms individually for each $i\in\N$ and then sum them up eventually.

Fix $i\in\N$. First, notice that the solution $\hat{\al}^i$ in \eqref{eq:alpha_t^i-hat} to the stochastic Fredholm equation in \eqref{eq:foc-infinite-reduced} arises from the Volterra equation
\be\label{eq:volterra-equation-hat-al^i}
\hat{\al}^i_t=\gamma_t^i-\int_0^t E^i(t,s)\hat{\al}^i_sds,
\ee
where $\gamma^i$ and $E^i$ are defined in \eqref{eq:gamma^i} and \eqref{eq:E^i}, respectively. The following lemma is needed to bound $\hat\al^i$.
\begin{lemma}\label{lemma:gamma_i-E^i-bound}
There exist constants $C_\gamma,C_E<\infty$ such that for all $i\in\N$ it holds that\begin{enumerate}
\item[\textbf{(i)}] $\int_0^T\E[(\gamma_t^i)^2]dt\leq C_\gamma\int_0^T\E[(\wt{b}_t^i)^2]dt$,
\item[\textbf{(ii)}] $\sup_{t\leq T}\int_0^T |E^i(t,s)|^2ds\leq C_E$.
\end{enumerate}
\end{lemma}
Now \eqref{eq:volterra-equation-hat-al^i}, the Cauchy-Schwarz inequality, part (ii) of Lemma~\ref{lemma:gamma_i-E^i-bound}, and Fubini's theorem imply that for all $t\in[0,T]$ it holds that
\be\begin{aligned}
\E[(\hat\al_t^i)^2]&\leq 2\E[(\gamma_t^i)^2]+2\E\Big[\Big(\int_0^t E^i(t,s)\hat{\al}^i_sds\Big)^2\Big]\\
&\leq 2\E[(\gamma_t^i)^2]+2\int_0^t |E^i(t,s)|^2ds \int_0^t\E\big[(\hat{\al}^i_s)^2\big]ds\\
&\leq 2\E[(\gamma_t^i)^2]+2\,C_E \int_0^t\E\big[(\hat{\al}^i_s)^2\big]ds
\end{aligned}\ee
It follows from Grönwall's inequality applied to the function $\E[(\hat\al^i)^2]:[0,T]\to\R$ that 
\be\begin{aligned}
\E[(\hat\al_t^i)^2]&\leq 2\E[(\gamma_t^i)^2] +\int_0^t 2\E[(\gamma_s^i)^2]2C_E e^{2C_E(t-s)}ds\\
&\leq 2\E[(\gamma_t^i)^2] +4C_Ee^{2C_ET}\int_0^T \E[(\gamma_s^i)^2]ds,
\end{aligned}\ee
so that integration with respect to $t$ and part (i) of Lemma~\ref{lemma:gamma_i-E^i-bound} yield
\be\label{eq:bound-alpha^i-hat}
\int_0^T\E[(\hat\al_t^i)^2]dt\leq
C_{\al}\int_0^T\E[(\wt{b}_t^i)^2]dt,
\ee
where $C_\al:=(2+4C_E T e^{2C_ET}) C_\gamma>0$ independent of $i\in\N$.
Now it follows from the definition of the coefficients $\smash{\wt{b}^i}$ in \eqref{eq:alpha-tilde-b-tilde}, Tonelli's theorem, Parseval's identity and the assumption that $b\in L^2( \Omega\times[0,T]\times [0,1],\R)$ that
\be\begin{aligned}\label{eq:bound-tilde-b^i}
\sum_{i=1}^\infty \int_0^T\E\big[(\wt{b}^i_t)^2\big]dt
&= \E\Big[\int_0^T\sum_{i=1}^\infty\langle\varphi_i,b_t \rangle_{L^2([0,1],\R)}^2dt\Big]\\
&=\E\Big[\int_0^T\int_0^1(b^u_t)^2dudt\Big]<\infty.
\end{aligned}\ee
Finally, \eqref{eq:admissible-equation}, \eqref{eq:bound-alpha^i-hat} and \eqref{eq:bound-tilde-b^i} together complete the proof.
\end{proof}
It remains to prove Lemma~\ref{lemma:gamma_i-E^i-bound} for which the following lemma is needed. 
\begin{lemma} \label{lemma:inverse operator}
Let $F$ be a bounded linear operator from a real Hilbert space $V$
into itself. Suppose that there exists a constant $c>0$ such that $\langle Fx,x\rangle \geq c_0 \langle x,x\rangle$ for all $x\in  V$. Then $F$ is invertible and $\|F^{-1}\|_{\operatorname{op}}\leq c^{-1}$.
\end{lemma} 
\begin{proof}
    See Problem 165 in Chapter~10 of \citet{torchinsky2015problems}.
\end{proof}

\begin{proof}[Proof of Lemma~\ref{lemma:gamma_i-E^i-bound}]
Fix $i\in\N$. 

(i) Recall that 
\be
\gamma_t^i =   \wt{b}^i_t  -  \langle  \mathds{1}_{\{t< \cdot\}}  \wt{C}^i(\cdot,t),  ({\boldsymbol D}_t^i)^{-1}\mathds{1}_{\{t< \cdot\}} \E_t[\wt{b}^i_{\cdot}] \rangle_{L^2}.
\ee
It suffices to show that there exists a $\wt{C}_\gamma<\infty$ independent of $i\in \N$ such that
\be \label{eq:admissibility-1}
\E\left[ \int_0^T \left|\int_t^T\wt{C}^i(r,t) \big(({\boldsymbol D}_t^i)^{-1}\mathds{1}_{\{t< \cdot\}} \E_t[ \wt{b}^i_{\cdot}]\big)(r)dr\right|^2dt\right]\leq \wt{C}_\gamma \E\Big[\int_0^T(\wt{b}_t^i)^2dt\Big],
\ee
since then setting $C_\gamma:=2+2\wt{C}_\gamma$ yields (i). Next, note that choosing $c_W\leq\lam$ in Assumption~\ref{assum:nonnegative-definite-graphon} together with Lemma~\ref{lemma:C^i-nonnegative} and Lemma~\ref{lemma:inverse operator} yields
\be\label{eq:bound-(D_t^i)^-1}
\sup_{i\in\N}\sup_{t\leq T}\|({\boldsymbol D}_t^i)^{-1}\|_{\operatorname{op}}\leq\frac{\lam}{c_W}<\infty.
\ee
Now from \eqref{eq:bound-(D_t^i)^-1}, the conditional Jensen inequality, Fubini's theorem and the tower property we get for all $0\leq t\leq T$ that
\be \label{gg6} 
\begin{aligned} 
\E\Big[ \int_0^T\big|\big(({\boldsymbol D}_t^i)^{-1}\mathds{1}_{\{t< \cdot\}} \E_t[ \wt{b}^i_{\cdot}]\big)(r)\big|^2 dr\Big] 
&\leq\lambda^2c_W^{-2}\E\Big[ \int_0^T\big|\mathds{1}_{\{t< \cdot\}} \E_t[ \wt{b}^i_r]\big|^2 dr\Big] \\ 
&\leq \lambda^2c_W^{-2}\int_0^T\E\big[ \E_t[( \wt{b}^i_r)^2]\big] dr\\
& = \lambda^2c_W^{-2} \int_0^T\E\big[(\wt{b}_r^i)^2\big]dr. 
\end{aligned} 
\ee
Next, notice that the eigenvalues $(\vartheta_i)_{i\in\N}$ of $\boldsymbol{W}$ are uniformly bounded by
\be\label{eq:bound-theta_i}
|\vartheta_i|\leq \Big(\int_0^1\int_0^1 W(u,v)^2dudv\Big)^{1/2},\quad i\in\N,
\ee
(see equation (7.20) in \citet{lovasz2012large}). Thus it follows from the integrability condition \eqref{eq:sup-A+B-condition} on $\wt{A}$ and $\wt{B}$ and \eqref{eq:bound-theta_i} that there is a $C_C<\infty$ independent of $i\in\N$ such that
\be\label{eq:bound_C^i}
\sup_{t\leq T}\int_0^T |\wt{C}^i(t,s)|^2ds+\sup_{s\leq T}\int_0^T |\wt{C}^i(t,s)|^2dt\leq C_C,\quad i\in\N.
\ee
The monotonicity of the integral,  the Cauchy-Schwarz inequality, Fubini's theorem, \eqref{gg6}  and \eqref{eq:bound_C^i} yield
\be
 \begin{aligned} 
&\E\left[ \int_0^T \left|\int_t^T\wt{C}^i(r,t) \big(({\boldsymbol D}_t^i)^{-1}\mathds{1}_{\{t< \cdot\}} \E_t[ \wt{b}^i_{\cdot}]\big)(r)dr\right|^2dt\right] \\
&\leq \E\left[ \int_0^T\Big(\int_0^T \big|\wt{C}^i(r,t)\big| \big|\big(({\boldsymbol D}_t^i)^{-1}\mathds{1}_{\{t< \cdot\}} \E_t[ \wt{b}^i_{\cdot}]\big)(r)\big|dr\Big)^2
dt\right] \\
&\leq  \int_0^T\int_0^T \big|\wt{C}^i(r,t)\big|^2dr \hspace{1mm}\E\left[\int_0^T \big|\big(({\boldsymbol D}_t^i)^{-1}\mathds{1}_{\{t< \cdot\}} \E_t[ \wt{b}^i_{\cdot}]\big)(r)\big|^2dr\right]
dt \\ 
&\leq TC_C\lambda^2c_W^{-2} \int_0^T\E\big[(\wt{b}_t^i)^2\big]dt,\quad i\in \N. 
\end{aligned} 
\ee
This verifies \eqref{eq:admissibility-1} with $\wt{C}_\gamma:=TC_C\lambda^2c_W^{-2}<\infty$ and completes the proof of (i). 

(ii) Recall that 
\be\label{eq:E^i-recalled}
E^i(t,s) = - \mathds{1}_{\{t>s\}}  \big( \langle  \mathds{1}_{\{t< \cdot\}}   \wt{C}^i(\cdot,t), ({\boldsymbol D}_t^i)^{-1}   \mathds{1}_{\{t< \cdot\}}  \wt{C}^i(\cdot,s)  \rangle_{L^2}    -   \wt{C}^i(t,s)   \big).
\ee
Due to \eqref{eq:bound_C^i}, it suffices to show that there is a $\wt{C}_E<\infty$ such that
\be
\sup_{t\leq T}\int_0^t\Big|\int_t^T \wt{C}^i(r,t)\big(({\boldsymbol D}_t^i)^{-1}\mathds{1}_{\{t< \cdot\}} \wt{C}^i(\cdot,s)\big)(r) dr\Big|^2ds\leq\wt{C}_E.    
\ee
The monotonicity of the integral, the Cauchy-Schwarz inequality, \eqref{eq:bound-(D_t^i)^-1} and \eqref{eq:bound_C^i} imply for all $i\in\N$ that
\be\begin{aligned}
&\sup_{t\leq T}\int_0^t\Big|\int_t^T \wt{C}^i(r,t)\big(({\boldsymbol D}_t^i)^{-1}\mathds{1}_{\{t< \cdot\}} \wt{C}^i(\cdot,s)\big)(r) dr\Big|^2ds\\&
\leq\sup_{t\leq T}\int_0^T\Big(\int_0^T \big|\wt{C}^i(r,t)\big|\big|\big(({\boldsymbol D}_t^i)^{-1}\mathds{1}_{\{t< \cdot\}} \wt{C}^i(\cdot,s)\big)(r)\big| dr\Big)^2ds\\&
\leq\sup_{t\leq T}\int_0^T\int_0^T \big|\wt{C}^i(r,t)\big|^2dr\int_0^T\big|\big(({\boldsymbol D}_t^i)^{-1}\mathds{1}_{\{t< \cdot\}} \wt{C}^i(\cdot,s)\big)(r)\big|^2 drds\\&
\leq\sup_{t\leq T}\int_0^T\int_0^T \big|\wt{C}^i(r,t)\big|^2dr  \lam^2c_W^{-2}\int_0^T\big|\mathds{1}_{\{t< r\}} \wt{C}^i(r,s)\big|^2 drds\\&
\leq\sup_{t\leq T}\int_0^T \big|\wt{C}^i(r,t)\big|^2dr  \lam^2c_W^{-2}TC_C\\&
\leq \lam^2c_W^{-2}TC_C^2=:\wt{C}_E<\infty.
\end{aligned}\ee
Recalling \eqref{eq:E^i-recalled} and setting $C_E:=2\wt{C}_E+2C_C$ completes the proof of (ii).
\end{proof}

\end{document}